\documentclass[11pt]{article}
\usepackage{amsmath, amssymb, amscd, amsthm, amsfonts,bbm}
\usepackage{mathrsfs}
\usepackage{mathtools}
\usepackage{booktabs,longtable,array}
\usepackage{microtype}
\usepackage{graphicx}
\usepackage{float} 
\usepackage{xcolor}
\usepackage{subfigure}
\usepackage[hyperfootnotes=false,colorlinks=true,linkcolor=blue!55!black,citecolor=blue!55!black,urlcolor=blue!55!black]{hyperref}
\usepackage{enumitem}
\numberwithin{equation}{section}
\usepackage{listings}
\usepackage{setspace}
\title{The Brownian loop-catcher}
\author{Gefei Cai}
\date{\today}

\newcommand{\R}{\mathbb{R}}

\newcommand{\Q}{\mathbb{Q}}
\newcommand{\D}{\mathbb{D}}
\newcommand{\C}{\mathbb{C}}
\newcommand{\Z}{\mathbb{Z}}

\newcommand{\cL}{\mathcal{L}}

\newcommand{\E}{\mathbb{E}}

\renewcommand{\P}{\mathbb{P}}

\newcommand{\BE}{\mathrm{BE}}
\newcommand{\LE}{\mathrm{LE}}

\DeclareMathOperator{\SLE}{SLE}

\newcommand{\dist}{\mathrm{dist}}

\newtheorem{theorem}{Theorem}[section]
\newtheorem{definition}[theorem]{Definition}
\newtheorem{lemma}[theorem]{Lemma}

\newtheorem{corollary}[theorem]{Corollary}

\newtheorem{algorithm}[theorem]{Algorithm}
\newtheorem{proposition}[theorem]{Proposition}
\newtheorem{remark}[theorem]{Remark}

\newcommand\wh[1]{\widehat{#1}}
\newcommand\ol[1]{\overline{#1}}

\makeatletter
\def\@rst #1 #2other{#1}
\newcommand\MR[1]{\relax\ifhmode\unskip\spacefactor3000 \space\fi
  \MRhref{\expandafter\@rst #1 other}{#1}}
\newcommand{\MRhref}[2]{\href{http://www.ams.org/mathscinet-getitem?mr=#1}{MR#2}}

\def\MR#1{\href{http://www.ams.org/mathscinet-getitem?mr=#1}{MR#1}}
\makeatother

\providecommand{\Cref}[1]{\ref{#1}}

\newcommand{\diam}{\operatorname{diam}}
\newcommand{\supp}{\operatorname{supp}}
\newcommand{\one}{\mathbf 1}
\newcommand{\Prob}{\mathbb P}
\newcommand{\dd}{\,\mathrm d}
\newcommand{\Loop}{\mu^{\mathrm{loop}}}

\newcommand{\Lc}{\mathcal L}

\newcommand{\Ac}{\mathcal A}
\newcommand{\Gc}{\mathcal G}
\newcommand{\rad}{\mathrm{rad}}
\newcommand{\balpha}{\boldsymbol{\alpha}}
\newcommand{\bbeta}{\boldsymbol{\beta}}
\newcommand{\bp}{\mathbf p}
\newcommand{\bq}{\mathbf q}
\newcommand{\br}{\mathbf r}
\newcommand{\bh}{\mathbf h}
\newcommand{\bH}{\mathbf H}
\newcommand{\brho}{\boldsymbol{\rho}}
\newcommand{\bnu}{\boldsymbol{\nu}}

\newcommand{\Cell}{\operatorname{Cell}}

\begin{document}

\maketitle

\begin{center}
\emph{Dedicated to Greg Lawler on the occasion of his 70th birthday.}
\end{center}
\vspace{0.5em}

\begin{abstract}
We introduce a family of random connected closed subsets of planar Brownian motion, called Brownian loop-catchers, which interpolate between the continuum loop-erased random-walk (LERW) and the Brownian trace. This provides the canonical continuation of Brownian loop soup clusters to central charges $-2\le c<0$. For each such $c$, the corresponding loop-catcher satisfies the following recovery property: adding all loops from an independent Brownian loop soup of intensity $-c/2$ that intersect it recovers the Brownian trace. Furthermore, no such law exists for $c<-2$. We also show that its outer boundary is locally SLE$_\kappa$ with $\kappa = \frac{1}{3}\left(13 - c - \sqrt{(1-c)(25-c)}\right)\in[2,\frac83)$, and the probability that it intersects an interior ball of radius $\varepsilon$ is asymptotically proportional to $|\log\varepsilon|^{-1+\frac{c}{2}}$ when $-2<c<0$. Therefore, a planar Brownian trace contains an SLE$_\kappa$-type curve for every $\kappa\in[2,\frac83]$. 

Our construction begins with a random-walk loop-catcher on any finite graph such that recursively inserting loops from an independent random-walk loop soup to it recovers the original random walk. The key ingredient is a new entangled multipath LERW, which recovers a union of independent random-walk paths when decorated with a single common random-walk loop soup.
We then prove that the random-walk loop-catcher converges to the Brownian loop-catcher under lattice approximations. 
To this end, we propose a novel Green function test which converts the recovery property of the Brownian loop-catcher into all mixed moments of Green functions in the remaining domain, based on the entangled multipath LERW. Consequently, the recovery property characterizes the full law of the Brownian loop-catcher, not only its filling. The Green function test also extends to the three-dimensional case.
\end{abstract}

\section{Introduction}

Planar Brownian motion is the most basic conformally invariant stochastic process in two dimensions, and it is the canonical random object of central charge zero in the language of conformal field theory (CFT). As shown in~\cite{lsw-restriction}, its outer boundary is described by Schramm--Loewner evolution (SLE) with parameter $\kappa=\frac{8}{3}$. To construct random objects with positive central charge $c\in(0,1]$, one can use the Poisson point process of Brownian loops introduced in~\cite{LW04}, namely the Brownian loop soup of intensity $\frac{c}{2}$, and consider the clusters formed by these loops.  The outer boundary of each loop-soup cluster is locally $\SLE_\kappa$~\cite{shef-werner-cle}, with
\begin{equation}\label{eq:c-kappa}
        \kappa = \frac{1}{3}\left(13 - c - \sqrt{(1-c)(25-c)}\right).
\end{equation}
Note that $\kappa\in(\frac{8}{3},4]$ when $c\in(0,1]$, and $c=0$ corresponds to $\kappa=\frac{8}{3}$.

While the Brownian loop soup provides a powerful probabilistic construction for positive central charge, it is natural to ask whether an analogous Brownian construction exists for \emph{negative} central charge and hence whether one can obtain $\SLE_\kappa$ with $\kappa<\frac{8}{3}$ from Brownian motion. Little was previously known in this direction: there is no Brownian loop soup of negative intensity, and the only known example is the loop-erased random walk (LERW), corresponding to $c=-2$ with scaling limit $\SLE_2$~\cite{schramm0,LawlerSchrammWerner2004}.

In this paper we give an affirmative and sharp answer to this question. For each
$c\in[-2,0)$, in
every bounded Jordan domain \(D\subset\C\) with two
distinct marked boundary points \(a,b\), we show that there is a canonical law of random connected closed set $K$ joining \(a\) to \(b\) such that, by adding
to $K$ all loops from an independent Brownian loop soup of intensity
$\lambda:=-\frac{c}{2}$ that intersect $K$, one recovers the law of the Brownian excursion in $D$ from $a$ to $b$. The $c=0$ case reduces to the Brownian excursion itself.  At $c=-2$ (i.e.~$\lambda=1$), such $K$ is the scaling limit of LERW, namely 
$\SLE_2$ in two dimensions. For $c<-2$, no such law of $K$ exists.  We call $K$ a
\emph{$\lambda$-Brownian loop-catcher}.  Furthermore, we show that the outer
boundary of $K$ is described by $\SLE_\kappa$, where
$\kappa\in[2,\frac{8}{3})$ satisfies~\eqref{eq:c-kappa}.  In particular, a
Brownian trace contains an $\SLE_\kappa$-type curve for every
$\kappa\in[2,\frac{8}{3}]$.

We first work on arbitrary finite graphs. There we introduce the \emph{entangled multipath LERW}, and then recursively use it to construct a directed path such that adding loops from an independent random-walk loop soup with intensity $\lambda\in[0,1]$ recovers the original random-walk path (when $\lambda=1$, it reduces to the LERW). We call its range the $\lambda$-\emph{random-walk loop-catcher}; see Section~\ref{subsec:organization}.

Under lattice
approximations of a Jordan domain, the family of random-walk loop-catchers is tight and has subsequential limits. To identify all subsequential limits, we introduce a \emph{Green function test} using random interior multipath probes, which are constructed by subsequential limits of multipath versions of $\lambda$-random walk loop-catchers.
This Green function test then characterizes the limiting closed set, not only its hull, and yields the desired full convergence (Theorems~\ref{thm:catcher-convergence} and~\ref{thm:full-trace-unique}). Similar boundary Poisson kernel tests
identify its outer boundary with SLE (Theorems~\ref{thm:hull-unique} and~\ref{thm:sle}).  At $\lambda=1$, such a characterization solves the
two-dimensional chordal version of~\cite[Conjecture~1.3]{SS18}; see
Theorem~\ref{thm:le-uniqueness}. We emphasize that such a Green function test is quite robust and does not require any simple connectedness, and can also work for even three dimensions; see Section~\ref{rmk:3d}.

We also prove the one-point density of Brownian loop-catcher in Theorem~\ref{thm:one-arm}. As shown in this paper, the geometric properties of Brownian loop-catcher are parallel to those of Brownian loop soup clusters, indicating that it is the canonical continuation of Brownian loop soup clusters.

Below, we first state precise results on Brownian loop-catchers and random-walk loop-catchers on finite graphs; see Theorems~\ref{thm:main} and~\ref{thm:finite-positivity}, respectively.

\subsection{The Brownian loop-catcher}\label{subsec:continuum-main}

Let \(D\subset\C\) be a bounded
Jordan domain with distinct marked points \(a,b\in\partial D\).  Let \(\gamma_{\BE}^{D;a,b}\) denote the trace of Brownian excursion
in \(D\) from \(a\) to \(b\), and let \(\mu_D^{\mathrm{loop}}\) be the Brownian
loop measure in \(D\).  For compact sets \(K,A\subset\overline D\), let
\begin{equation}
 m_D(K,A)
 :=\mu_D^{\mathrm{loop}}\{\ell:\ell\cap K\ne\varnothing,
        \ell\cap A\ne\varnothing\}.
\end{equation}
Throughout the paper we set $\lambda:=-\frac{c}{2}$.

\begin{theorem}[The Brownian loop-catcher]\label{thm:main}
For every \(0<\lambda\le 1\) (i.e.~\(-2\le c<0\)), there is a
unique probability law on compact connected sets
\(K\subset\overline D\) with
\(K\cap\partial D=\{a,b\}\) such that
\begin{equation}\label{eq:intro-main-id}
        \E\left[\one_{\{K\cap A=\varnothing\}}
        \exp\left\{-\lambda m_D(K,A)\right\}\right]
        =\Prob\left[\gamma_{\BE}^{D;a,b}\cap A=\varnothing\right],
\end{equation}
for every compact \(A\subset\overline D\) such that $A\cap\{a,b\}=\varnothing$ (with no requirement that \(A\) attach to \(\partial D\)).

For \(\lambda>1\) (i.e.~\(c<-2\)), no probability law on compact
connected sets satisfies~\eqref{eq:intro-main-id}.
\end{theorem}

We call \(K_\lambda^{D;a,b}:=K\) the \(\lambda\)-Brownian loop-catcher in
\((D;a,b)\).
Let
\(\cL_D^\lambda\) be an independent Brownian loop soup in \(D\) with
loop-measure intensity $\lambda=-\frac{c}{2}$, and define the closed union
\begin{equation*}
        \mathcal D_\lambda(K):=\ol{K\cup
        \bigcup_{\ell\in\cL_D^\lambda:\,\ell\cap K\ne\varnothing}\ell}.
\end{equation*}
Then we have $\Prob[\mathcal D_\lambda(K)\cap A=\varnothing\mid K]=\one_{\{K\cap A=\varnothing\}}\exp\left\{-\lambda m_D(K,A)\right\}$ for each compact $A\subset\overline D$ away from $a$ and $b$,
and hence \(\Prob[\mathcal D_\lambda(K)\cap A=\varnothing]=\Prob[\gamma_{\BE}^{D;a,b}\cap A=\varnothing]\) by~\eqref{eq:intro-main-id}. Since these avoidance probabilities determine the law of a random compact set (see e.g.~\cite[Section~1.1]{Molchanov2017}), we have
\begin{equation}\label{eq:decoration}
\mathcal D_\lambda(K)\overset{\rm law}=\gamma_{\BE}^{D;a,b}.
\end{equation}
At \(\lambda=1\), the scaling limit of LERW on $(D;a,b)$ (i.e.~chordal \(\SLE_2\)) satisfies~\eqref{eq:intro-main-id}~\cite{SS18}; see also~\cite{Ambrosio2020,BS26}.
Thus \(K\) interpolates between the trace of Brownian excursion and the scaling limit of LERW in \((D;a,b)\).

Since~\eqref{eq:intro-main-id} characterizes the law of $K_\lambda^{D;a,b}$ (see Theorem~\ref{thm:full-trace-unique} below), by conformal covariance of Brownian excursion and Brownian loop measure, we have
\[
 \phi(K_\lambda^{D;a,b})\ \overset{\rm law}=
 K_\lambda^{D';\phi(a),\phi(b)}
\]
for every conformal map \(\phi:D\to D'\) between bounded Jordan domains;
the maps extend homeomorphically to the boundary.

The following theorem gives the one-point density of the Brownian loop-catcher $K_\lambda^{D;a,b}$.
\begin{theorem}\label{thm:one-arm}
Let $(D;a,b)$ be as above. For every $z\in D$ and $0<\lambda<1$, there is a constant
\(C_{D,a,b,z,\lambda}\in(0,\infty)\) such that as $\varepsilon\downarrow0$,
\begin{equation}\label{eq:one-arm-main}
 \Prob\!\left[K_\lambda^{D;a,b}\cap B(z,\varepsilon)\ne\varnothing\right]
= C_{D,a,b,z,\lambda}|\log\varepsilon|^{-1-\lambda}(1+o(1)).
\end{equation}
Furthermore, we have \(C_{D,a,b,z,\lambda}\to 0\) as \(\lambda \uparrow 1\).
\end{theorem}
For $\lambda=1$, since $K_\lambda^{D;a,b}$ is chordal $\SLE_2$, the right side of~\eqref{eq:one-arm-main} is instead  $C\varepsilon^{3/4}(1+o(1))$~\cite{beffara-dim,lawler-rezai-nat}. This indicates that the Brownian loop-catcher undergoes a phase transition when $\lambda\uparrow 1$.

Theorem~\ref{thm:one-arm} suggests that a $\lambda$-Brownian loop-catcher has Hausdorff dimension two and Minkowski gauge $x\mapsto x^2|\log x|^{1+\lambda}$. For $c\in(0,1]$, Brownian loop soup clusters have Minkowski gauge $x\mapsto x^2|\log x|^{1-\frac{c}{2}+o(1)}$~\cite{jego2023conformally}. Since $\lambda=-\frac{c}{2}$, Theorem~\ref{thm:one-arm} then indicates that Brownian loop-catchers could be viewed as the continuation of Brownian loop soup clusters to $c<0$.

\subsection{The random-walk loop-catcher on finite graphs}\label{subsec:finite-main}

The proof of Theorem~\ref{thm:main} relies on the construction of
\(\lambda\)-\emph{random-walk loop-catchers} on finite graphs, which satisfies the discrete analog of~\eqref{eq:decoration}.

Let $V$ be a finite set of vertices, let $Q=(Q_{xy})_{x,y\in V}$ be a nonnegative substochastic transition matrix with spectral radius $\rho(Q)<1$, and let $\balpha=(\alpha_x)_{x\in V},\bbeta=(\beta_x)_{x\in V}$ be entrance and exit weights with $\balpha,\bbeta\ge0$ and $\balpha^{\mathsf T}(I-Q)^{-1}\bbeta>0$.  The associated path weights have partition function
\[
 Z_V=\sum_{n\ge0}\sum_{x_0,x_1,\ldots,x_n\in V}
 \alpha_{x_0}\left(\prod_{i=0}^{n-1}Q_{x_i x_{i+1}}\right)\beta_{x_n}
 =\balpha^{\mathsf T}(I-Q)^{-1}\bbeta.
\]
The normalized \emph{random-walk excursion} assigns probability $Z_V^{-1}\alpha_{x_0}\Bigl(\prod_{j<n}Q_{x_jx_{j+1}}\Bigr)\beta_{x_n}$ to a path \(\omega=(x_0,\ldots,x_n)\).
For \(R\subset V\), let \(\balpha_R\) and \(\bbeta_R\) be the coordinate restrictions of \(\balpha\) and \(\bbeta\) to \(R\), respectively, and let \(Q_R\) be the principal submatrix of \(Q\) indexed by \(R\). Write
\begin{equation}\label{eq:z-d-bh}
   Z_R=\balpha_R^{\mathsf T}(I-Q_R)^{-1}\bbeta_R,
   \qquad
   d_R=\det(I-Q_R),
   \qquad
   \bh(R)=\frac{Z_R}{Z_V}
\end{equation}
with \(d_\varnothing=1\) and \(Z_\varnothing=0\).
Here $\bh(R)$ is the probability that the excursion trace $\omega$ stays in $R$.

Let \(m_V(S,A)\) be the mass of unrooted random-walk loops hitting both $S$ and $A$.
Thus, for
a vector \(\bp=(\bp(S))_{S\subset V}\), the finite-graph counterpart of~\eqref{eq:intro-main-id} is
\begin{equation}\label{eq:finite-transform-intro}
   (T_\lambda^V\bp)(R)
   :=\sum_{S\subset V}\bp(S)\one_{\{S\subset R\}}
     \exp\{-\lambda m_V(S,V\setminus R)\}
   =\sum_{S\subset R}\bp(S)
     \left(\frac{d_Vd_{R\setminus S}}{d_{V\setminus S}d_R}\right)^\lambda
   =\bh(R),
   \qquad \forall R\subset V.
\end{equation}
Here we used the relation between loop masses and determinants; see e.g.~\cite[Proposition~5.2]{Lawler2018}.

Note that the matrix $T_\lambda^V$ is triangular on the Boolean lattice $2^V$ and has
positive diagonal, so~\eqref{eq:finite-transform-intro} always has a
unique real solution $\bp_\lambda$.  There is, however, no a priori
reason for its coefficients to be nonnegative.  Proving this nonnegativity
is our main contribution on finite graphs.

\begin{theorem}[The random-walk loop-catcher]\label{thm:finite-positivity}
For every finite datum $(V,Q,\balpha,\bbeta)$ as above and every $\lambda\in[0,1]$, the unique real solution $\bp_\lambda=(T_\lambda^V)^{-1}\bh$
of~\eqref{eq:finite-transform-intro} is a probability measure on $2^V$. Namely,
\begin{equation}
        \bp_\lambda(S)\ge0\quad(\forall S\subset V),
        \qquad
        \sum_{S\subset V}\bp_\lambda(S)=1.
\end{equation}
When $\lambda>1$, there exists $(V,Q,\balpha,\bbeta)$ such that the corresponding $\bp_\lambda$ has a negative coordinate.
\end{theorem}

We call a \emph{random set} with law $\bp_\lambda$ a \(\lambda\)-random-walk loop-catcher on $(V,Q,\balpha,\bbeta)$. By the same argument used
after Theorem~\ref{thm:main}, adding all loops that intersect this set
from an independent random-walk loop soup of intensity $\lambda$ recovers the law of the trace of the random-walk excursion associated with
$(Q,\balpha,\bbeta)$.
At $\lambda=1$, $\bp_1$ is the law of the corresponding LERW trace.  

In Section~\ref{sec:finite-positivity}, based on the entangled multipath LERW, we will construct a \emph{directed path} whose range satisfies~\eqref{eq:finite-transform-intro} and such that recursively inserting loops from an independent random-walk loop soup with intensity $\lambda$ to it recovers the law of the random walk (as a directed path)\footnote{One might wonder whether a ``partially loop-erased'' random walk could provide such a process, by erasing each newly formed loop of a random walk independently with some fixed probability \(q\in(0,1)\). Nevertheless, this is not the case, and for every \(q>0\), its scaling limit is expected to be the same as that of LERW (see~\cite[the end of Section 7]{wiese-fedorenko-lerw-field}).}. In Appendix~\ref{app:finite-positivity}, we provide an alternative linear-algebraic proof of Theorem~\ref{thm:finite-positivity} by showing that the logarithmic derivative $\mathcal{A}_\lambda^V$ of the operator $T_\lambda^V$ has zero column sums and becomes a pure-birth generator, and therefore~\eqref{eq:finite-transform-intro} has a unique nonnegative solution.

\begin{figure}[H]
    \centering

    \begin{minipage}[t]{0.32\textwidth}
        \centering
        \includegraphics[
            width=\linewidth,
            trim={180 53 180 53},
            clip
        ]{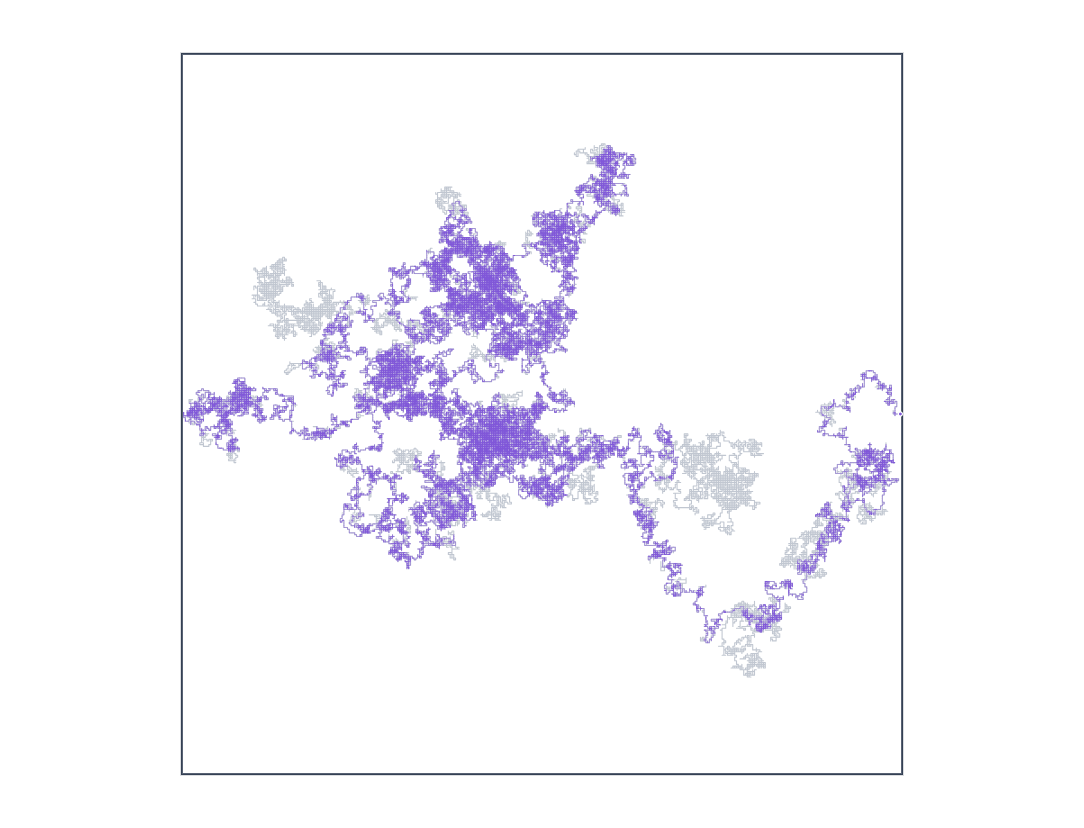}
        \vspace{-0.6em}

        {\small $\lambda=0.2$}
    \end{minipage}
    \hfill
    \begin{minipage}[t]{0.32\textwidth}
        \centering
        \includegraphics[
            width=\linewidth,
            trim={180 53 180 53},
            clip
        ]{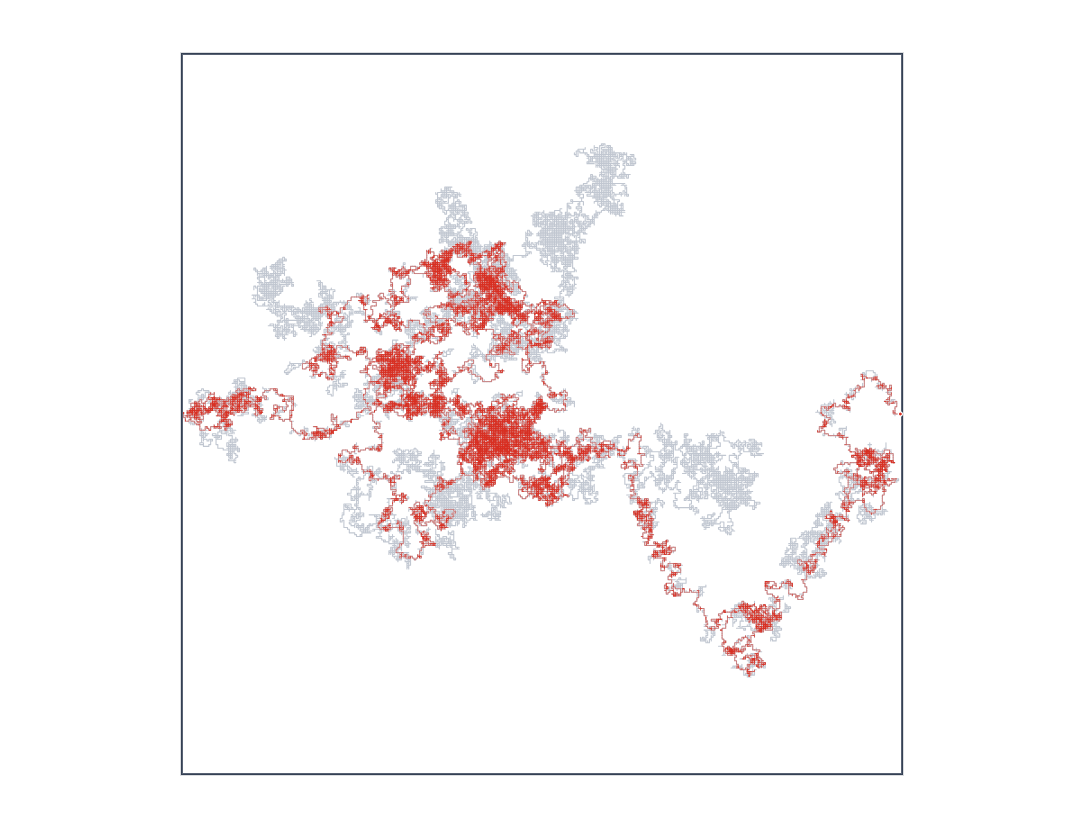}
        \vspace{-0.6em}

        {\small $\lambda=0.5$}
    \end{minipage}
    \hfill
    \begin{minipage}[t]{0.32\textwidth}
        \centering
        \includegraphics[
            width=\linewidth,
            trim={180 53 180 53},
            clip
        ]{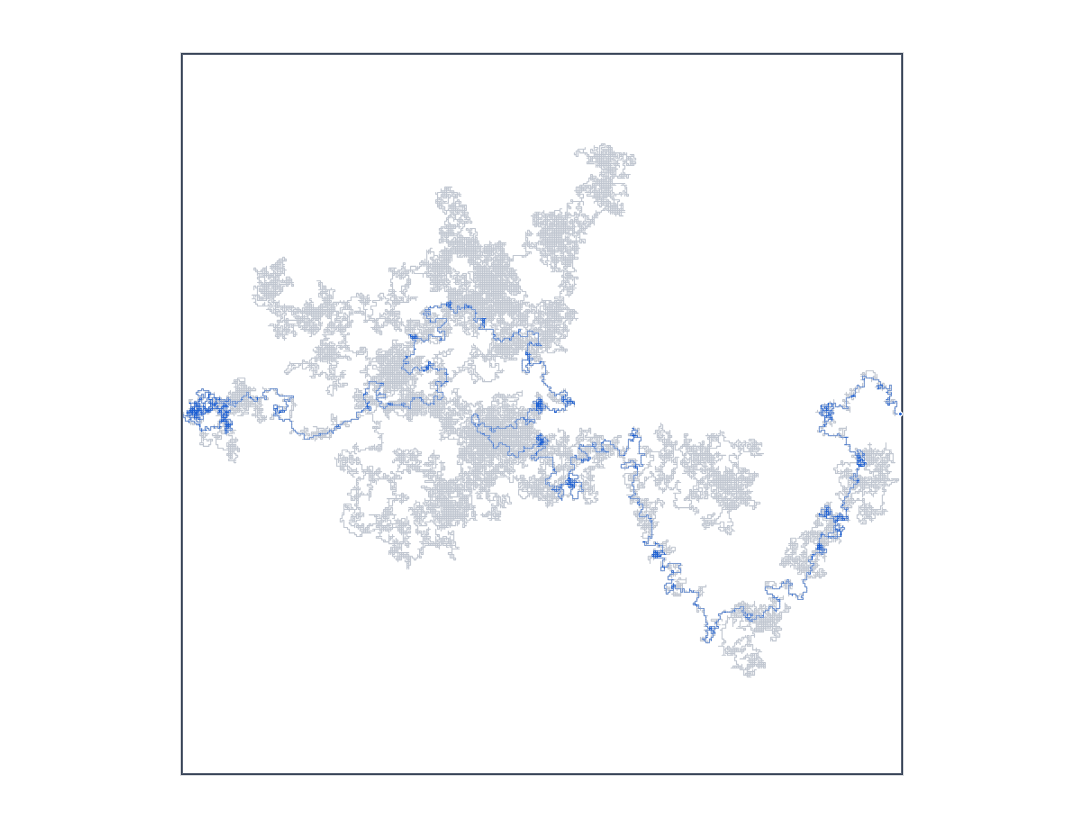}
        \vspace{-0.6em}

        {\small $\lambda=0.8$}
    \end{minipage}

    \caption{Simulations of random-walk loop-catchers on $\frac{1}{400}\Z^2\cap[-1,1]^2$ for different values of
    \(\lambda\), constructed from the same underlying random walk (shown in gray) using Algorithm~\ref{alg:time-catcher}.}
    \label{fig:loop-catcher-simulations}
\end{figure}

We then use square-lattice approximations to $(D;a,b)$ and take the limit to define the desired Brownian loop-catcher. See the beginning of Section~\ref{sec:continuum} for detailed convergence setup.

\begin{theorem}[Convergence]\label{thm:catcher-convergence}
For \(0<\lambda\leq1\), fix \((D;a,b)\) as in Section~\ref{subsec:continuum-main}.  Let \(\delta\downarrow0\), and let
\(D^\delta\) be a finite connected induced subgraph of
\(\delta\Z^2\cap D\) such that the union of its mesh squares is a Jordan domain converging
to \(D\) both in the Carath\'eodory topology and in boundary Fr\'echet
  distance, equipped with exterior boundary vertices
  \(\widehat a^\delta,\widehat b^\delta\) converging to 
  \(a,b\).  Set \(Q^\delta_{xy}=\frac14\one_{\{|x-y|=\delta\}}\) for
\(x,y\in D^\delta\), and assign weight \(\frac14\) to every edge entering
\(D^\delta\) from \(\widehat a^\delta\) or exiting \(D^\delta\) to
\(\widehat b^\delta\); these weights define the entrance and exit
vectors.  
  With \((D;a,b)\) fixed, let \(K_\lambda^\delta\) be its 
\(\lambda\)-random-walk loop-catcher on $D^\delta$, realized as the union of its closed mesh cells
  and two \(o_\delta(1)\)-diameter connectors from \(a\) and \(b\), respectively, to the corresponding endpoint cells.
Then in the Hausdorff topology on closed subsets of $\ol D$,
\[
 K_\lambda^{\delta}\ \Longrightarrow\ K_\lambda^{D;a,b}
 \qquad(\delta\downarrow0).
\]
\end{theorem}

One can also define radial, bridge and loop versions of random-walk and Brownian loop-catchers similarly. For simplicity, we focus on the chordal setup in this paper.

\subsection{Characterization of the full trace and the outer boundary}\label{subsec:identify-sle}

Theorem~\ref{thm:catcher-convergence} relies on the following
full-trace characterization.
\begin{theorem}[Full-trace uniqueness]\label{thm:full-trace-unique}
Let \(D\) be a bounded Jordan domain and \(a\ne b\in\partial D\).
For \(0<\lambda\leq1\), a probability law $\P$ on compact connected
\(K\subset\overline D\) with \(K\cap\partial D=\{a,b\}\) is determined by
\begin{equation}
 \E\!\left[
 \one_{\{K\cap A=\varnothing\}}e^{-\lambda m_D(K,A)}\right],
 \qquad \forall A\subset \ol D\ \hbox{compact}.
\end{equation}
Here $\E$ denotes the expectation with respect to $\P$.
\end{theorem}

To prove Theorem~\ref{thm:catcher-convergence}, we first
extract subsequential limits of $K_\lambda^\delta$ and show that they satisfy~\eqref{eq:intro-main-id}. Theorem~\ref{thm:full-trace-unique} then identifies their laws, proving convergence of the entire sequence.

The proof of Theorem~\ref{thm:full-trace-unique} is based on deriving all the moments of Green functions in the remaining domain $D\setminus K$, using subsequential limits of $0<\lambda\le1$ evolutions of entangled multipath LERWs as test sets. This Green function test does not require $D\setminus K$ to be simply connected (it may have countably many connected components), and can be adapted to three-dimensional cases. Therefore, our proof of Theorem~\ref{thm:catcher-convergence} could be viewed as extensions of the conformal restriction characterization idea developed in~\cite{lsw-restriction,wu15,qian-trichordal,qian2021generalized,BJ2024,gordina2024infinitesimal,CG25}, which usually characterizes the filling of $K$.

To identify the outer boundary of $K_\lambda^{D;a,b}$ with $\SLE_\kappa$, we need the following characterization of its outer boundary, or its \emph{filling hull}.
Let \(\partial^+D\) and \(\partial^-D\) denote the two open components of
\(\partial D\setminus\{a,b\}\), with a fixed labelling.  For a compact
connected \(K\subset\ol D\) with \(K\cap\partial D=\{a,b\}\), let
\(D_K^+\) and \(D_K^-\) be the components of \(D\setminus K\) adjacent to
\(\partial^+D\) and \(\partial^-D\), respectively, and
set
\[
 \operatorname{Fill}(K)
 :=\ol{D\setminus(D_K^+\cup D_K^-)}.
\]
We call $\operatorname{Fill}(K)$ the \emph{two-sided filling hull} of $K$. Similarly,  we define the
\emph{one-sided filling hull} of $K$ to be $\operatorname{Fill}^{+}(K):=\ol{D\setminus D_K^-}$. We also call a closed connected subset $H\subset \ol D$ a two-sided (resp.~one-sided) hull if $H\cap\partial D=\{a,b\}$ and $D\setminus H$ has two connected components (resp.~if $H\cap\partial D=\partial^+D\cup\{a,b\}$ and $D\setminus H$ is simply connected with boundary containing $\partial^-D$).

\begin{theorem}[Hull uniqueness]\label{thm:hull-unique}
The following two statements hold.
\begin{enumerate}[label=\textup{(\roman*)},leftmargin=2.3em]
\item\label{item:two-sided-unique}
For \(0<\lambda\leq1\), the law $\mathsf P$ of a random two-sided hull $H$ is
determined by the numbers $\E_{\mathsf P}\left[\one_{\{H\cap A=\varnothing\}}
 e^{-\lambda m_D(H,A)}\right]$ for compact
\(A\subset\overline D\) with
\(A\cap\{a,b\}=\varnothing\) and \(D\setminus A\) simply connected.
\item\label{item:one-sided-unique}
For \(0<\lambda\leq1\), the law $\mathsf P$ of a random one-sided hull $H$ is
determined by the numbers $\E_{\mathsf P}\left[\one_{\{H\cap A=\varnothing\}}
 e^{-\lambda m_D(H,A)}\right]$ for compact
\(A\subset\overline D\) with
\(A\cap\partial^+D=\varnothing\) and \(D\setminus A\) simply connected.
\end{enumerate}
In particular, there is at most one law of two-sided (resp.~one-sided) hulls satisfying
\begin{equation}\label{eq:hull-id}
 \E\!\left[\one_{\{H\cap A=\varnothing\}}
 e^{-\lambda m_D(H,A)}\right]=\Prob\left[\gamma_{\BE}^{D;a,b}\cap A=\varnothing\right]
\end{equation}
for every allowed \(A\) in the corresponding case.
\end{theorem}

Its proof parallels that of Theorem~\ref{thm:full-trace-unique}, using moments of boundary Poisson kernels instead of Green functions.
As a corollary, Theorem~\ref{thm:hull-unique} implies the uniqueness of generalized chordal restriction measures introduced in~\cite{qian2021generalized} with central charges $-2\le c<0$. The same proof strategy also applies to generalized radial and trichordal restriction measures~\cite{qian-trichordal,qian2021generalized}, extending the uniqueness result of~\cite{CG25}. We also expect that the uniqueness of generalized chordal restriction measures can be proved using the boundary variations and Virasoro structures from~\cite{BJ2024} (see also~\cite{gordina2024infinitesimal}).

Now we can identify the two-sided filling hull of the Brownian loop-catcher with SLE.  Let \(a^+\) and \(a^-\) be the two approaches to
\(a\) along \(\partial^+D\) and \(\partial^-D\), respectively. The following Theorem~\ref{thm:sle} follows readily from Theorem~\ref{thm:hull-unique} and~\cite[Proposition~6.2]{qian2021generalized}; see the end of Section~\ref{subsec:sec5-boundary-fans}.
\begin{theorem}[Boundary identification]\label{thm:sle}
For \(\lambda=-c/2\in(0,1]\), let
\(\kappa\in[2,\frac83)\) be in~\eqref{eq:c-kappa}.  
Let $\eta_-$ and $\eta_+$ be the lower and upper boundaries of \(K_\lambda^{D;a,b}\), respectively (i.e.~$\eta_\pm=\ol{\partial D_K^\pm\setminus \partial^\pm D}$). Then \(\eta_-\) is chordal
\(\SLE_\kappa(\kappa-2)\) in \(D\) from \(a\) to \(b\), with force point
at \(a^+\); conditionally on \(\eta_-\), \(\eta_+\) is chordal \(\SLE_\kappa(\kappa-4)\) in the component adjacent
to \(\partial^+D\),
with force point at \(a^-\).\footnote{At \(\kappa=2\), the law of the second curve
can be viewed as the \(\rho\downarrow-2\) limit of chordal \(\SLE_2(\rho)\) with force
point at \(a^-\); equivalently, $\eta_\pm$ are the two prime-end sides
of a chordal \(\SLE_2\) trace.} As a consequence, we have the following coupling between $\SLE_\kappa$ and Brownian excursion:
\[
 \eta_-\cup\eta_+\subset\gamma_{\BE}^{D;a,b}.
\]
\end{theorem}

That $\SLE_\kappa$ can be contained in a planar Brownian motion was previously known only at \(\kappa=\frac83\) \cite{lsw-restriction}
and \(\kappa=2\)~\cite{zhan2010}.
Theorem~\ref{thm:sle} gives all the intermediate values. Note that for $c\in(0,1]$, the lower boundary of the closed union of $\gamma_{\BE}^{D;a,b}$ and all its intersecting Brownian loop soup clusters from an independent $\cL_D^{c/2}$ is chordal $\SLE_\kappa(\kappa-2)$ for $\kappa\in(\frac{8}{3},4]$ in~\eqref{eq:c-kappa}~\cite{werner-wu-cle-sle}. (The joint law of lower and upper boundaries can be further identified by~\cite{CG25}.) Thus, Theorem~\ref{thm:sle} further indicates that Brownian loop-catchers are the canonical negative-central-charge counterparts of Brownian loop soup clusters.

Theorem~\ref{thm:full-trace-unique} also has the following application. In dimensions two and three, as shown in~\cite{SS18}, Brownian motion can be decomposed into a simple path and the loops from an independent intensity-one Brownian loop soup that hit it. The authors then asked whether this decomposition characterizes the law of that simple path~\cite[Conjecture~1.3]{SS18}. Taking $c=-2$, Theorem~\ref{thm:full-trace-unique} immediately proves the corresponding two-dimensional chordal case.

Let $\mathfrak m$ be a probability measure on unparametrized simple curves in $\overline\D$ from $-1$ to $1$.  We say that $\mathfrak m$ satisfies the \emph{loop-erased property} if, when $\gamma\sim\mathfrak m$ and \(\cL_\D^1\) is an independent intensity-one Brownian loop soup in $\D$, the closed union of $\gamma$ and all loops in \(\cL_\D^1\) that intersect $\gamma$ has the law of $\gamma_{\BE}^{\D;-1,1}$.

\begin{theorem}\label{thm:le-uniqueness}
Chordal $\SLE_2$ on $\D$ from $-1$ to $1$ is the unique probability measure on unparametrized simple curves satisfying the loop-erased property above.
\end{theorem}

\begin{proof}[Proof of Theorem~\ref{thm:le-uniqueness}, assuming Theorem~\ref{thm:full-trace-unique}]
The loop-erased property is equivalent to
\eqref{eq:intro-main-id} with \(c=-2\) and \(\lambda=1\).  By~\cite[Theorem 1.1]{SS18} (see also~\cite{Ambrosio2020,BS26}), chordal $\SLE_2$ satisfies the loop-erased property. The result follows from Theorem~\ref{thm:full-trace-unique}.
\end{proof}

We mention that uniqueness of the usual generalized restriction measures does not by itself imply Theorem~\ref{thm:le-uniqueness}, because the loop-erased property provides neither the laws of conformal images of $\mathfrak m$ nor the corresponding Radon--Nikodym derivatives. As mentioned before, our proof of Theorem~\ref{thm:full-trace-unique} is based on deriving moments of Green functions in the remaining domain, which can also be extended to the three-dimensional case of~\cite[Conjecture 1.3]{SS18}; see Section~\ref{rmk:3d}.

\subsection{Proof strategies}\label{subsec:organization}

Our proof has three stages, organized around two innovations introduced in this paper: the entangled multipath LERW and the Green function test.

\medskip
\noindent\emph{Finite graphs.}
We first introduce the entangled multipath LERW on arbitrary finite graphs. Fix $r\ge2$ ordered endpoint pairs \((a_i,b_i)_{i=1}^r\), and sample independent
normalized random-walk paths \(\omega_i:a_i\to b_i\).  Chronological loop
erasure of \(\omega_i\) gives its LERW path \(\eta_i\) and an
ordered record list of erased loops.  Starting with \(C_1=\eta_1\),
process these records successively.  For each \(2\le i\le r\), inspect the erased
loops in their erasure order and retain a loop precisely when it
intersects \(C_{i-1}\) or a loop already retained at that stage (and discard it otherwise).  After
all erased loops have been inspected, add \(\supp(\eta_i)\) to obtain
\(C_i\).  We call the final output \(C_r\) the (range of) \emph{entangled multipath LERW}.

The motivation of the entangled multipath LERW is to obtain a random set whose decoration by an
independent intensity-one random-walk loop soup has the law of the union
of the independent random-walk paths.  The construction above satisfies
\[
 \E\!\left[
 \one_{\{C_r\subset R\}}
 e^{-m_V(C_r,V\setminus R)}
 \right]
 =
 \prod_{i=1}^r
 \frac{\Gc_R(a_i,b_i)}{\Gc_V(a_i,b_i)}.
\]
See Section~\ref{subsec:entangled-multipath} for details.
The additionally retained erased loops exactly record interactions with paths processed earlier, which
are absent from the union of the individual LERWs.

\begin{figure}[H]
    \centering
    \begin{minipage}[t]{0.4\textwidth}
        \centering
        \includegraphics[
            width=\linewidth,
            trim={280 140 280 90},
            clip
        ]{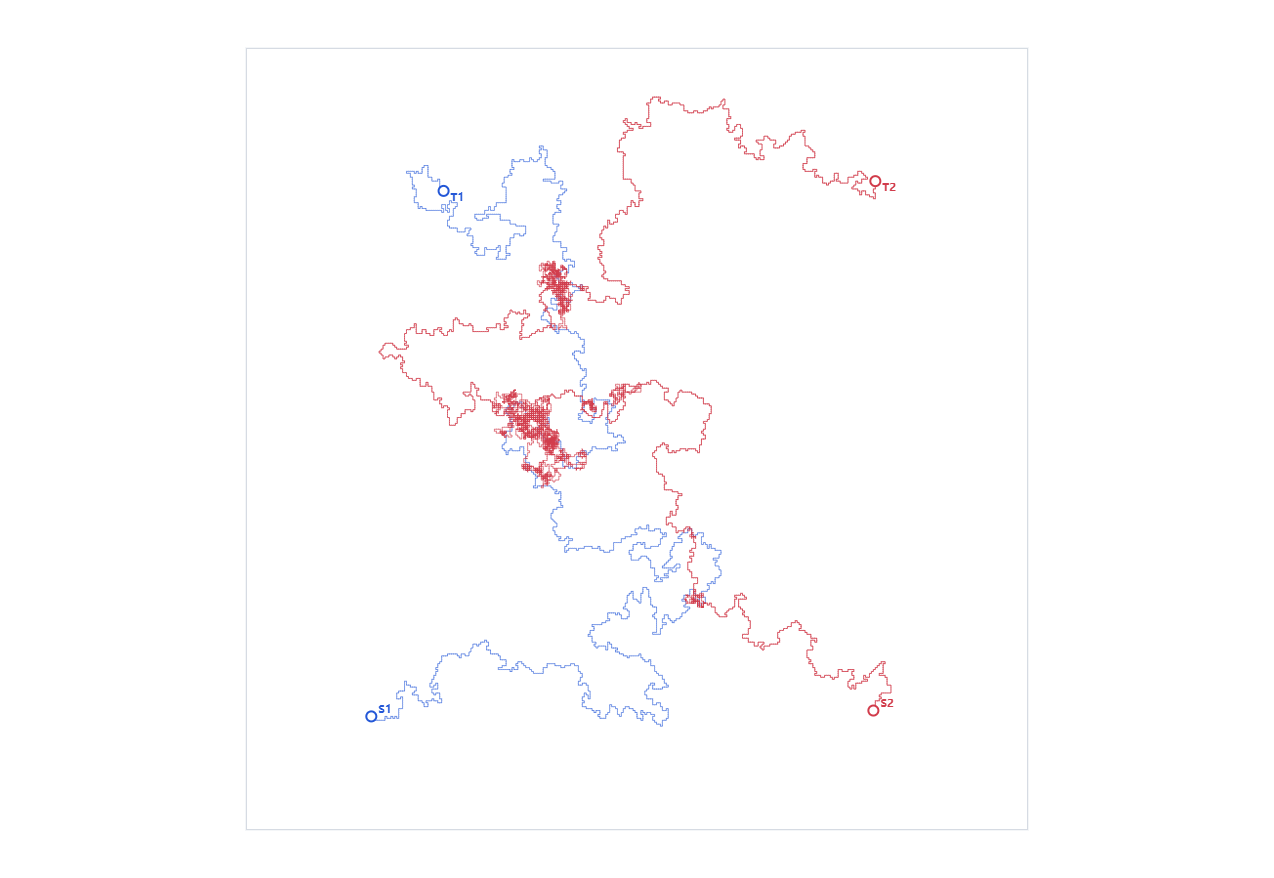}
    \end{minipage}
    \hspace{0.05\textwidth}
    \begin{minipage}[t]{0.4\textwidth}
        \centering
        \includegraphics[
            width=\linewidth,
            trim={280 140 280 90},
            clip
        ]{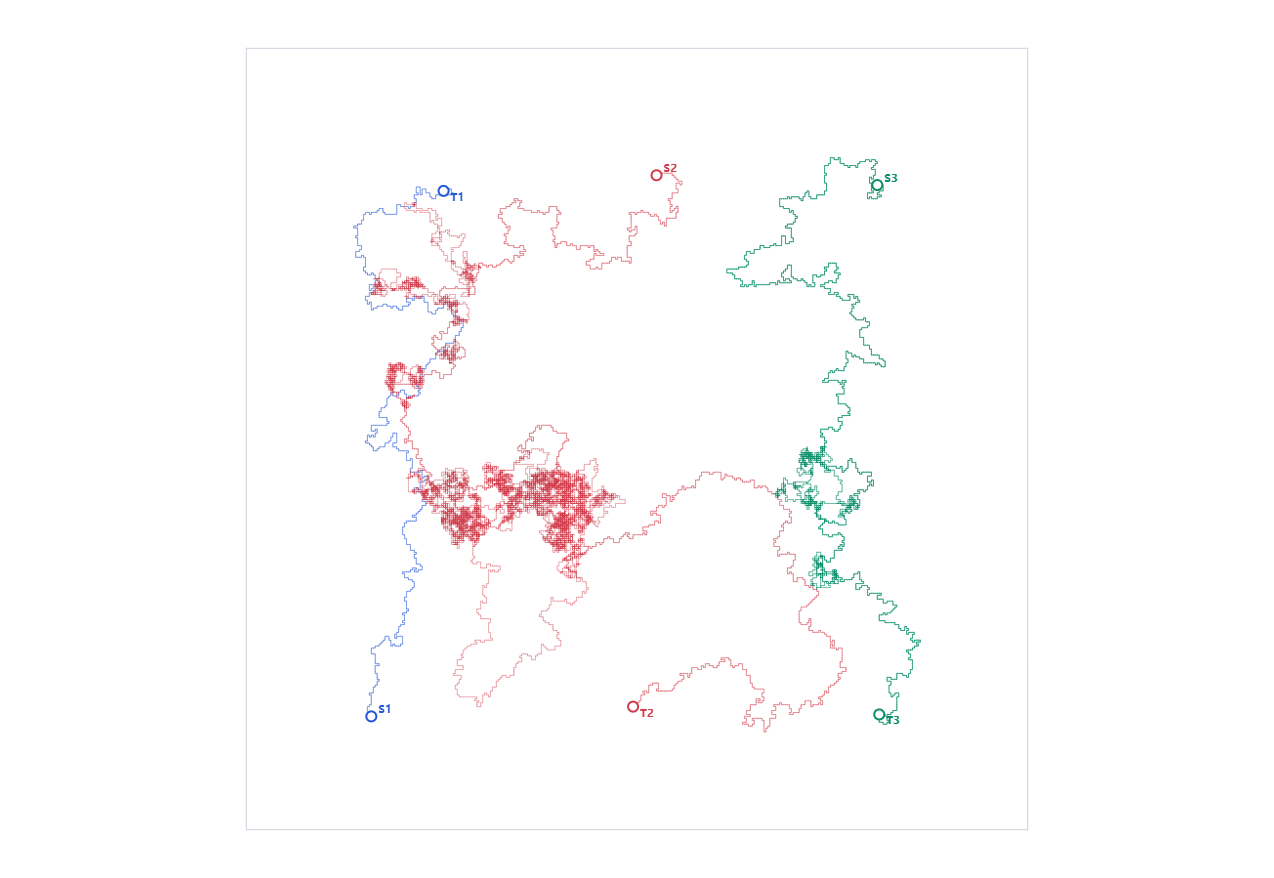}
    \end{minipage}
    \caption{Simulations of entangled multipath LERWs on $\frac{1}{400}\Z^2\cap[-1,1]^2$. \textbf{Left}: $r=2$; \textbf{Right}: $r=3$.}
    \label{fig:entangled-multipath-lerw-simulations}
\end{figure}

We now explain our main idea of the construction of the (time-parametrized) random-walk loop-catcher. For \(0\leq\lambda\leq1\), start from a random walk \(\omega\) and its chronological loop erasure \(\eta\). The random-walk loops attached at the first-visit times of \(\eta\) are deleted independently with probability \(\lambda\). Each retained loop is then cut into its maximal excursions in the complement of \(\eta\), and the entangled multipath loop-erasure is applied to the resulting ordered collection of random-walk excursions. This produces a collection of paths together with further random-walk loops in the complement of \(\eta\). We recursively apply the same decomposition to these loops. Since each recursive call takes place in a strictly smaller vertex set, the procedure terminates after finitely many steps. Substituting the resulting paths back through the recursive layers produces a finite directed path, which is the \emph{time-parametrized} \(\lambda\)-random-walk loop-catcher (and its \emph{range} is the $\lambda$-random-walk loop-catcher in Section~\ref{subsec:finite-main}). At \(\lambda=0\) (resp.~$\lambda=1$), the output is $\omega$ (resp.~$\eta$).
Conversely, inserting an independent random-walk loop soup with intensity \(\lambda\) from the deepest recursive layer back to layer $0$ restores the deleted loops at every stage, and hence recovers the original random walk with its time parametrization. See Algorithms~\ref{alg:time-catcher} and~\ref{alg:time-recovery} for details.

\medskip
\noindent\emph{Subsequential limits to continuum.}
For lattice approximations of the Jordan domain $D$, coupling the random-walk loop-catchers with the
random-walk excursions gives tightness in the Hausdorff topology and connected
subsequential limits. Convergence of random-walk excursions and loop measures passes the discrete identity~\eqref{eq:finite-transform-intro} to~\eqref{eq:intro-main-id}, and the Beurling estimate controls possible discontinuities. The main ingredients in this step come from~\cite{KozdronLawler2005,LawlerTrujillo2007,LawlerLimic2010}. This yields the existence of Brownian loop-catchers, while the uniqueness is left to the next step.

\medskip
\noindent\emph{Green function tests: characterization of continuum limits.}
The second innovation of the paper is to extend the
conformal-restriction testing idea from boundary-attached obstacles to
random interior multipath probes, which captures the moments of Green functions on any interior points. Based on the entangled multipath LERW introduced above and taking subsequential limits, Proposition~\ref{prop:sec5-continuum-probe} constructs, for any finite list of pairs of interior points $(x_i,y_i)$, a random compact \(C\Subset D\) satisfying
\[
 \E_C\!\left[\one_{\{C\cap E=\varnothing\}}e^{-\lambda m_D(C,E)}\right]
 =\prod_i\frac{G_{D\setminus E}(x_i,y_i)}{G_D(x_i,y_i)},\qquad \forall \text{ compact } E\subset\ol D.
\]
Taking such $C$ in~\eqref{eq:intro-main-id}, we then obtain all mixed moments of
\(X_{x,y}(K)=G_{D\setminus K}(x,y)/G_D(x,y)\).  Different pairs give joint products, while repeated pairs give powers.
We emphasize that our multipath construction is crucial as a single-path probe gives only first moments of Green functions.
Since these coordinates take values in \([0,1]\), their mixed moments
determine their joint laws.
We also show that the collection of Green functions on a countable dense set of
pairs determines the law of $K$ (see Lemma~\ref{lem:sec5-green-injective}). Consequently, all subsequential
limits obtained in the above stage coincide, proving full convergence in
Theorem~\ref{thm:catcher-convergence} and conformal invariance. Boundary probes give the analogous hull characterization in Theorem~\ref{thm:hull-unique}, and comparing with the conformal restriction properties of SLEs~\cite{dubedat-rho,qian2021generalized} then yields Theorem~\ref{thm:sle}. Theorem~\ref{thm:one-arm} follows from comparing Brownian loop-catchers and $\SLE_\kappa$ loop measures, and then applying a Tauberian theorem.

At a conceptual level, this idea of Green function tests can be viewed as the negative central charge extension of~\cite{CG25} in proving the uniqueness of generalized restriction measures. We stress again that the Green function test does not require any simple connectedness, and is also valid for three dimensions (see below).

\subsection{The Green function test in three dimensions}\label{rmk:3d}

In this section, we outline the proof of the three-dimensional (3D) case of~\cite[Conjecture 1.3]{SS18} using the Green function test. A fully detailed proof will appear in our future work~\cite{CCLS26}.

To be explicit, let \(B:=\{x\in\R^3:|x|<1\}\) be the unit ball in $\R^3$, and let \(\Gamma\) be the space of simple curves $K$ from \(0\) to \(\partial B\). As in the two-dimensional case, we say that a probability measure $\mathfrak m$ on $\Gamma$ satisfies the \emph{loop-erased property} if, when $K\sim\mathfrak m$ and \(\cL_B^1\) is an independent intensity-one Brownian loop soup in $B$, the closed union of $K$ and all loops in \(\cL_B^1\) that intersect $K$ has the same law as the trace of the 3D standard Brownian motion starting from $0$ until hitting $\partial B$. Then
\begin{equation}\label{thm:appB-conditional-uniqueness}
\textit{There is a unique probability measure } \mathfrak m \textit{ on } \Gamma \textit{ that satisfies the loop-erased property above}.
\end{equation}

Here we sketch the proof of \eqref{thm:appB-conditional-uniqueness}. Let $K$ be sampled from $\mathfrak m$. Note that if a local piece $J$ of $K$ were polar, then with positive probability no Brownian loop attached to $K$ would enter the neighborhood of $J$, leaving a bare polar $J$ of dimension at most \(1\) (see e.g.~\cite[Theorem 5.9.6]{ArmitageGardiner2001}), whereas every nonempty local piece of a 3D Brownian trace has dimension \(2\).
Thus, every piece of $K$ is nonpolar in $\R^3$, and hence $K$ is determined by all the Green functions in $B\setminus K$. Then compared with the proof of Theorem~\ref{thm:full-trace-unique} in Section~\ref{subsec:sec5-full-trace},
only two new dimension-dependent inputs are needed (recall that the construction of the entangled multipath LERW holds on any finite graph).  First, one
needs to show that each subsequential scaling limit of the entangled multipath LERW in 3D is hittable (in the sense of~\cite[Theorem 3.1]{SS18}).  Second, one needs a 3D analogue of the loop-mass convergence in Section~\ref{sec:continuum}.  The first point essentially follows from the construction of entangled multipath LERW, see Lemma~\ref{lem:multipath-initial-segment}, which shows every part of the entangled multipath LERW is locally an ordinary LERW. Hence the quantitative hittability estimates in~\cite[Lemmas 3.2 and 3.3]{SS18} apply, and sufficiently strong polynomial tails there allow a direct union bound.  For the second point, the required coupling and
approximation estimates can be obtained from~\cite{SS18} (see also~\cite{Qian2026}). We will provide complete details in~\cite{CCLS26}.

Historically, the existence of the scaling limit of LERW on $\Z^3$ was first established in~\cite{kozma-lerw-3d}, together with its invariance under rotations and
dilations. Subsequent work proved the existence of
the growth exponent and identified it with the Hausdorff dimension of
the scaling limit, established convergence in the natural
parametrization and sharp H{\"o}lder regularity, and showed that the
limiting occupation measure is equivalent to the corresponding
Minkowski content~\cite{ShiGrowth,ShiDimension,LSNatural,LSHolder,HLSMinkowski}.
However, what remained open was an intrinsic characterization of the limiting
law, independent of lattice approximation. This was formulated in~\cite[Conjecture~1.3]{SS18} as the conjecture that the loop-erased property characterizes the scaling limit, which is stated in \eqref{thm:appB-conditional-uniqueness}. As noted in~\cite{SS18}, this characterization gives a new route
to the existence and rotational invariance of the scaling limit; it
also provides a route to inversion invariance. In~\cite{CCLS26}, we will
use it to establish the universality of 3D LERW and,
through Wilson's algorithm, of the 3D uniform spanning
tree.

We end this section by discussing the Brownian loop-catchers in 3D.
Note that random-walk loop-catchers can be defined on 3D lattice approximations; their traces are still tight and have subsequential scaling limits. We expect that our Green function test characterizes these subsequential limits as well: according to the construction in Algorithm~\ref{alg:time-catcher}, each subsequential limit locally stochastically dominates the scaling limit of 3D LERW~\cite{kozma-lerw-3d,SS18} and is therefore hittable. We will provide the full details in~\cite{CCLS26}. A further question is the Hausdorff dimension of a 3D Brownian loop-catcher $K^{(3)}_\lambda$.  At $\lambda=0$, one recovers a 3D Brownian trace (of Hausdorff dimension two), whereas at $\lambda=1$, it becomes the scaling limit of 3D LERW whose dimension is numerically $\beta\approx1.624$~\cite{Wilson2010}.  Does $K^{(3)}_\lambda$ have dimension two throughout $0\leq\lambda<1$ and then undergo a phase transition at $\lambda=1$ (similar to Theorem~\ref{thm:one-arm})?  Or does its Hausdorff dimension decrease continuously from $2$ to $\beta$? Numerical experiments appear to support the latter scenario, while a rigorous proof remains open. 

\subsection{Outlook}
\label{sec:future-directions}

We list several further directions and open questions arising from this work.

\begin{enumerate}[label=\textbf{\arabic*.},leftmargin=*]

\item \textbf{Analytic continuation from exact Brownian loop soup formulas.}
Recently, using couplings between SLE, Liouville quantum gravity, and Brownian loop soups, we computed exact probabilities for several natural topological events of Brownian loop soup clusters, including the touching probability~\cite{touching} and the disconnection probability~\cite{CFSX25,BM-intersection}.  The formulas are analytic in $c$ for $c\in[0,1]$ and admit analytic continuations to $c<0$. As indicated in this paper, we expect these analytic continuations to coincide precisely with the probabilities of the corresponding events for Brownian loop-catchers with $c<0$.  We plan to prove this rigorously in the future.

\item \textbf{Natural parametrization of Brownian loop-catcher.}
Section~\ref{sec:finite-positivity} realizes the finite-graph loop-catcher as the range of a directed path, such that adding loops from an independent random-walk loop soup recovers the original random walk with time parametrization.  We expect that for $0<\lambda<1$, after taking the normalization $|\log\delta|^{-\lambda}$, the parametrized random-walk loop-catcher on $\delta\Z^2\cap\D$ will converge to the corresponding Brownian loop-catcher with its natural parametrization. (Theorem~\ref{thm:catcher-convergence} currently gives only the convergence as random compact sets.)  At $\lambda=1$, this reduces to the natural parametrization of LERW and $\SLE_2$~\cite{lawler-viklund-lerw-nat}; while for $0<\lambda<1$, the Brownian loop-catcher is not a simple curve and requires further delicate estimates.

\item \textbf{A purely continuum construction of Brownian loop-catchers.}
Our construction of $K_\lambda$ is obtained by taking the scaling limit of random-walk loop-catchers on lattice approximations.  Can Brownian loop-catchers instead be constructed directly in the continuum, without reference to a lattice approximation?  It would be particularly interesting to find an explicit continuum operation that interpolates between the Brownian trace at $\lambda=0$ and its loop erasure at $\lambda=1$.

\item \textbf{Other SLEs inside planar Brownian motion.}
Theorem~\ref{thm:sle} extracts $\SLE_\kappa$-type curves with any $\kappa\in[2,\frac{8}{3}]$ from a planar Brownian trace by taking boundaries of Brownian loop-catchers.  Can one find $\SLE_\kappa$-type curves with other values of $\kappa$ inside planar Brownian motion?  In particular, can one find $\SLE_\kappa$ traces with $\kappa<2$?  This would further help to understand the geometric features of planar Brownian motion.

\end{enumerate}

\medskip
\noindent
\textbf{Organization of the paper.}
Section~\ref{sec:finite-positivity} constructs the time-parametrized random-walk loop-catcher on finite graphs and proves Theorem~\ref{thm:finite-positivity}. Section~\ref{sec:continuum} provides necessary convergence inputs, and shows the existence of Brownian loop-catchers. In Section~\ref{sec:sle-proof}, we use the Green function test to establish Theorems~\ref{thm:full-trace-unique} and~\ref{thm:hull-unique}, thus proving the full convergence (Theorem~\ref{thm:catcher-convergence}), boundary identification with SLE (Theorem~\ref{thm:sle}), as well as the non-existence of Brownian loop-catcher when $\lambda>1$. This completes Theorem~\ref{thm:main}. Section~\ref{sec:one-arm} proves the one-point density of Brownian loop-catcher, i.e.~Theorem~\ref{thm:one-arm}. Appendix~\ref{app:finite-positivity} gives an alternative linear-algebraic proof of Theorem~\ref{thm:finite-positivity}.

\medskip
\noindent
\textbf{Acknowledgments.}
We are grateful to Yifan Gao, Ewain Gwynne, Saraí Hernández-Torres, Tom Hutchcroft, Richard Kenyon, Greg Lawler, Xinyi Li, Wei Qian, Scott Sheffield, Daisuke Shiraishi, Xin Sun, Wendelin Werner, and Baojun Wu for many helpful discussions on various aspects of this work. We particularly thank Xinyi Li for suggesting the names \emph{Brownian loop-catcher} and \emph{entangled multipath LERW}. We also thank the organizers of the workshop \emph{Random Explorations: From Random Walks to Random Geometry}, held at the University of Chicago from July 6 to 10, 2026, in celebration of Greg Lawler's 70th birthday. Part of this work was presented there.

This work was supported by the National Key R\&D Program of China (No. 2021YFA1002700), Beijing Natural Science Foundation (JQ26001), and the National Natural Science Foundation of China (Grant No.12526204).

\section{The random-walk loop-catcher}\label{sec:finite-positivity}

Recall the notation from Section~\ref{subsec:finite-main}.  We fix a finite
set $V$, a nonnegative substochastic matrix $Q=(Q_{xy})_{x,y\in V}$ with
$\rho(Q)<1$, nonnegative entrance and exit vectors $\balpha,\bbeta$, and
$Z_V=\balpha^{\mathsf T}(I-Q)^{-1}\bbeta>0$.  For $R\subset V$, recall
$d_R$, $Z_R$, and $\bh(R)$ from~\eqref{eq:z-d-bh}, with
$d_\varnothing=1$ and $Z_\varnothing=0$.  For $A,B\subset V$, write
$Q_{A,B}=(Q_{xy})_{x\in A,y\in B}$ and $Q_A=Q_{A,A}$, and put
$\Gc_R=(I-Q_R)^{-1}$.  For a finite directed path or loop $\xi$, let
$q(\xi)$ be the product of its edge weights and $\supp(\xi)$ its vertex set.

For $a,b\in R$ with $\Gc_R(a,b)>0$, a normalized random walk $\omega:a\to b$ in $R$ has law $\Prob[\omega]=\frac{q(\omega)}{\Gc_R(a,b)}.$
We will always view $\omega$ as a directed path.

\subsection{The entangled multipath LERW}
\label{subsec:entangled-multipath}

We first introduce the following entangled multipath LERW.  Let $R\subset V$, and let $\omega$ be
a directed path in $R$ from $a$ to $b$, possibly of length zero.  Its
chronological loop erasure produces
\[
 \omega\longmapsto P(\omega)
 =\bigl(\eta;\ell_1,\ldots,\ell_m\bigr),
 \qquad \eta=\LE(\omega),
\]
where $\eta$ is self-avoiding and $\ell_1,\ldots,\ell_m$ are the simple
directed loops in erasure order; loops of length one are included.  We
identify directed loops under cyclic shifts of the starting vertex (but not
under reversal), and distinguish repeated copies.  We call $P(\omega)$ the
\emph{record} of $\omega$.

Generally, a \emph{record} is a tuple of this form, modulo adjacent exchanges
of directed simple loops with disjoint supports. For a record $P
 =\bigl(\eta;\ell_1,\ldots,\ell_m\bigr)$, let $\prec$ be the transitive closure of
\[
 \ell_i\prec\ell_j
 \quad\text{if }i<j\text{ and }
 \supp(\ell_i)\cap\supp(\ell_j)\ne\varnothing,
 \qquad
 \ell_i\prec\eta
 \quad\text{if }
 \supp(\ell_i)\cap\supp(\eta)\ne\varnothing.
\]
We call $\prec$ the \emph{entangled order}.  A record is \emph{legal} if
every directed simple loop precedes $\eta$.  Chronological loop erasure gives a
weight-preserving bijection between directed paths and legal records;
see \cite[Theorem~A.5]{Hel16}.  For a legal record
$P=(\eta;\ell_1,\ldots,\ell_m)$, write
\[
 q(P)=q(\eta)\prod_{j=1}^m q(\ell_j),
 \qquad
 \supp(P)=\supp(\eta)\cup\bigcup_j\supp(\ell_j).
\]
Thus $q(P(\omega))=q(\omega)$.

Fix $A\subset R$.  Given a legal record $P=(\eta;\ell_1,\ldots,\ell_m)$,
extend the entangled order by setting $A\prec\ell_k$ when $A\cap\supp(\ell_k)\ne\varnothing$, and take the transitive closure.  Choose a representative, start with $A_0=A$, and inspect its directed simple loops.  If
\[
  \supp(\ell_k)\cap A_{k-1}\ne\varnothing,
\]
retain $\ell_k$ and put
\[
  A_k=A_{k-1}\cup\supp(\ell_k);
\]
otherwise discard it and put $A_k=A_{k-1}$.  
Let \(\Phi_A(P)\) be the record obtained by retaining \(\eta\) and precisely those erased loops \(\ell\) satisfying \(A\prec\ell\). Then \(\Phi_A(P)\) is still legal.

\begin{definition}[entangled multipath loop-erasure]\label{def:entangled}
For $r\ge1$ and an ordered tuple of directed paths $\omega_i:a_i\to b_i$, $1\le i\le r$, put
\begin{equation}\label{eq:sequential-reduction}
 H_i=P(\omega_i),\qquad C_0=\varnothing,\qquad
 P_i=\Phi_{C_{i-1}}(H_i),\qquad
 C_i=C_{i-1}\cup\supp(P_i).
\end{equation}
Let $\gamma_i$ be the directed path corresponding to the record $P_i$ (equipped with the time parametrization).  We call $(\gamma_1,\ldots,\gamma_r)$
the entangled multipath loop-erasure of $(\omega_1,\ldots,\omega_r)$, whose range is $C_r$.
\end{definition}

Note that for $r=1$, no erased loop is retained, and $\gamma_1=\LE(\omega_1)$. When $(\omega_1,\ldots,\omega_r)$ is sampled from independent normalized random walks, we call the corresponding $(\gamma_1,\ldots,\gamma_r)$ the \emph{entangled multipath LERW}.
The following proposition establishes the recovery property: the original tuple of random walks can be recovered by adding loops to the entangled multipath LERW from a common independent random-walk loop soup.

\begin{proposition}\label{prop:entangled-recovery}
Let $(P_1,\ldots,P_r)$ be the output records of independent normalized random walks
$\omega_i:a_i\to b_i$ ($1\le i\le r$) in Definition~\ref{def:entangled}, and let $\gamma_i$ be the corresponding
 entangled multipath LERW.  Independently sample an unrooted random-walk loop soup in $R$ with intensity $1$, and assign each loop that intersects $C_r$ to the smallest $i$ for which
it intersects $\gamma_i$, and then to the earliest path time $s$ on $\gamma_i$
at which such an intersection occurs.  If the loop visits $\gamma_i(s)$ more than once, choose uniformly one of those visits as its starting time. For each time on the entangled multipath LERW, order its assigned random-walk loops by an independent
uniform permutation and insert them in that order. Then conditioned on
$(P_1,\ldots,P_r)$, the resulting tuple of paths has the same law as the conditional law of $(\omega_1,\ldots,\omega_r)$ given $(P_1,\ldots,P_r)$.
\end{proposition}

The $r=1$ version of Proposition~\ref{prop:entangled-recovery} reduces to the classical decomposition of a random walk path via its loop erasure and an independent random-walk loop soup, see e.g.~\cite[Proposition 5.9]{Lawler2018}.
To prove Proposition~\ref{prop:entangled-recovery}, we need the following lemma.

\begin{lemma}\label{lem:one-step}
Fix $A\subset R$,  let $H$ be a legal record from $a$ to $b$, and set
$P=\Phi_A(H)$.  List $\supp(P)\setminus A$ as $u_1,\ldots,u_m$ in first-visit
order along the directed path $\gamma$ corresponding to $P$. For each $1\le j\le m$, let $t_j$ be the first-visit time of $u_j$
and
\[
 D_j=R\setminus\bigl(A\cup\{u_1,\ldots,u_{j-1}\}\bigr).
\]
Then there are unique ordered lists $(\ell_{j,s})_{1\le s\le n_j}$ of loops in
$D_j$ from $u_j$ to $u_j$, with no intermediate visit to $u_j$, whose insertion
in $\gamma$ at $t_j$ reconstructs $H$.  Conversely, every such collection of lists determines one legal
record $H$ with $\Phi_A(H)=P$, and
\begin{equation}\label{eq:positive-fiber-weight}
 q(H)=q(P)\prod_{j=1}^m\prod_{s=1}^{n_j}q(\ell_{j,s}).
\end{equation}
Consequently,
\begin{align}
 d_{R\setminus A}
 \sum_{\substack{H\text{ legal in }R\\\Phi_A(H)=P}}q(H)
 &=q(P)d_{R\setminus(A\cup\supp(P))},\label{eq:fiber}\\
 \Gc_R(a,b)d_{R\setminus A}
 &=\sum_{\substack{P:a\to b\text{ legal in }R\\\Phi_A(P)=P}}
 q(P)d_{R\setminus(A\cup\supp(P))}.\label{eq:one-step}
\end{align}

As a corollary, suppose $H$ is the record of a normalized random walk from $a$ to $b$ in $R$. Then conditioned on $P=\Phi_A(H)$, the lists are
independent and
\begin{equation}\label{eq:rooted-walk-law}
 \Prob[(\ell_{j,1},\ldots,\ell_{j,n_j})=(\xi_1,\ldots,\xi_d)\mid P]
 =\frac{\prod_{s=1}^d q(\xi_s)}{\Gc_{D_j}(u_j,u_j)}.
\end{equation}
\end{lemma}

\begin{proof}
Let $\mathcal D$ be the directed simple loops of $H$ discarded by $\Phi_A$.  They
avoid $A$, and every predecessor of a loop in $\mathcal D$ is also in
$\mathcal D$.  Process $u_1,\ldots,u_m$ in order.  The unassigned loops visiting
$u_j$ are pairwise comparable because they intersect at $u_j$; assign the
last one and all its unassigned predecessors to $u_j$.  With the trivial terminal path
$(u_j)$ they form a legal record, hence a closed walk from $u_j$ to $u_j$; cutting it at its successive returns to $u_j$ yields the stated ordered list.  It lies in $D_j$, since after step
$i<j$ no unassigned loop visits $u_i$.

Note that every loop in $\mathcal D$ is assigned. Otherwise, follow an increasing chain from
a maximal unassigned loop to $P$; its first assigned member would force the
original loop to have been assigned at the same or an earlier step.  The
same observation preserves every order relation between intersecting loops,
so insertion at the times $t_j$ recovers $H$.  Conversely, the list inserted
at $u_j$ avoids $A,u_1,\ldots,u_{j-1}$; each new erased loop therefore
precedes every retained loop that it meets, and $\Phi_A$ discards exactly the
inserted loops.  Thus insertion is the inverse construction.  Multiplicativity of edge weights gives
\eqref{eq:positive-fiber-weight}.

The total weight of all ordered lists of such walks in $D$ is
$\Gc_D(u,u)$.  Hence
\[
 \prod_{j=1}^m\Gc_{D_j}(u_j,u_j)
 =\prod_{j=1}^m\frac{d_{D_j\setminus\{u_j\}}}{d_{D_j}}
 =\frac{d_{R\setminus(A\cup\supp(P))}}{d_{R\setminus A}};
\]
see e.g.~\cite[Proposition~3.5]{Lawler2018}.
Summing~\eqref{eq:positive-fiber-weight} proves \eqref{eq:fiber}, then
\eqref{eq:one-step}.
\end{proof}

\begin{proof}[Proof of Proposition~\ref{prop:entangled-recovery}]
For each $i$, list $\supp(P_i)\setminus C_{i-1}$ in first-visit order along
$\gamma_i$ as $(u_{i,1},\ldots,u_{i,m_i})$, and set
$D_{i,j}=R\setminus\bigl(C_{i-1}\cup \{u_{i,1},\ldots,u_{i,j-1}\}\bigr)$.  The loops assigned to $u_{i,j}$ are precisely those contained in
$D_{i,j}$ and visiting $u_{i,j}$. By Lemma~\ref{lem:one-step} and Poisson restriction, it suffices to show that at each $u_{i,j}$, the procedure in Proposition~\ref{prop:entangled-recovery} yields the same law of the ordered list $(\xi_1,\ldots,\xi_d)$ as that in Lemma~\ref{lem:one-step}.

Recall that we root every such loop assigned to $u_{i,j}$ uniformly at its visits to
$u_{i,j}$, and uniformly order these loops.  For a fixed ordered list
$(\xi_1,\ldots,\xi_d)$, the probability that it occurs is equal to
\[
 \frac{\prod_{s=1}^d q(\xi_s)}{\Gc_{D_{i,j}}(u_{i,j},u_{i,j})}
 \sum_{M\ge1}\ \sum_{k_1+\cdots+k_M=d}
 \frac1{M!k_1\cdots k_M}
 =\frac{\prod_{s=1}^d q(\xi_s)}{\Gc_{D_{i,j}}(u_{i,j},u_{i,j})};
\]
the left side sums over the cases that $(\xi_1,\ldots,\xi_d)$ is produced by $M$ loops from the random-walk loop soup, which have multiplicities $(k_1,...,k_M)$ at $u_{i,j}$, respectively. 
Here the sum is equal to the coefficient of $z^d$ in
$\exp(\sum_{j\ge1}z^j/j)=(1-z)^{-1}$, which is $1$. The empty list has probability
$\Gc_{D_{i,j}}(u_{i,j},u_{i,j})^{-1}$. This is exactly~\eqref{eq:rooted-walk-law}, as desired.
\end{proof}

As an immediate consequence of Proposition~\ref{prop:entangled-recovery}, we have for every $W\subset R$,
\begin{equation}\label{eq:multipath-recovery-at-one}
 \E\left[
  \one_{\{C_r\subset W\}}
  e^{-m_R(C_r,R\setminus W)}
 \right]
 =\prod_{i=1}^r\frac{\Gc_W(a_i,b_i)}{\Gc_R(a_i,b_i)},
\end{equation}
where the right-side is zero if an endpoint is outside $W$.

We also record the following lemma, which shows that every vertex of the entangled multipath LERW lies on the loop erasure of an initial segment of one of the random walks.

\begin{lemma}\label{lem:multipath-initial-segment}
Let $C_r$ be the range of the entangled multipath LERW obtained from
$(\omega_i)_{1\le i\le r}$.  For every \(w\in C_r\), there exists an \(\omega_i\) which has an initial segment \(\omega_i'\) ending at \(w\) such that $\supp\!\left(\LE(\omega_i')\right)\subset C_r$.
\end{lemma}

\begin{proof}
Let \(i\) be least with \(w\in C_i\), and write
\(P(\omega_i)=(\eta_i;\ell_1^i,\ldots,\ell_{m_i}^i)\).
If \(w\in\supp(\eta_i)\), we can clearly take \(\omega_i'\) so that
\(\LE(\omega_i')\) is the initial segment of \(\eta_i\) ending at \(w\).

Otherwise \(w\in\supp(\ell_j^i)\) for a loop retained in \(P_i\), so
\(C_{i-1}\prec\ell_j^i\).  Let
\(\omega_i'\) end at the visit to \(w\) that survives until \(\ell_j^i\)
is erased, and write
\(\LE(\omega_i')=(z_0,\ldots,z_s)\).  After this erasure, exactly
\(z_0,\ldots,z_{p_0}\) remain, with
\(z_{p_0},\ldots,z_s\in\supp(\ell_j^i)\).
Whenever the first later loop \(\ell_{j_{q+1}}^i\) shortens the remaining
list \(z_0,\ldots,z_{p_q}\), write \(z_0,\ldots,z_{p_{q+1}}\) for the new
list.  Then \(p_{q+1}<p_q\) and
\(z_{p_q},z_{p_{q+1}}\in\supp(\ell_{j_{q+1}}^i)\).  Thus the successive
loops are ordered by \(\prec\).  Since \(C_{i-1}\prec\ell_j^i\), all of
them are retained in \(P_i\).  Every removed vertex lies on one of these
loops, while the remaining vertices lie on \(\eta_i\).  Hence
$ \supp\!\left(\LE(\omega_i')\right)
 \subset\supp(\eta_i)\cup
 \bigcup_{\{k:C_{i-1}\prec\ell_k^i\}}\supp(\ell_k^i)
 =\supp(P_i)\subset C_r$.
\end{proof}

\subsection{The (time-parametrized) random-walk loop-catcher}
\label{subsec:finite-construction}

Now we use the entangled multipath LERW to construct the random-walk loop-catcher.
Fix $0\le\lambda\le 1$.  
For $a\ge0$, write
\[
 (a)^{\overline0}=1,\qquad
 (a)^{\overline n}=a(a+1)\cdots(a+n-1),\quad n\ge1.
\]

The random-walk loop-catcher (with time parametrization) is defined by the following algorithm.
\begin{algorithm}[Construction]
\label{alg:time-catcher} Let $(\omega_1,\ldots,\omega_r)$ be an ordered tuple of directed paths in $R\subset V$.

\begin{enumerate}[label=\textup{(\arabic*)},leftmargin=2.4em]
\item Let $(P_i,C_i)$ be as in~\eqref{eq:sequential-reduction}; let $\gamma_i$
be the path corresponding to $P_i$, and put
\[
 S=C_r,\qquad B=R\setminus S.
\]
These $\gamma_i$ form the current layer.\label{alg:1}

\item At each first-visit time in Lemma~\ref{lem:one-step}, let
$\ell_1,\ldots,\ell_d$ be the uniquely determined ordered list of loops.  Independently
at different first-visit times, choose $K\in\{0,\ldots,d\}$ with
\begin{equation}\label{eq:initial-block-law}
 \Prob[K=k]
 =\frac{(\lambda)^{\overline k}}{k!}
  \frac{(1-\lambda)^{\overline{d-k}}}{(d-k)!},
 \qquad 0\le k\le d,
\end{equation}
which is the beta--binomial distribution with parameters $d$, $\lambda$, and $1-\lambda$.
Delete $\ell_1,\ldots,\ell_K$ and retain
$\ell_{K+1},\ldots,\ell_d$.\label{alg:2}

\item For every retained loop $\ell$, list in its time order all maximal
segments
\[
 x,b_0,\ldots,b_n,y,
 \qquad x,y\in S,\quad b_0,\ldots,b_n\in B,\quad n\ge0.
\]
View $(b_0,\ldots,b_n)$ as a directed path in $B$. Label
this path by $(i,t,h,s)$, where $i$ is the index of the current path
$\gamma_i$, $t$ is the first-visit time of $\gamma_i$ at which $\ell$ is
attached, $h$ is the index of $\ell$ in the ordered list of loops at time $t$, and $s$ is the starting time of the segment within $\ell$. Order the paths $(b_0,\ldots,b_n)$ lexicographically by these
labels. Denote the resulting ordered tuple by
$(\zeta_1,\ldots,\zeta_N)$.\footnote{The portions of each retained loop lying in $S$ are left unchanged. In particular, a retained loop contained entirely in $S$ contributes no path to $(\zeta_1,\ldots,\zeta_N)$.}\label{alg:3}

\item If $N>0$, apply this algorithm recursively to the ordered
tuple $(\zeta_1,\ldots,\zeta_N)$ in $B$.  Denote the returned paths by
$(\zeta_1^\lambda,\ldots,\zeta_N^\lambda)$.  Replace each $\zeta_j$ by
$\zeta_j^\lambda$ at its recorded position.  If $N=0$, make no recursive
call. Note that the recursion terminates since $B$ is a proper subset of $R$.

\item After these replacements, insert the resulting retained loops (in their recorded order) at the corresponding first-visit times of the
$\gamma_i$.
\end{enumerate}

The output is a tuple of directed paths
$\omega_1^\lambda,\ldots,\omega_r^\lambda$. We also set $S_{\lambda,\times}:=\bigcup_{i=1}^r\supp(\omega_i^\lambda)$.
\end{algorithm}

 Intuitively,~\eqref{eq:initial-block-law} is obtained by independently deleting the random-walk loops from the random-walk loop soup with intensity $1$. Since such a random-walk loop may contribute several returns to the corresponding first-visit time, the induced number $K$ of deleted loops in $\ell_1,\ldots,\ell_d$, conditional on the total number $d$, is beta--binomial rather than binomial. 

Note that $C_r\subset S_{\lambda,\times}
 \subset\bigcup_{i=1}^r\supp(\omega_i)$, and every component of $S_{\lambda,\times}$ contains a path joining one of the
labelled endpoint pairs. Moreover, $\omega_i^1=\gamma_i$ and $\omega_i^0=\omega_i$.

\begin{definition}\label{def:catcher}
When $r=1$ and $\omega=\omega_1$ is sampled from the normalized random-walk excursion on $(V,Q,\balpha,\bbeta)$, the corresponding output $S_{\lambda,\times}$ in Algorithm~\ref{alg:time-catcher} is called the $\lambda$-random-walk loop-catcher. We also call $\omega_1^\lambda$ the time-parametrized random-walk loop-catcher.
\end{definition}
In the following, we show the random-walk loop-catcher satisfies~\eqref{eq:finite-transform-intro}, thus prove the existence part of Theorem~\ref{thm:finite-positivity} when $0\le\lambda\le 1$.
We first index the recursive layers in Algorithm~\ref{alg:time-catcher}.  Set
\[
 R_0=R,\qquad
 \Omega^{(0)}
 =(\omega_1^\lambda,\ldots,\omega_r^\lambda).
\]
For $n\ge0$, given an ordered path tuple $\Omega^{(n)}$ in $R_n$, let $\boldsymbol\gamma^{(n)}
 =(\gamma_1^{(n)},\ldots,\gamma_{r_n}^{(n)})$ be its entangled multipath loop-erasure, and let
\[
 S_n=\bigcup_{i=1}^{r_n}\supp(\gamma_i^{(n)}),
 \qquad
 R_{n+1}=R_n\setminus S_n.
\]
Applying Lemma~\ref{lem:one-step} successively to $\Omega^{(n)}$ determines the ordered lists of loops attached at the first-visit times of $\boldsymbol\gamma^{(n)}$.
Let
$\Omega^{(n+1)}$ be the ordered tuple of all maximal subpaths
of these loops lying in \(R_{n+1}\), ordered as in
Step~\ref{alg:3} of Algorithm~\ref{alg:time-catcher}. 
If $\Omega^{(n+1)}$ is
empty, set $L=n$ and stop. The recursive
layers are then $\boldsymbol\gamma^{(0)},\ldots,\boldsymbol\gamma^{(L)}$, with $S_{\lambda,\times}=\bigcup_{n=0}^L S_n$.
Note that the integer $L$ and, for every $0\le n\le L$, the tuples
$\Omega^{(n)}$ and $\boldsymbol\gamma^{(n)}$, the set $R_n$,
the ordered list of loops attached at the first-visit times of
$\boldsymbol\gamma^{(n)}$, and the positions of their maximal paths in
$R_{n+1}$ are all determined by 
$\Omega^{(0)}=(\omega_1^\lambda,\ldots,\omega_r^\lambda)$.

We now define an algorithm that inserts an independent random-walk loop soup into the output of Algorithm~\ref{alg:time-catcher}, which will recover the original random walks.

\begin{algorithm}[Recovery]
\label{alg:time-recovery}
Let $\Omega^{(0)}=(\omega_1^\lambda,\ldots,\omega_r^\lambda)$ be the output of
Algorithm~\ref{alg:time-catcher}, and let
$\Lc_R^\lambda$ be a realization of the random-walk loop soup in $R$ with intensity $\lambda$.

\begin{enumerate}[label=\textup{(\arabic*)},leftmargin=2.4em]
\item Recover all recursive layers $0\le n\le L$ as above.

\item For each $\ell\in\Lc_R^\lambda$ that intersects
$S_{\lambda,\times}$, let $n$ be the smallest index such that
$\ell\cap S_n\ne\varnothing$. Within layer~$n$, choose the smallest path index $i$ for
which $\ell$ meets $\gamma_i^{(n)}$, and then the earliest path time $t$ at
which it is met. If $\ell$ visits $\gamma_i^{(n)}(t)$ more than once, root
it uniformly at one of those visits. At each fixed $(n,i,t)$, order the
assigned rooted loops by an independent uniform permutation.

\item Process the layers in the order $L,L-1,\ldots,0$. At each
first-visit time, prepend the assigned rooted loops (understood as cut at their successive returns to the assigned root), in their chosen order,
to the ordered list of loops already there, and insert the resulting list as
in Lemma~\ref{lem:one-step}. After completing layer~$n>0$, substitute the
resulting paths into their recorded positions in the loops retained at
layer~$n-1$.
\end{enumerate}

The output is the directed paths obtained at layer~$0$.
\end{algorithm}

\begin{proposition}\label{prop:path-recovery}
Fix $R\subset V$, ordered endpoint pairs $(a_i,b_i)_{1\le i\le r}$ with
$\Gc_R(a_i,b_i)>0$, and $0\le\lambda\le 1$. Let
$\omega_i:a_i\to b_i$ be independent normalized random walks in $R$, and let $\Omega^{(0)}
 =(\omega_1^\lambda,\ldots,\omega_r^\lambda)$ be the output of Algorithm~\ref{alg:time-catcher}. Sample an independent random-walk loop soup $\Lc_R^\lambda$ in $R$ with intensity $\lambda$. Then the output of Algorithm~\ref{alg:time-recovery} has
the same joint law as $(\omega_1,\ldots,\omega_r)$.
\end{proposition}
\begin{proof}
We first record a loop-soup calculation.  Fix $u\in D\subset R$ and write $G=\Gc_D(u,u)$.
Consider an intensity-$\alpha$ random-walk loop soup in $D$.  Root each
loop visiting $u$ uniformly at one of its visits to $u$, uniformly order
the rooted loops, and cut them at their successive returns to $u$.  For
an ordered list $(\xi_1,\ldots,\xi_d)$ of loops that have multiplicity $1$ at $u$, the probability of obtaining this list is
\begin{equation}\label{eq:alpha-root-list}
 \Prob\bigl[(\xi_1,\ldots,\xi_d)\bigr]
 =
 G^{-\alpha}\frac{(\alpha)^{\overline d}}{d!}
 \prod_{s=1}^d q(\xi_s).
\end{equation}
Indeed, the loop mass of random-walk loops in $D$ visiting $u$ is $\log G$, and if
the list is produced by $M$ random-walk loops and the $m$-th loop is cut into
$j_m$ consecutive paths, then $j_1+\cdots+j_M=d$
and the corresponding contribution is $ G^{-\alpha}
 \frac{\alpha^M}{M!j_1\cdots j_M}
 \prod_{s=1}^d q(\xi_s)$.
Summing over $M$ and the ordered compositions of $d$ gives~\eqref{eq:alpha-root-list} (doing the same calculation as in the proof of Proposition~\ref{prop:entangled-recovery}).

We next apply~\eqref{eq:alpha-root-list} to Step~\ref{alg:2} of
Algorithm~\ref{alg:time-catcher}.  For an ordered tuple of independent normalized random walks, apply Algorithm~\ref{alg:time-catcher} and let $\boldsymbol P=(P_1,\ldots,P_m)$
be the records obtained in Step~\ref{alg:1}.  Conditioned on
$\boldsymbol P$, Lemma~\ref{lem:one-step}, applied successively in path
order, gives independent ordered lists of loops at the first-visit times.
Fix a first-visit time, let $u$ and $D$ be its vertex and associated domain
in Lemma~\ref{lem:one-step}, and let $(\ell_1,\ldots,\ell_d)$
be the corresponding ordered list of loops at the beginning of Step~\ref{alg:2}.  For
$0\le k\le d$, \eqref{eq:rooted-walk-law} and
\eqref{eq:initial-block-law} imply that the conditional probability $\Prob\!\left[
   (\ell_1,\ldots,\ell_d)=(\xi_1,\ldots,\xi_d),\,K=k
   \,\middle|\,\boldsymbol P
 \right]$ is equal to
\begin{align}\label{eq:list-splitting}
 \frac{\prod_{s=1}^d q(\xi_s)}{G}
 \frac{(\lambda)^{\overline k}}{k!}
 \frac{(1-\lambda)^{\overline{d-k}}}{(d-k)!}=
 \left[
  G^{-\lambda}
  \frac{(\lambda)^{\overline k}}{k!}
  \prod_{s=1}^k q(\xi_s)
 \right]
 \left[
  G^{-(1-\lambda)}
  \frac{(1-\lambda)^{\overline{d-k}}}{(d-k)!}
  \prod_{s=k+1}^d q(\xi_s)
 \right].
\end{align}
By \eqref{eq:alpha-root-list}, the two factors on the right side of~\eqref{eq:list-splitting} are the probabilities of obtaining $(\xi_1,\ldots,\xi_k)$ and $(\xi_{k+1},\ldots,\xi_d)$
from independent loop soups of intensities $\lambda$ and $1-\lambda$,
respectively.  Hence, conditioned on $\boldsymbol P$, the list deleted in
Step~\ref{alg:2} is independent of the retained list and has the law
obtained by an independent intensity-$\lambda$ loop soup.  This holds
jointly at all first-visit times.

We now prove the proposition by induction on $|R|$.  Let
\[
 S=S_0,\qquad B=R\setminus S,
\]
and let $(\zeta_1,\ldots,\zeta_N)$ be the ordered tuple constructed in
Step~\ref{alg:3} of Algorithm~\ref{alg:time-catcher}.  Conditioned on the
portions of the retained loops outside these maximal subpaths, their
conditional weight is proportional to $\prod_{j=1}^N q(\zeta_j)$.
Consequently, conditioned on their endpoints, the $(\zeta_j)_{1\le j\le N}$ are independent
normalized random walks in $B$.
The restriction of $\Lc_R^\lambda$ to loops contained in $B$ is an
intensity-$\lambda$ random-walk loop soup in $B$. If $N>0$, then
$B\subsetneq R$, so the induction hypothesis applied in $B$ shows that Algorithm~\ref{alg:time-recovery} recovers the joint
law of $(\zeta_1,\ldots,\zeta_N)$.  Replacing the recovered paths at their
recorded positions therefore recovers the retained loops at layer~$0$.

It remains to recover the lists deleted at recursive layer~$0$.  The loops of
$\Lc_R^\lambda$ intersecting $S$ are assigned by
Algorithm~\ref{alg:time-recovery} according to the same rule as in
Proposition~\ref{prop:entangled-recovery}.  By
\eqref{eq:list-splitting}, its ordered list on each first-visit
time has the conditional law of the list deleted in
Step~\ref{alg:2}, independently of the retained list. Prepending it therefore recovers the full list in
Lemma~\ref{lem:one-step}.
By Proposition~\ref{prop:entangled-recovery}, conditioned on
$\boldsymbol P$, the paths obtained at layer~$0$ have the conditional
joint law of $(\omega_1,\ldots,\omega_r)$.  Averaging over
$\boldsymbol P$ then yields the result.
\end{proof}

Let $\Lc_R^\lambda$ be an independent random-walk loop soup in $R$ of intensity $\lambda$. Apply
Algorithm~\ref{alg:time-recovery}, and let
$(\widehat\omega_1,\ldots,\widehat\omega_r)$ be its output. By Proposition~\ref{prop:path-recovery}, $(\widehat\omega_1,\ldots,\wh\omega_r)$ has the joint law of independent normalized random walks. For every
$W\subset R$, note that
\[
 \bigcap_{i=1}^r\{\supp(\widehat\omega_i)\subset W\}
 =\{S_{\lambda,\times}\subset W\}
 \cap
 \left\{\nexists\,\ell\in\Lc_R^\lambda:
 \ell\cap S_{\lambda,\times}\ne\varnothing,\
 \ell\cap(R\setminus W)\ne\varnothing\right\}.
\]
Since conditioned on $S_{\lambda,\times}$, the number of loops $\ell\in\Lc_R^\lambda$ intersecting both $S_{\lambda,\times}$ and $R\setminus W$ is a Poisson random variable with mean
$\lambda m_R(S_{\lambda,\times},R\setminus W)$, we obtain
\begin{equation}\label{eq:multipath-recovery-transform}
 \E\left[
  \one_{\{S_{\lambda,\times}\subset W\}}
  e^{-\lambda m_R(S_{\lambda,\times},R\setminus W)}
 \right]
 =\prod_{i=1}^r\frac{\Gc_W(a_i,b_i)}{\Gc_R(a_i,b_i)},
\end{equation}
where the right side is zero if an endpoint is outside $W$.

\begin{proof}[Proof of Theorem~\ref{thm:finite-positivity}, the existence when $0\le\lambda\le 1$]
Taking $r=1$ in \eqref{eq:multipath-recovery-transform}, summing over endpoints according to $\balpha$ and $\bbeta$, and writing $S_\lambda:=S_{\lambda,\times}$, we obtain
\(
 \E\left[\one_{\{S_\lambda\subset R\}}
 e^{-\lambda m_V(S_\lambda,V\setminus R)}\right]=\bh(R)
\).
Hence, the law of $S_\lambda$ gives the desired probability distribution $\bp_\lambda$.
\end{proof}

\begin{remark}
Although Algorithm~\ref{alg:time-catcher} gives
\(S_1\subset S_\lambda\subset S_0\) for each fixed \(\lambda\), the monotonicity
in \(0\le\lambda\le1\) is not apparent from this construction.  A
simultaneous monotone coupling of 
\((\bp_\lambda)_{0\le\lambda\le1}\) is obtained instead from the
linear-algebraic approach in Appendix~\ref{app:finite-positivity}; see
Corollary~\ref{cor:finite-monotone-coupling}.
\end{remark}

\subsection{The non-existence for \texorpdfstring{$\lambda>1$}{lambda > 1}}\label{subsec:finite-lambda-greater-one}

Finally we show the sharpness of the interval $\lambda\in[0,1]$, thereby finishing the proof of Theorem~\ref{thm:finite-positivity}.

\begin{proof}[Proof of Theorem~\ref{thm:finite-positivity}, non-existence for \(\lambda>1\)]
\label{rmk:general-non-existence}
Here we show the non-existence on general bipartite graphs.
For $(V,Q,\boldsymbol{\alpha},\boldsymbol{\beta})$, assume that
$G_V^{\mathrm{und}}$ is connected and that
$Q_{xy}>0\Longleftrightarrow Q_{yx}>0$ for all $x\neq y$.
Let $a,b\notin V$, and define an undirected graph $\wh G$ by
\[
V(\wh G):=V\cup\{a,b\},\qquad
E(\wh G):=E(G_V^{\mathrm{und}})
\cup\bigl\{\{a,x\}: \boldsymbol{\alpha}_x>0\bigr\}
\cup\bigl\{\{y,b\}: \boldsymbol{\beta}_y>0\bigr\}.
\]
We claim that if $\wh G$ is bipartite, then either $\wh G$ is a simple path from $a$ to $b$, in which case
\[
    \bp_\lambda(V)=1,\qquad \bp_\lambda(R)=0\ (\forall R\subsetneq V)
    \qquad \text{ for all } \lambda\geq0,
\]
or, for every $\lambda>1$, the unique solution
$\bp_\lambda=(T_\lambda^V)^{-1}\bh$ has a negative coordinate.
Thus, except in the path case, $[0,1]$ is the maximal interval on which
$\bp_\lambda$ is nonnegative.

The claim is immediate in the path case. Suppose therefore that $\wh G$
is not a simple path. Let $S\subset V$ be the vertex trace of a shortest path from $a$ to $b$ in $\wh G$. Since $S\neq V$, connectedness gives a vertex $z\in V\setminus S$ adjacent to $S$. Since $\wh G$ is bipartite and the path is shortest, $z$ has either one neighbor on the path or two neighbors separated by exactly one vertex on the path. Set $U=S\cup\{z\}$, and note that $\bp_1(U)=0$. 
Now let $C\subsetneq U$ with $Z_C>0$. Note that $\Delta_C
:=m_V(C,V\setminus C)-m_V(C,V\setminus U)>0$,
and $\bp_\lambda(A)=0$ for all $A\subsetneq C$. Comparing the row-$C$
equations of~\eqref{eq:finite-transform-intro} at $\lambda>1$ and $\lambda=1$ then gives
\[
\bp_\lambda(C)
=
\bp_1(C)\exp\!\bigl\{(\lambda-1)m_V(C,V\setminus C)\bigr\}.
\]
Substituting this into the row-$U$ equation and using $\bp_1(U)=0$, we obtain
\[
\bp_\lambda(U)e^{-\lambda m_V(U,V\setminus U)}
=
\sum_{C\subsetneq U,\ Z_C>0}
\bp_1(C)e^{-m_V(C,V\setminus U)}
\left[
1-e^{(\lambda-1)\Delta_C}
\right].
\]
Since $\Delta_C>0$, and the shortest path trace $S$ satisfies
$\bp_1(S)>0$, we obtain $\bp_\lambda(U)<0$.
Thus $\bp_\lambda$ cannot be a probability measure for any $\lambda>1$,
unless $\wh G$ is a simple path from $a$ to $b$.
In particular, a finite subgraph of $\mathbb Z^d$ admits no random-walk
loop-catcher for $\lambda>1$, unless it is a path.
\end{proof}

The above proves sharpness on finite graphs.  The non-existence for \(\lambda>1\) in the continuum case is proved in
Section~\ref{subsec:sec5-non-existence} using different methods.

\section{The subsequential limit in a Jordan domain}
\label{sec:continuum}

This section proves the existence part of Theorem~\ref{thm:main} for
\(0\leq\lambda\leq1\) by taking subsequential limits; see Proposition~\ref{prop:catcher-cluster}. The main ingredients used in this section come from~\cite{KozdronLawler2005,LawlerTrujillo2007,LawlerLimic2010}. Recall that \(D\subset\C\) is a bounded Jordan
domain and \(a\ne b\) are two marked points on $\partial D$.

\subsection{Discrete setup and Hausdorff topology}
\label{subsec:continuum-embedding}
Let \(D^\delta\subset\delta\Z^2\) be standard lattice approximations whose
union-of-squares domains converge to \(D\), with boundary parametrizations
converging uniformly.  Choose exterior vertices
\(\widehat a^\delta,\widehat b^\delta\) whose parameters under these
parametrizations tend to those of \(a,b\), and put
\[
 N_x^\delta:=\{v\in D^\delta:v\sim\widehat x^\delta\},\qquad x=a,b.
\]
The walk is killed on leaving \(D^\delta\), each step has duration
\(\delta^2/2\), and the entrance and exit weights are
\(\balpha_v^\delta=\frac14\one_{\{v\in N_a^\delta\}}\) and
\(\bbeta_v^\delta=\frac14\one_{\{v\in N_b^\delta\}}\).  The resulting
random-walk excursion includes all first and last lattice edges at
the two exterior vertices, as in the boundary kernel of
\cite{KozdronLawler2005}.  Write \(Q^\delta\) for the
killed transition matrix and \(m_{D^\delta}^\delta\) for its unrooted
random-walk loop measure.  For \(v\in D^\delta\), let
\[
 \square_\delta(v)
 :=\bigl(v+[-\delta/2,\delta/2]
                    +i[-\delta/2,\delta/2]\bigr)
\]
and, for \(S\subset D^\delta\), set
\[
\Cell_\delta(S):=\ol D\cap\bigcup_{v\in S}\square_\delta(v).
\]
The associated union-of-squares domain is
\[
\mathscr D^\delta
 :=\operatorname{int}\!\bigcup_{v\in D^\delta}
   \bigl(v+[-\delta/2,\delta/2]^2\bigr).
\]
For \(r>0\), put $D_{-r}:=\{z\in D:\dist(z,\partial D)>r\}$ and $D^{+r}:=\{z:\dist(z,D)<r\}$.
This uniform convergence implies that, for each
\(r>0\),
\begin{equation}\label{eq:domain-inclusions}
 D_{-r}\subset\mathscr D^\delta\subset D^{+r}
 \qquad\hbox{for all sufficiently small }\delta.
\end{equation}
The \emph{cell trace} is the union of a path's vertex cells.  For a
boundary-to-boundary path, adjoin connected sets \(c_a^\delta,c_b^\delta\subset\ol D\) of diameter \(o(1)\), joining
\(a\) to every cell in \(N_a^\delta\) and every cell in
\(N_b^\delta\) to \(b\).
The superscript \({}^\square\), as in \(\Gamma_\delta^\square\), denotes
this cell trace with connectors.  

Let \(\mathcal K(\ol D)\) be the compact metric space of nonempty compact
subsets of \(\ol D\), equipped with Hausdorff distance.  A
\emph{continuum} is a member of \(\mathcal K(\ol D)\) that is connected. Continua are closed under Hausdorff topology, and the relation \(\{(K,L):K\subset L\}\) is also closed in
\(\mathcal K(\ol D)^2\).
If \(S\) is nearest-neighbor connected, then \(\Cell_\delta(S)\) is at Hausdorff distance at most $\delta/\sqrt{2}$ from the linear interpolation of $S$\footnote{The linear interpolation of $S\subset D^\delta$ is the union of $S$ and all lattice edges $\ol{vw}$ with $v,w\in S$ and $v\sim w$.}. Consequently, every Hausdorff subsequential limit of such cell traces is a continuum.  

\subsection{Random-walk bridges and excursions}
\label{subsec:continuum-paths}

A normalized random-walk bridge \(x^\delta\to y^\delta\) assigns a finite
path \(\eta=(x_0,\ldots,x_n)\) probability
\[
 \frac{\prod_{j<n}Q^\delta_{x_jx_{j+1}}}
      {\Gc_{D^\delta}(x^\delta,y^\delta)},
\qquad \Gc_{D^\delta}=(I-Q^\delta)^{-1}.
\]
For distinct \(x,y\in D\), let \(p_D(t,x,y)\) be the killed heat kernel,
put \(G_D(x,y):=\int_0^\infty p_D(t,x,y)\,\dd t\), and let
\(\mathbf P_{x,y}^{D,t}\) be the killed Brownian bridge of duration \(t\).
The normalized Brownian bridge is
\begin{equation}\label{eq:green-bridge-definition}
 \Prob(\Gamma_D^{x,y}\in\cdot)
 =\frac1{G_D(x,y)}\int_0^\infty
 p_D(t,x,y)\mathbf P_{x,y}^{D,t}(\cdot)\,\dd t,
\end{equation}
and, for compact \(E\subset D\) with \(x,y\notin E\), satisfies $\Prob[\Gamma_D^{x,y}\cap E=\varnothing]=\frac{G_{D\setminus E}(x,y)}{G_D(x,y)}$.
 
\begin{lemma}
\label{lem:uniform-killed-bridge}
Let \(K_1,K_2\Subset D\) be disjoint, let
\(x^\delta=x+O(\delta),y^\delta=y+O(\delta)\), and denote by
\(\nu_\delta^{x,y}\) the unnormalized measure of the duration and
interpolated trace of a killed walk from \(x^\delta\) to \(y^\delta\).  Put
\[
 \nu^{x,y}(\dd t,\dd\gamma)
 =p_D(t,x,y)\,\dd t\,\mathbf P_{x,y}^{D,t}(\dd\gamma).
\]
On every \([\varepsilon,T]\), uniformly for
\((x,y)\in K_1\times K_2\),
\begin{equation}\label{eq:uniform-killed-window}
 \nu_\delta^{x,y}\Longrightarrow2\nu^{x,y}
\end{equation}
against bounded uniformly continuous functions of duration and Hausdorff
trace.  Both sides have uniform short- and long-duration tails:
\begin{align}\label{eq:uniform-killed-tails}
 \lim_{\varepsilon\downarrow0}\limsup_{\delta\downarrow0}
 \sup_{K_1\times K_2}\nu_\delta^{x,y}\{t<\varepsilon\}=0, \qquad \lim_{T\uparrow\infty}\limsup_{\delta\downarrow0}
 \sup_{K_1\times K_2}\nu_\delta^{x,y}\{t>T\}=0,
\end{align}
and likewise for \(2\nu^{x,y}\).  Consequently
\begin{equation}\label{eq:uniform-green-mass-from-bridges}
 \Gc_{D^\delta}(x^\delta,y^\delta)\longrightarrow2G_D(x,y)
 \quad\hbox{uniformly on }K_1\times K_2.
\end{equation}
\end{lemma}
\begin{proof}
Put \(t_n=n\delta^2/2\), and write \(p_n^\delta\) and \(p_\C\) for the
unrestricted \(n\)-step transition probability and planar heat kernel.  Uniformly for \((x,y)\in K_1\times K_2\) and allowed parities, the local central limit theorem
\cite[Theorem~2.1.3]{LawlerLimic2010} gives
\(p_n^\delta(x^\delta,y^\delta)=2\delta^2p_\C(t_n,x,y)+o(\delta^2)\).
The bridge coupling of
 \cite[Theorem~6.4]{LawlerTrujillo2007}, applied to the two diagonal
 coordinates, couples the full-plane bridges uniformly on
 \([\varepsilon,T]\).  By
 \eqref{eq:domain-inclusions}, on the event that the coupling error is at most $r$, a discrepancy between the killing indicators of the random-walk bridge and the Brownian bridge can occur only on $\{\gamma\subset D,\ \gamma\not\subset D^{-2r}\}\cup\{\gamma\subset D^{+2r},\ \gamma\not\subset D\}$.  The mass of this event tends uniformly to zero by heat-kernel continuity, boundary regularity, and Dini's theorem.  Letting \(\delta\downarrow0\), then \(r\downarrow0\), and
using the time mesh \(\delta^2\) proves \eqref{eq:uniform-killed-window}.

The two uniform tails follow from
\[
 Q^{\delta,n}(x^\delta,y^\delta)
 \leq \frac Cn e^{-c\dist(K_1,K_2)^2/(n\delta^2)},
 \qquad
 Q^{\delta,n}(x^\delta,y^\delta)
 \leq C\delta^2e^{-cn\delta^2}\quad(n\delta^2\geq1).
\]
The first is the Gaussian estimate; the second follows by splitting the
semigroup and using the discrete Dirichlet spectral gap in a containing square.  Summation proves
\eqref{eq:uniform-killed-tails};
the bounds for Brownian bridges are analogous.  Combined with~\eqref{eq:uniform-killed-window}, we obtain \eqref{eq:uniform-green-mass-from-bridges}.
\end{proof}

Lemma~\ref{lem:uniform-killed-bridge} yields the following convergence of random-walk bridges.
\begin{corollary}\label{lem:green-path-convergence}
The normalized random-walk bridge with endpoints tending to \(x\ne y\) converges in
distribution, under the Hausdorff metric, to the normalized Brownian bridge \(\Gamma_D^{x,y}\).
The convergence is uniform over two fixed disjoint compact endpoint sets and holds jointly for every fixed finite family of independent
paths, hence for their union.
\end{corollary}

Fix a conformal map \(\phi:D\to\D\) extending homeomorphically to the
closures, with
\(\phi(a)=-1,\phi(b)=1\), and put \(\zeta_a=-1,\zeta_b=1\).  For
\(x\in\{a,b\}\) and small fixed \(r\), set
\(U_x^r=\phi^{-1}(\D\cap B(\zeta_x,r))\) and let \(\Sigma_x^r\)
be its interior crosscut.  Normalized Riemann maps
\(\phi_\delta:\mathscr D^\delta\to\D\) converge uniformly on the
closures.  Let \(U_{x,\delta}^r\) be the component of \(\{v\in D^\delta:\phi_\delta(v)\in B(\zeta_x,r)\}\) containing \(N_x^\delta\), with
its holes filled, and let
\(\Sigma_{x,\delta}^r\) separate it from the rest
of \(D^\delta\) such that these crosscuts converge in Fr\'echet distance.  For each cap fix
\(o_x^r\in U_x^r\) and \(o_{x,\delta}^r\to o_x^r\).  Let
\(\psi_x^r:U_x^r\to\D\) and
\(\psi_{x,\delta}^r:\mathscr U_{x,\delta}^r\to\D\) be Riemann maps
normalized at these points, where \(\mathscr U_{x,\delta}^r\) is the
associated union-of-squares domain.  Write \(\vartheta_x^r,\vartheta_{x,\delta}^r\)
for their boundary arguments,
\(H_\cdot\) for the boundary Poisson kernel, \(h_\cdot\) for discrete harmonic
measure, and \(h_{\partial \cdot}\) for the discrete boundary kernel.  For
\(u\in\Sigma_x^r\), set
\[
 k_x^r(u):=
 \frac{H_{U_x^r}(o_x^r,u)}
 {1-\cos(\vartheta_x^r(x)-\vartheta_x^r(u))},\qquad
 \omega_x^r(du):=\frac{k_x^r(u)|du|}
 {\int_{\Sigma_x^r}k_x^r(v)|dv|}.
\]
Thus \(\omega_x^r\) is a probability measure on the crosscut.  For \(u\in\Sigma_{a,\delta}^r\), let
\(q_{a,\delta}^r(u)\) be the total weight of paths in \(U_{a,\delta}^r\) from
\(\widehat a^\delta\) to their first crosscut hit \(u\); define
\(q_{b,\delta}^r\) by reversal.
\begin{lemma}
\label{lem:edge-crosscut-law}
For \(x=a,b\), as \(\delta\downarrow0\), we have
\[\mu_{x,\delta}^r:=\frac{\sum_uq_{x,\delta}^r(u)\delta_u}{\sum_uq_{x,\delta}^r(u)}
 \Longrightarrow \omega_x^r.
\]
The entrance probability of the \(s\)-neighborhoods of the two endpoints of \(\Sigma_{x,\delta}^r\) under $\mu_{x,\delta}^r$
tends to zero first as \(\delta\downarrow0\), then as \(s\downarrow0\).
\end{lemma}
\begin{proof}
The weight \(q_{x,\delta}^r(u)=
h_{\partial U_{x,\delta}^r}(\widehat x^\delta,u)\) sums over all
choices of the first edge. The inradius of
\(\delta^{-1}\mathscr U_{x,\delta}^r\) diverges. Since
\(x\notin\overline{\Sigma_x^r}\), Radó's theorem~\cite{Rad23} implies $\inf_{u\in\Sigma_{x,\delta}^r}
 \left[1-\cos\bigl(
 \vartheta_{x,\delta}^r(\widehat x^\delta)
 -\vartheta_{x,\delta}^r(u)\bigr)\right]>c_r>0$
for all small \(\delta\); here we use the same averaging convention of $\vartheta_{x,\delta}^r(\widehat x^\delta)$ as~\cite[above Equation~(21)]{KozdronLawler2005}.
Hence \cite[Theorem~1.1]{KozdronLawler2005} gives, uniformly in \(u\),
\begin{equation}\label{eq:bound}
 q_{x,\delta}^r(u)=c_{x,\delta}^r
  \frac{h_{U_{x,\delta}^r}(o_{x,\delta}^r,u)}
  {1-\cos(\vartheta_{x,\delta}^r(\widehat x^\delta)
                    -\vartheta_{x,\delta}^r(u))}(1+o(1)),
\end{equation}
where \(c_{x,\delta}^r=(\pi/2)
h_{U_{x,\delta}^r}(o_{x,\delta}^r,\widehat x^\delta)\).
On compact subarcs away from the endpoints, Radó's theorem and
convergence of discrete harmonic measure
\cite[Proposition~7.7.2]{LawlerLimic2010} identify the limit with
\(k_x^r(u)|du|\). A fixed such subarc has positive limiting mass,
while~\eqref{eq:bound} controls the endpoint
neighborhoods of $\Sigma_x^r$ by harmonic measure. Their normalized mass tends to
zero because harmonic measure in a Jordan domain is nonatomic.
Thus both claims follow.
\end{proof}

Combining the above results gives the following convergence of random-walk excursions.
\begin{proposition}
\label{prop:excursion-convergence}
The cell trace \(\Gamma_\delta^\square\) of the normalized
random-walk excursion, including its connectors, satisfies $\Gamma_\delta^\square\Longrightarrow\gamma_{\BE}^{D;a,b}$
under the Hausdorff metric.
\end{proposition}
\begin{proof}
Fix disjoint \(U_{a,\delta}^r,U_{b,\delta}^r\).  Cutting at the first exit
\(u\) from \(U_{a,\delta}^r\) and last entrance \(v\) into
\(U_{b,\delta}^r\) gives a random-walk bridge of weight
\[
 q_{a,\delta}^r(u)\Gc_{D^\delta}(u,v)q_{b,\delta}^r(v).
\]
For \(x\in\{a,b\}\), denote the two endpoints of
\(\Sigma_{x,\delta}^r\) by \(e_{x,\delta}^{r,1},e_{x,\delta}^{r,2}\), and let
\[
 E_{x,s}^\delta:=\{u\in\Sigma_{x,\delta}^r:
 \min_{i=1,2}|u-e_{x,\delta}^{r,i}|\leq s\}.
\]
Let \(\mu_{x,\delta}^r\) be as in
Lemma~\ref{lem:edge-crosscut-law}.  Choose compact subarcs
\(I_x\Subset\Sigma_x^r\) with \(\omega_x^r(I_x)>0\).  The Green functions
between the two separated crosscuts are uniformly bounded, while
\(G_D\) has a positive minimum on \(I_a\times I_b\).  Hence the normalizing
denominator $Z_\delta^r:=\sum_{u',v'}\mu_{a,\delta}^r(u')
 \Gc_{D^\delta}(u',v')\mu_{b,\delta}^r(v')$ is bounded away from zero for small \(\delta\), and the
relative contribution of pairs with
\(u\in E_{a,s}^\delta\) or \(v\in E_{b,s}^\delta\) is at most
\[
 C_r\bigl(\mu_{a,\delta}^r(E_{a,s}^\delta)
          +\mu_{b,\delta}^r(E_{b,s}^\delta)\bigr).
\]
By Lemma~\ref{lem:edge-crosscut-law}, this bound tends to zero first as
\(\delta\downarrow0\) and then as \(s\downarrow0\).  Outside the endpoint
neighborhoods,
Lemmas~\ref{lem:edge-crosscut-law} and
\ref{lem:green-path-convergence} give a mixture of Brownian bridges
with weight \(\omega_a^r(du)G_D(u,v)\omega_b^r(dv)\); the factor \(2\)
cancels.  The Markov property and reversal identify it as the middle of
\(\gamma_{\BE}^{D;a,b}\).  Let \(\delta\downarrow0\), then \(s\downarrow0\), then
\(r\downarrow0\).  Uniform continuity of
\(\phi^{-1}\) makes the portions inside \(U_{a,\delta}^r,U_{b,\delta}^r\)
shrink.  Thus \(\delta/\sqrt2+
 \max\{\diam c_a^\delta,\diam c_b^\delta\}=o(1)\).
\end{proof}
\subsection{Stability of loop masses}
\label{subsec:continuum-loop-mass}
Write \(m_\delta:=m_{D^\delta}^\delta\).
For vertex sets \(C_V,B_V\subset D^\delta\), write
\[
 m_\delta(C_V,B_V)
 :=m_\delta\{\ell:\ell\cap C_V\ne\varnothing,
                    \ \ell\cap B_V\ne\varnothing\}.
\]
The superscript \({}^{[\varepsilon,T],\eta}\) means restriction to \[
 \{\ell:\operatorname{dur}(\ell)\in[\varepsilon,T]
          \ \hbox{and its linearly interpolated trace lies in }D_\eta\},
 \qquad
 D_\eta:=D_{-\eta}.
\]
For Brownian loops, the same notation requires the continuous trace to lie in \(D_\eta\).

For compact \(E\Subset D\) and \(s>0\), put
\[
 E^{(s)}=\{z\in D:\dist(z,E)\leq s\}.
\]
\begin{lemma}\label{lem:parallel-regular}
If \(E^{(s)}\Subset D\), then for every compact \(F\) disjoint from
\(E^{(s)}\),
\[
 m_D(F,E^{(s)})=m_D(F,(E^{(s)})^\circ).
\]
\end{lemma}
\begin{proof}
For \(z\in\partial E^{(s)}\), choose \(e\in E\) with \(|z-e|=s\).  The
ball \(B(e,s)\subset(E^{(s)})^\circ\) is tangent at \(z\), so \(z\) is
regular for that open set.  Root a loop when it first enters a
small neighborhood of \(E^{(s)}\). Such regularity and the local absolute continuity with respect to Brownian motion shows that almost every loop hitting \(E^{(s)}\) immediately enters its interior.
\end{proof}

Write \(p_n(z,w)\) for unrestricted \(n\)-step transition probabilities and
\(\mathbf P_{z,z}^n\) for the corresponding bridge law, with coordinate path \(S\). The following lemma controls the loop masses of random-walk loops and Brownian loops approaching a set without hitting it.
\begin{lemma}\label{lem:loop-approach-bound}
Fix \(0<\varepsilon<T<\infty\) and \(\eta>0\).  Let
\(C_V^\delta,B_V^\delta\subset D^\delta\), with \(C_V^\delta\) connected,
and put \(C^\delta=\Cell_\delta(C_V^\delta)\).  Suppose that $\diam C^\delta\geq d$ and $\dist(C^\delta,B_V^\delta)\geq a$
for fixed \(a,d>0\).  If
\(0<\rho<(a\wedge d)/64\) and \(\delta<\rho/8\), then
\begin{align}
 m_\delta^{[\varepsilon,T],\eta}
 \{\ell:\dist(\ell,C^\delta)\leq2\rho,\
          \ell\cap C_V^\delta=\varnothing,\
          \ell\cap B_V^\delta\ne\varnothing\}\leq
 M_{\varepsilon,T}
 \left(\frac{\rho}{a\wedge d}\right)^{1/2}.
 \label{eq:uniform-discrete-near-miss}
\end{align}
The constant is uniform in
\(\eta,\delta,C_V^\delta,B_V^\delta\).  The analogous Brownian loop estimate
holds for a connected compact \(C\) and compact \(B\) with
\(\diam C\geq d\) and \(\dist(C,B)\geq a\).
\end{lemma}
\begin{proof}
For an event \(\mathcal A\) invariant under cyclic time shifts, the definition of the random-walk loop measure and \(p_n(z,z)\leq C/n\) give
\begin{align}
 m_\delta(\mathcal A)
 &=\sum_{n\geq1}\frac1n\sum_{z\in D^\delta}
 p_n(z,z)\mathbf P_{z,z}^n(\mathcal A;\ell\subset D^\delta),
 \label{eq:rooted-loop-representation}\\
 \sum_{\varepsilon\leq n\delta^2/2\leq T}\frac1n
 \sum_{z\in D^\delta}p_n(z,z)&\leq M_\varepsilon .
 \label{eq:rooted-window-mass}
\end{align}
For a contributing loop, choose times \(i,j\) with
\(\dist(S_i,C_V^\delta)\leq3\rho\) and \(S_j\in B_V^\delta\); the
first exists because the linearly interpolated loop comes within \(2\rho\)
of \(C^\delta\) and \(\delta<\rho/8\).  Reverse time if necessary so
that the cyclic arc \(i\to j\) has length at most \(n/2\), and retain
the roots lying \(0,\ldots,\lfloor n/4\rfloor\) steps before \(i\).  For
each retained root, set
\(\tau:=\inf\{k\geq0:\dist(S_k,C_V^\delta)\leq3\rho\}\) and
\(\sigma:=\inf\{k\geq\tau:S_k\in B_V^\delta\}\).  Then
\(\tau\leq n/4\), \(\sigma\leq3n/4\), and
\(S_k\notin C_V^\delta\) for \(\tau\leq k\leq\sigma\).  Time reversal and
the choice of roots introduce a factor of at most \(2\cdot4=8\).

For even \(n\), the two-sided local central limit bounds imply
\[
 \frac{p_{n-\sigma}(S_\sigma,z)}{p_n(z,z)}\leq C
 \qquad(\sigma\leq3n/4);
\]
odd \(n\) do not contribute.  The bridge Markov property at \(\sigma\) and the strong Markov property at \(\tau\) bound this probability by a constant times the supremum below.

Put \(L=(a\wedge d)/64\), fix
\(\dist(w,C_V^\delta)\leq3\rho\), and choose \(c\in C_V^\delta\) with
\(|w-c|\leq3\rho\).  The cell definition gives
\(\diam C_V^\delta\geq d-\sqrt2\,\delta\); hence connectedness supplies
a lattice path in \(C_V^\delta\) from \(c\) to \(B(c,16L)^c\), while
\(\dist(C^\delta,B_V^\delta)\geq a\) gives
\(B_V^\delta\cap B(c,16L)=\varnothing\).  Since
\(\delta<\rho/8\), \(|w-c|<4\rho\), and
\(16L\leq(a\wedge d)/4\), a walk from \(w\) to \(B_V^\delta\) while
avoiding \(C_V^\delta\) must leave \(B(c,8L)\) without hitting this path.  Extend the path
beyond its first exit; the discrete Beurling estimate
\cite[Theorem~6.8.1]{LawlerLimic2010} gives
\[
 \sup_{\dist(w,C_V^\delta)\leq3\rho}
 \mathbf P^w\{\tau_{B_V^\delta}<\tau_{C_V^\delta}\}
 \leq C\left(\frac{4\rho}{8L}\right)^{1/2}
 \leq C(\rho/L)^{1/2}.
\]
Together with \eqref{eq:rooted-window-mass} this proves
\eqref{eq:uniform-discrete-near-miss}.

For Brownian loops, let \(X\) be the rooted path and \(p_t\) the planar heat kernel. Root with density \((2\pi t^2)^{-1}\dd z\,\dd t\), and use \(p_{t-\sigma}(X_\sigma,z)/p_t(z,z)\leq 4\) together with the Beurling estimate in \cite[proof of Lemma~5.3]{LawlerTrujillo2007}. More precisely, for a starting point \(w\) with \(\operatorname{dist}(w,C)\leq 3\rho\), choose \(c\in C\) such that \(|w-c|\leq 3\rho\), and take decreasing connected open neighborhoods \(U_k\downarrow C\). In each \(U_k\), choose a curve crossing the annulus \(B(c,8L)\setminus B(c,4\rho)\) and extend it beyond \(\partial B(c,8L)\). The Beurling estimate bounds the probability that Brownian motion started at \(w\) exits \(B(c,8L)\) while avoiding this curve by \(O((\rho/L)^{1/2})\). Since the events of avoiding \(U_k\) up to this exit increase to the event of avoiding \(C\), the same bound holds for avoiding \(C\). Finiteness of Brownian loop measure on \([\varepsilon,T]\) proves the analogous bound for Brownian loops.
\end{proof}

We also need the following estimates on the tails of loop masses.
\begin{lemma}
\label{lem:loop-mass-estimates}
These estimates hold with
constants independent of the mesh.
\begin{enumerate}
\item For every \(\varepsilon>0\),
\begin{align*}
 \lim_{\eta\downarrow0}\limsup_{\delta\downarrow0}
 m_\delta\{\operatorname{dur}(\ell)\geq\varepsilon,\
       \ell\not\subset D_\eta\}=0,\qquad
 \lim_{\eta\downarrow0}
 m_D\{\operatorname{dur}(\ell)\geq\varepsilon,\
       \ell\not\subset D_\eta\}=0.
\end{align*}
\item If two sets \(A_\delta,B_\delta\subset D^\delta\) have distance at
least \(a>0\), then, for \(\delta<a/8\),
\begin{equation}
 m_\delta\{\operatorname{dur}(\ell)<\varepsilon,\
       \ell\cap A_\delta\ne\varnothing,\ \ell\cap B_\delta\ne\varnothing\}
 \leq \frac C{a^2}e^{-ca^2/\varepsilon}.
\end{equation}
The same bound holds for Brownian loops and two sets at distance \(a\).
\item For \(T\geq1\),
\begin{equation}
 m_\delta\{\operatorname{dur}(\ell)>T\}
 +m_D\{\operatorname{dur}(\ell)>T\}
 \leq C\int_T^\infty t^{-2}e^{-ct}\,\dd t .
\end{equation}
\end{enumerate}
\end{lemma}
\begin{proof}
Fix \(T>\varepsilon\).  By the loop-soup coupling of
\cite[Theorem~1.1 and Corollary~5.4]{LawlerTrujillo2007},
\begin{align*}
 \limsup_{\delta\downarrow0}
 m_\delta\{\operatorname{dur}(\ell)\in[\varepsilon,T],\
             \ell\not\subset D_\eta\}\leq
 m_{D^{+2\eta}}\{\operatorname{dur}(\ell)\in
                 [\varepsilon/2,2T],\
                 \ell\not\subset D_{2\eta}\}.
\end{align*}
As \(\eta\downarrow0\), the events on the right decrease to loops
contained in \(\ol D\) that touch \(\partial D\), a loop-measure null
set by boundary regularity.  Continuity from above gives the Brownian loop assertion; Part~(3) proved below then lets \(T\uparrow\infty\) and proves (1).

For (2), a loop hitting both sets has diameter at least \(a\).  Splitting
the bridge into four time blocks and using the Gaussian bound gives $\mathbf P_{z,z}^n\{\diam\ell\geq a\}
 \leq Ce^{-ca^2/(n\delta^2)}$.
Substitution in \eqref{eq:rooted-loop-representation} and comparison with an
integral give
\[
 C\delta^{-2}\sum_{n\delta^2/2<\varepsilon}
 n^{-2}e^{-ca^2/(n\delta^2)}
 \leq \frac C{a^2}e^{-ca^2/\varepsilon}.
\]
The bound for Brownian loops follows by scaling.

The discrete Dirichlet spectral gap in a fixed square containing
\(D\) and
the on-diagonal bound give, for even \(n\),
\[
 \operatorname{Tr}(Q^\delta)^n
 \leq \frac C{n\delta^2}e^{-cn\delta^2};
\]
odd traces vanish.  Summing \(n^{-1}\operatorname{Tr}(Q^\delta)^n\)
proves the discrete part of (3).  The Brownian heat trace satisfies
\(\int_Dp_D(t,z,z)\,\dd z\leq Ct^{-1}e^{-ct}\); integration against
\(dt/t\) proves the remaining assertion.
\end{proof}
\begin{proposition}[Loop-mass convergence]
\label{prop:separated-loop}
Suppose \(r\) is fixed and \(C_V^\delta=\bigcup_{j=1}^rC_{j,V}^\delta\), with each \(C_{j,V}^\delta\) connected.  Put
\(C_j^\delta=\Cell_\delta(C_{j,V}^\delta)\), and assume $C_j^\delta\longrightarrow C_j$ in Hausdorff distance,
with every \(C_j\) of positive diameter, and put
\(C=\bigcup_jC_j\) (note that every $C_j$ is a continuum).
If \(B\Subset D\), \(\dist(C,B)>0\), and $m_D(C,B)=m_D(C,B^\circ)$,
then for \(B_V^\delta=\{v\in D^\delta:\dist(v,B)\leq2\delta\}\),
\begin{equation}\label{eq:finite-component-loop-limit}
 m_{D^\delta}^\delta(C_V^\delta,B_V^\delta)
 \longrightarrow m_D(C,B).
\end{equation}
\end{proposition}
\begin{proof}
Put \(a=\dist(C,B)>0\), and choose \(d>0\) such that
\(\diam C_j^\delta\geq d\) for all \(j\) and small \(\delta\).
Fix \(0<\varepsilon<T\).  For a Brownian loop \(\ell\subset D\), define $b_D(\ell):=\inf_{z\in\ell}\dist(z,\partial D)$.
Choose \(\eta>0\) outside the atoms of \(b_D\) under
\(\mu_D^{\mathrm{loop}}\) restricted to durations in \([\varepsilon/2,2T]\), and take \(\rho<(a\wedge d\wedge\eta)/64\).  Fix
\(0<\zeta<\min\{\varepsilon/2,(T-\varepsilon)/2\}\).  Apply the loop-soup
coupling of \cite[Theorem~1.1 and Corollary~5.4]{LawlerTrujillo2007} in a
fixed square containing \(D^{+1}\).  Let
\(\mathcal U_\delta,\mathcal U\) be the loops for which the coupling
supplies no partner.  Then $\E[\#\mathcal U_\delta+\#\mathcal U]=o(1)$,
and every remaining loop of duration in \([\varepsilon-\zeta,T+\zeta]\) is
paired as \(\ell^\delta\leftrightarrow\ell\).  Denote $d_{\rm H}$ to be the Hausdorff distance. For small \(\delta\), note that $d_{\rm H}(\ell^\delta,\ell)<\rho$, $|\operatorname{dur}(\ell^\delta)-\operatorname{dur}(\ell)|<\zeta$,
and $\ell\subset D_{\eta+\rho}\Longrightarrow\ell^\delta\subset D_\eta$ as well as $\ell^\delta\subset D_\eta\Longrightarrow\ell\subset D_{\eta-\rho}$.

Convergence under Hausdorff metric shows that if a pair disagrees about
hitting \(C_j\), then either
\[
 \ell^\delta\cap C_{j,V}^\delta=\varnothing
 \quad\hbox{and}\quad
 \dist(\ell^\delta,C_j^\delta)\leq2\rho,
\]
or
\[
 \ell\cap C_j=\varnothing
 \quad\hbox{and}\quad
 \dist(\ell,C_j)\leq2\rho.
\]
A union bound over \(j\) and the definition of \(B_V^\delta\) compare
hitting \(B_V^\delta\) with hitting \(B_{-\rho}\) or \(B^{+\rho}\).
Lemma~\ref{lem:loop-approach-bound} then gives, with
\(R_\rho=CM_{\varepsilon-\zeta,T+\zeta}
r(\rho/(a\wedge d))^{1/2}\),
\begin{align}
m_D^{[\varepsilon+\zeta,T-\zeta],\eta+\rho}
 (C,B_{-\rho})-R_\rho &\leq\liminf_{\delta\downarrow0}
 m_\delta^{[\varepsilon,T],\eta}(C_V^\delta,B_V^\delta)
 \leq\limsup_{\delta\downarrow0}
 m_\delta^{[\varepsilon,T],\eta}(C_V^\delta,B_V^\delta)\notag\\
&\leq
 m_D^{[\varepsilon-\zeta,T+\zeta],\eta-\rho}
 (C,B^{+\rho})+R_\rho .
\end{align}
Let \(\zeta\downarrow0\), then \(\rho\downarrow0\).  Since \(B_{-\rho}\uparrow B^\circ\),
\(B^{+\rho}\downarrow B\), and
\(m_D(C,B^\circ)=m_D(C,B)\), the corresponding loop-mass convergence holds for fixed
\(\varepsilon,T,\eta\).

The bounds in Lemma~\ref{lem:loop-mass-estimates} are uniform in \(\delta\) and permit
\(\eta\downarrow0\), \(\varepsilon\downarrow0\), and \(T\uparrow\infty\);
for the second limit use \(\dist(C_V^\delta,B_V^\delta)\geq a/2\).  The bounds for Brownian loops are the same,
proving \eqref{eq:finite-component-loop-limit}.
\end{proof}

The following lemma gives the continuity of the Brownian loop mass.
\begin{lemma}\label{lem:moving-continuum}
Let \(F_n,F\subset\ol D\) be connected compact sets,
\(F_n\to F\) in Hausdorff distance, and \(\inf_n\diam F_n>0\).  If
\(E\Subset D\), \(\dist(F,E)>0\), and $m_D(F,E)=m_D(F,E^\circ)$,
then \(m_D(F_n,E)\to m_D(F,E)\).
\end{lemma}
\begin{proof}
Put \(a=\dist(F,E)>0\) and \(d=\inf_n\diam F_n>0\).  If
\(d_{\rm H}(F_n,F)<\rho<(a\wedge d)/64\), Lemma~\ref{lem:loop-approach-bound} gives, for fixed
\(\varepsilon,T,\eta\),
\[
 \left|m_D^{[\varepsilon,T],\eta}(F_n,E)
             -m_D^{[\varepsilon,T],\eta}(F,E)\right|
 \leq 2M_{\varepsilon,T}
       \left(\frac{C\rho}{a\wedge d}\right)^{1/2}.
\]
Let \(n\to\infty\), then \(\rho\downarrow0\).  The estimates in
Lemma~\ref{lem:loop-mass-estimates} are uniform in \(n\) (use separation
\(a/2\) for short loops), so letting \(\eta\downarrow0\),
\(\varepsilon\downarrow0\), and \(T\uparrow\infty\) proves the claim.  The condition in the lemma excludes mass from loops hitting \(E\) without entering
\(E^\circ\).
\end{proof}
\subsection{Existence along subsequences}

We are now ready to take subsequential limits of random-walk loop-catchers.
\label{subsec:continuum-catcher}
\begin{proposition}\label{prop:catcher-cluster}
For every \(0\leq\lambda\leq1\), every sequence \(\delta_n\downarrow0\)
has a subsequence along which the cell trace of the random-walk loop-catcher
converges in distribution for the Hausdorff metric to a random continuum
\(K_\lambda^{D;a,b}\subset\ol D\) with
\(K_\lambda^{D;a,b}\cap\partial D=\{a,b\}\), whose law satisfies \eqref{eq:intro-main-id} for every compact obstacle away
from \(a,b\).  Consequently the existence part of
Theorem~\ref{thm:main} holds.
\end{proposition}

\begin{proof}
Let \(S_\lambda^\delta\) have the law in
Theorem~\ref{thm:finite-positivity}.  It is connected according to the construction in Algorithm~\ref{alg:time-catcher}. Put
\[
 \widetilde K_\lambda^\delta=\Cell_\delta(S_\lambda^\delta),\qquad
 K_\lambda^\delta=\widetilde K_\lambda^\delta
                    \cup c_a^\delta\cup c_b^\delta .
\]
Then $d_{\rm H}(\widetilde K_\lambda^\delta,K_\lambda^\delta)=o(1)$.
Compactness of \(\mathcal K(\ol D)\) gives a subsequence
\(K_\lambda^{\delta_n}\Longrightarrow K=:K_\lambda^{D;a,b}\), and $K$ is a continuum (see the end of Section~\ref{subsec:continuum-embedding}).
On a Skorokhod representation, both \(K_\lambda^{\delta_n}\) and
\(\widetilde K_\lambda^{\delta_n}\) converge a.s.~to \(K\).
Thus \(K\) contains \(a,b\). According to the construction of $S_\lambda^\delta$, $S_\lambda^\delta$ is stochastically dominated by the random-walk excursion $S_0^\delta$. By Proposition~\ref{prop:excursion-convergence}, the cell trace $\Gamma_\delta^\square$ of $S_0^\delta$ weakly converges to the Brownian excursion $\gamma_{\BE}^{D;a,b}$. Therefore, we have
\(K\cap\partial D=\{a,b\}\) a.s.

For \(A^\delta\subset D^\delta\),~\eqref{eq:finite-transform-intro} gives
\begin{equation}\label{eq:discrete-identity}
 \E\left[
 \one_{\{S_\lambda^\delta\cap A^\delta=\varnothing\}}
 e^{-\lambda m_{D^\delta}^\delta(S_\lambda^\delta,A^\delta)}
 \right]
 =\Prob[S_0^\delta\cap A^\delta=\varnothing].
\end{equation}
Fix \(A\Subset D\).  Choose \(s>0\) such that \(A^{(s)}\Subset D\) and $\Prob[\dist(K,A)=s]=\Prob[\dist(\gamma_{\BE}^{D;a,b},A)=s]=0$. Set
\[
 A_{\delta,s}=\{v\in D^\delta:\dist(v,A)\leq s+2\delta\}.
\]
On \(\{\dist(K,A)>s\}\), Proposition~\ref{prop:separated-loop} and
Lemma~\ref{lem:parallel-regular} give convergence of the loop mass in
\eqref{eq:discrete-identity} with \(A^\delta=A_{\delta,s}\), while on \(\{\dist(K,A)<s\}\), the indicator is eventually zero as $\delta\downarrow0$.
Proposition~\ref{prop:excursion-convergence} gives the limit on the right side of~\eqref{eq:discrete-identity}.
Bounded convergence yields
\begin{equation}
 \E\!\left[
 \one_{\{K\cap A^{(s)}=\varnothing\}}
 e^{-\lambda m_D(K,A^{(s)})}\right]
 =\Prob[\gamma_{\BE}^{D;a,b}\cap A^{(s)}=\varnothing].
\end{equation}
Choose \(s_n\downarrow0\) outside the atoms of
\(\dist(K,A)\) and \(\dist(\gamma_{\BE}^{D;a,b},A)\).  On
\(\{K\cap A=\varnothing\}\), the disjoint compact sets \(K\) and
\(A\) satisfy \(\dist(K,A)>0\) and
\(m_D(K,A^{(s_n)})\downarrow m_D(K,A)\).  Monotone convergence gives
\begin{equation}\label{eq:identity-interior}
 \E\!\left[
 \one_{\{K\cap A=\varnothing\}}e^{-\lambda m_D(K,A)}\right]
 =\Prob\left[\gamma_{\BE}^{D;a,b}\cap A=\varnothing\right].
\end{equation}

For compact \(A\subset\ol D\) avoiding \(a,b\), apply
\eqref{eq:identity-interior} to
\(A_n=A\cap\{z\in D:\dist(z,\partial D)\geq1/n\}\).
Brownian loops stay in \(D\), while
\(K\cap\partial D=\gamma_{\BE}^{D;a,b}\cap\partial D=\{a,b\}\); bounded convergence as
\(n\to\infty\) proves \eqref{eq:intro-main-id}.
\end{proof}

The non-existence part of Theorem~\ref{thm:main} is proved in
Section~\ref{subsec:sec5-non-existence}.

\section{The Green function test: proof of Theorems~\ref{thm:full-trace-unique} and~\ref{thm:hull-unique}}
\label{sec:sle-proof}

In this section we prove Theorems~\ref{thm:full-trace-unique} and~\ref{thm:hull-unique}, and hence Theorems~\ref{thm:catcher-convergence} and~\ref{thm:sle}.
The proofs are all based on the idea of Green function test.
\begin{enumerate}[label=\textup{(\arabic*)},leftmargin=2.3em]
\item For Theorem~\ref{thm:full-trace-unique},
we construct random interior probes $C$ such that for any compact subset $E$, the expectation $\E_C\!\left[
 \one_{\{C\cap E=\varnothing\}}e^{-\lambda m_D(C,E)}\right]$ yields the product of Green functions on the remaining domain $D\setminus E$; see Proposition~\ref{prop:sec5-continuum-probe}. This is based on the construction of entangled multipath LERWs on finite graphs (see the end of Section~\ref{subsec:entangled-multipath}) as well as the convergence
inputs of Section~\ref{sec:continuum}.
Testing \(K\) against these probes yields all mixed moments of
\(G_{D\setminus K}(x,y)/G_D(x,y)\).  On a countable dense set these Green function coordinates determine the whole \(K\), as any missing local continuum has
positive planar capacity and changes a Green function.

\item For Theorem~\ref{thm:hull-unique}, we use boundary probes instead to obtain products of boundary Poisson kernels; see Proposition~\ref{prop:sec5-boundary-fan}.

\item For the non-existence parts in Theorems~\ref{thm:main} and
\ref{thm:hull-unique}, take \(\kappa<2\) such that $c<-2$. Working on $(\D;-1,1)$, we construct radial probes using the generalized radial restriction sample from~\cite{qian2021generalized}, which then determines the probability that the origin lies in one side-domain. See Lemma~\ref{lem:sec5-positive-radial-probes}. Consequently, this probability equals that of a
one-sided \(\SLE_\kappa(\kappa-2)\) hull.  Since
\(\kappa-2<0\), such probability is then \(>1/2\); the opposite side is the same by symmetry, contradicting disjointness.
\end{enumerate}
We will establish the three points above in Sections~\ref{subsec:sec5-full-trace}, \ref{subsec:sec5-boundary-fans}, and \ref{subsec:sec5-non-existence}, respectively.

\subsection{Full-trace uniqueness and convergence}
\label{subsec:sec5-full-trace}

Fix a bounded Jordan domain \(D\subset\C\) and distinct
\(a,b\in\partial D\), and set
\begin{equation}\label{eq:sec5-continuum-state}
 \mathfrak C_{D;a,b}
 :=\{K\subset\ol D:K\text{ is compact and connected},\
 K\cap\partial D=\{a,b\}\},
\end{equation}
which is a Borel subspace of
\(\mathcal K(\ol D)\).
The proof uses interior multipath probes to recover every mixed moment of
a countable family of Green function coordinates, then proves that those
coordinates separate the full compact trace.

\subsubsection{Interior probes}
\label{subsec:sec5-internal-probes}

Let \(V,Q,\Gc_R,d_R,T_\lambda^V\) have the meanings fixed in
Section~\ref{sec:finite-positivity}.  Fix ordered pairs
\((x_i,y_i)\), \(1\leq i\leq r\), with \(\Gc_V(x_i,y_i)>0\).  Write
\(\mathbf x=(x_1,\ldots,x_r)\), \(\mathbf y=(y_1,\ldots,y_r)\), and set $\bh_\times(R)
 :=\prod_{i=1}^r\frac{\Gc_R(x_i,y_i)}{\Gc_V(x_i,y_i)}$,
where a factor is zero unless both endpoints belong to \(R\).

The following lemma is the direct consequence of Proposition~\ref{prop:path-recovery}; see~\eqref{eq:multipath-recovery-transform}.
\begin{lemma}[Finite interior probe]\label{lem:sec5-finite-probe}
For $0\leq\lambda\leq1$, let $S_{\lambda,\times}$ be the output of Algorithm~\ref{alg:time-catcher} applied in $V$ to the independent normalized random walks with endpoint pairs $(x_i,y_i)$.  Then
\[
 \operatorname{Law}(S_{\lambda,\times})
 =\bp_{\lambda,\times}:=(T_\lambda^V)^{-1}\bh_\times.
\]
Note that for each fixed $\lambda$, the construction in
Algorithm~\ref{alg:time-catcher} couples these sets so that
$S_{1,\times}\subset S_{\lambda,\times}\subset S_{0,\times}$.
Here $S_{0,\times}$ is the union of $r$ independent normalized random-walk
bridges from $x_i$ to $y_i$, while $S_{1,\times}$ is the entangled multipath
LERW.  Every component of $S_{\lambda,\times}$ contains a path joining one
labelled endpoint pair.  Thus it has at most $r$ components, and it is
connected if all pairs share one endpoint.
\end{lemma}

We pass the finite interior probes to the continuum using the Brownian bridges and the convergence results of Section~\ref{sec:continuum}.

\begin{proposition}[Continuum interior probe]
\label{prop:sec5-continuum-probe}
Fix finitely many pairs
\((x_i,y_i)\in D^2\), \(1\leq i\leq r\), with \(x_i\ne y_i\);
repetitions of a pair are allowed.  For every
\(0\leq\lambda\leq1\) there is a random compact set \(C\Subset D\), with
at most \(r\) nondegenerate components, such that for every compact
\(E\subset\ol D\),
\begin{equation}\label{eq:sec5-continuum-transform}
 \E_C\!\left[
 \one_{\{C\cap E=\varnothing\}}e^{-\lambda m_D(C,E)}\right]
 =\prod_{i=1}^r
   \frac{G_{D\setminus E}(x_i,y_i)}{G_D(x_i,y_i)}.
\end{equation}
Here a numerator is zero if an endpoint belongs to \(E\) or the endpoints
 lie in different components of \(D\setminus E\).
If all pairs share one endpoint, \(C\) can be chosen connected.
\end{proposition}

\begin{proof}
Apply Lemma~\ref{lem:sec5-finite-probe} in \(D^\delta\), with lattice
approximations of the marked points.  Denote the coupled vertex sets by
\(S^{\lambda,\delta}\subset S^{0,\delta}\), and put
\[
 C^{\lambda,\delta}=\Cell_\delta(S^{\lambda,\delta}),
 \qquad C^{0,\delta}=\Cell_\delta(S^{0,\delta}).
\]
Lemma~\ref{lem:green-path-convergence} gives convergence of \(C^{0,\delta}\) to the union \(C^0\) of the independent Brownian bridges in
 \eqref{eq:green-bridge-definition}; almost surely, \(C^0\Subset D\).  For each \(i\), let
 \(S_i^{\lambda,\delta}\) be the component of
 \(S^{\lambda,\delta}\) containing \(x_i^\delta\), and let
 \(C_i^{\lambda,\delta}=\Cell_\delta(S_i^{\lambda,\delta})\).  Lemma
 \ref{lem:sec5-finite-probe} implies that it also contains \(y_i^\delta\)
 and that every component occurs in this list.  These tuples are tight.  After taking a subsequence and a Skorokhod coupling, almost surely,
 \[
 (C^{\lambda,\delta},C^{0,\delta},C_1^{\lambda,\delta},\ldots,
 C_r^{\lambda,\delta})
 \longrightarrow(C,C^0,C_1,\ldots,C_r).
 \]
 Note that \(C\subset C^0\), and every \(C_i\) is compact and connected (see the end of Section~\ref{subsec:continuum-embedding}).
Each \(C_i\) contains \(x_i,y_i\), hence is
nondegenerate, and \(C=\bigcup_iC_i\) because every discrete component
contains a labelled path.  In particular \(C\) has at most \(r\)
components.  If the paths have a common endpoint, the discrete sets are
connected by Lemma~\ref{lem:sec5-finite-probe}, and so is their limit.

Fix \(E\Subset D\).  Choose \(s>0\), with \(E^{(s)}\Subset D\), outside the
atoms of both \(\dist(C,E)\) and \(\dist(C^0,E)\).  Let
\(
 E_{\delta,s}=\{v\in V(D^\delta):\dist(v,E)\leq s+2\delta\}
\).
Applying~\eqref{eq:finite-transform-intro} to the vertex set \(E_{\delta,s}\) gives
\begin{align}\label{eq:sec5-probe-discrete}
 \E\left[\one_{\{S^{\lambda,\delta}\cap E_{\delta,s}=\varnothing\}}
 e^{-\lambda m_{D^\delta}^\delta
 (S^{\lambda,\delta},E_{\delta,s})}\right] =
 \prod_i\frac{\Gc_{D^\delta\setminus E_{\delta,s}}
 (x_i^\delta,y_i^\delta)}{\Gc_{D^\delta}(x_i^\delta,y_i^\delta)}
 =\Prob[S^{0,\delta}\cap E_{\delta,s}=\varnothing].
\end{align}
On \(\{\dist(C,E)>s\}\), Proposition~\ref{prop:separated-loop}, applied to
\(S^{\lambda,\delta}=\bigcup_iS_i^{\lambda,\delta}\), and
Lemma~\ref{lem:parallel-regular} give convergence of the loop mass.  On
\(\{\dist(C,E)<s\}\), the vertex-avoidance indicator is eventually zero.
By the choice of \(s\),
\(\Prob[\dist(C,E)=s]=\Prob[\dist(C^0,E)=s]=0\).  Corollary~\ref{lem:green-path-convergence}
gives the limit of the last probability in
\eqref{eq:sec5-probe-discrete}.  Bounded convergence yields
\[
 \E_C\!\left[
 \one_{\{C\cap E^{(s)}=\varnothing\}}
 e^{-\lambda m_D(C,E^{(s)})}\right]
 =\prod_i\frac{G_{D\setminus E^{(s)}}(x_i,y_i)}{G_D(x_i,y_i)}.
\]
 Choose a sequence \(s_n\downarrow0\) such that \(E^{(s_n)}\Subset D\) and
\(\Prob[\dist(C,E)=s_n]=\Prob[\dist(C^0,E)=s_n]=0\) for every \(n\).  Apply the preceding identity with \(s=s_n\) and let \(n\to\infty\).
Continuity of loop measure on the
separated event and domain monotonicity of Green functions give \eqref{eq:sec5-continuum-transform}.

For a general compact \(E\subset\ol D\), apply the above result to
\(E_n=E\cap\{z\in D:\dist(z,\partial D)\geq1/n\}\) 
and let \(n\to\infty\).  
Bounded convergence
and continuity from above prove the desired formula.
\end{proof}

\subsubsection{Green function coordinates and full-trace uniqueness}
\label{subsec:sec5-green-coordinates}

Let \(\mathcal Z=(\Q+i\Q)\cap D\).  For
\(K\in\mathfrak C_{D;a,b}\) and distinct \(x,y\in\mathcal Z\), define
\begin{equation}\label{eq:sec5-green-coordinate}
 X_{x,y}(K)
 :=\frac{G_{D\setminus K}(x,y)}{G_D(x,y)}\in[0,1],
\end{equation}
with numerator zero if an endpoint belongs to \(K\) or the two points lie
 in different components of \(D\setminus K\). Fix \(0\leq\lambda\leq1\) and probability laws
\(\mathsf P,\mathsf Q\) on \(\mathfrak C_{D;a,b}\) such that, for every
compact \(A\subset \ol D\),
\begin{align}\label{eq:sec5-transform-equality}
 \int \one_{\{K\cap A=\varnothing\}}
 e^{-\lambda m_D(K,A)}\,\mathsf P(\dd K)=\int \one_{\{K\cap A=\varnothing\}}
 e^{-\lambda m_D(K,A)}\,\mathsf Q(\dd K).
\end{align}
Then we have
\begin{lemma}\label{lem:sec5-green-moments}
For the two probability laws
\(\mathsf P,\mathsf Q\) as above, they have the same expectation of every finite monomial in the coordinates
\eqref{eq:sec5-green-coordinate}.
\end{lemma}

\begin{proof}
For a finite list \((x_i,y_i)\), take the interior probe \(C\) of
Proposition~\ref{prop:sec5-continuum-probe}.  For deterministic \(K\), that
proposition and symmetry of Brownian loop measure give $ \E_C\!\left[
 \one_{\{C\cap K=\varnothing\}}e^{-\lambda m_D(C,K)}\right]
 =\prod_iX_{x_i,y_i}(K)$,
and the integrand
is jointly Borel.
Since \(C\Subset D\) almost surely, we may apply \eqref{eq:sec5-transform-equality} conditionally on \(C\). Fubini's theorem gives
\begin{align*}
 \int\prod_iX_{x_i,y_i}(K)\,\mathsf P(\dd K)
 &=\E_C\int
   \one_{\{K\cap C=\varnothing\}}e^{-\lambda m_D(K,C)}
   \,\mathsf P(\dd K)\\
 &=\E_C\int
   \one_{\{K\cap C=\varnothing\}}e^{-\lambda m_D(K,C)}
   \,\mathsf Q(\dd K)
 =\int\prod_iX_{x_i,y_i}(K)\,\mathsf Q(\dd K).
\end{align*}
Repeated endpoint pairs give
arbitrary powers.
\end{proof}

The following lemma shows that Green function coordinates are injective.
\begin{lemma}
\label{lem:sec5-green-injective}
The map
\[
 \Phi:\mathfrak C_{D;a,b}\longrightarrow
 [0,1]^{\{(x,y)\in\mathcal Z^2:x\ne y\}},
 \qquad \Phi(K)=(X_{x,y}(K))_{x\ne y},
\]
is Borel and injective.
\end{lemma}

\begin{proof}
The Brownian bridge identity $X_{x,y}(K)=\Prob[\Gamma_D^{x,y}\cap K=\varnothing]$ and openness of disjointness in the Hausdorff topology prove
Borel measurability.

Suppose \(K_1\ne K_2\).  After interchanging them, choose
\(z\in(K_2\setminus K_1)\cap D\) and concentric disks $\ol B_0\subset B\Subset D\setminus K_1$ such that $z\in B_0$.
Let \(J\) be the closure in \(K_2\) of the component of the relatively
open set \(K_2\cap B_0\) containing \(z\).  The boundary-bumping theorem
\cite[Theorem~5.4]{Nadler1992} shows that \(J\) meets \(\partial B_0\);
thus it is a nondegenerate continuum $J\subset K_2\cap\ol B_0$
which contains \(z\).  Such a planar continuum
has positive logarithmic capacity (or positive planar Sobolev capacity, equivalently).

Assume for contradiction that \(\Phi(K_1)=\Phi(K_2)\), and put
\(U_j=D\setminus K_j\).  If \(K_2\cap B\) has nonempty interior, a
rational point in that interior and a second rational point of \(B\) give
zero Green function in \(U_2\) and positive Green function in \(U_1\), a
contradiction.  We may thus assume that \(B\setminus K_2\) is dense in \(B\).

For any two rational points of \(B\setminus K_2\), their Green
function in \(U_1\) is positive.  Equality of the coordinates forces their Green function in \(U_2\) to be positive, so all components of \(U_2\) meeting
\(B\) are one component, say \(V\).  Density and continuity away from the
diagonal then give
\begin{equation}\label{eq:sec5-local-green-equality}
 G_{U_1}(u,v)=G_V(u,v),
 \qquad u,v\in B\setminus K_2,\quad u\ne v.
\end{equation}
Moreover \(J\subset\partial V\).  By Kellogg's theorem, the irregular
boundary points of \(V\) have capacity zero. Choose a regular
\(\xi\in J\). Then with \(v\in V\cap B\) fixed and \(u\to\xi\) inside
\(V\cap B\),
\[
 G_V(u,v)\longrightarrow0,\qquad
 G_{U_1}(u,v)\longrightarrow G_{U_1}(\xi,v)>0,
\]
contradicting \eqref{eq:sec5-local-green-equality}. Hence \(K_1=K_2\).
\end{proof}

\begin{proof}[Proof of Theorem~\ref{thm:full-trace-unique}]
Lemma~\ref{lem:sec5-green-moments}, polynomial density on finite cubes, and
a monotone-class argument give
\(\Phi_*\mathsf P=\Phi_*\mathsf Q\).
The state space \eqref{eq:sec5-continuum-state} is standard Borel, while
Lemma~\ref{lem:sec5-green-injective} makes
\(\Phi\) a Borel injection.
Lusin--Suslin theorem supplies a Borel inverse, so \(\mathsf P=\mathsf Q\).
\end{proof}

\begin{proof}[Proof of Theorem~\ref{thm:catcher-convergence}]
Every subsequential limit of \((K_\lambda^\delta)\) belongs to
\(\mathfrak C_{D;a,b}\) and satisfies \eqref{eq:intro-main-id} by
Proposition~\ref{prop:catcher-cluster}. Theorem~\ref{thm:full-trace-unique}
makes this limit unique, proving convergence of the full sequence.
\end{proof}

\subsection{Hull uniqueness and SLE identification}
\label{subsec:sec5-boundary-fans}

By conformal invariance, without loss of generality, we take
\((D;a,b)=(\D;-1,1)\), write
\[
 \mathfrak C_{-1,1}
 :=\{K\subset\ol\D:K\text{ is compact and connected},\
 K\cap\partial\D=\{-1,1\}\}.
\]
Recall that a \emph{two-sided filling hull} is an \(H\in\mathfrak C_{-1,1}\) satisfying
\begin{equation}\label{eq:sec5-hull-recovery}
 H=\ol{\D\setminus(D_H^+\cup D_H^-)}.
\end{equation}
A \emph{one-sided filling hull} is a connected compact \(H\subset\ol\D\)
such that \(H\cap\partial\D=\{z\in\partial\D:\Im z\leq0\}\) and \(U_H:=\D\setminus H\) is simply
connected.

\subsubsection{Boundary probes}

We say that an open set \(U\subset\D\) \emph{agrees with \(\D\) in a collar of
\(\xi\in\partial\D\)} if \(U\cap N=\D\cap N\) for some neighborhood
\(N\) of \(\xi\) in $\ol\D$.  Denote the connected component of $U$ containing such collar by $U_\xi$. Let \(P_U(\xi,z)\) be the boundary Poisson kernel of $U_\xi$, and set it to zero when \(z\) is not in $U_\xi$.

\begin{proposition}[Boundary probes]
\label{prop:sec5-boundary-fan}
Fix $0\leq\lambda\leq1$.

\smallskip
\noindent\textup{(i) One-sided boundary probe.}
Fix $\xi\in\partial\D$ and $z_1,\ldots,z_r\in\D$, where $r\geq1$.
There is a random connected compact set
$C=C^\lambda_{\xi;z_1,\ldots,z_r}\subset\overline\D$ such that
$C\cap\partial\D=\{\xi\}$ and $z_1,\ldots,z_r\in C$, and, for every
compact $E\subset\overline\D$ with $\xi\notin E$,
\begin{equation}\label{eq:sec5-boundary-fan-transform}
 \mathbb E_C\!\left[
 \mathbf 1_{\{C\cap E=\varnothing\}}e^{-\lambda m_\D(C,E)}
 \right]
 =\prod_{j=1}^r
 \frac{P_{\D\setminus E}(\xi,z_j)}{P_\D(\xi,z_j)}.
 \end{equation}

\smallskip
\noindent\textup{(ii) Two-sided boundary probe.}
Fix distinct $\xi_+,\xi_-\in\partial\D$ and
$z_1,\ldots,z_r,w_1,\ldots,w_s\in\D$, where $r,s\geq1$.
There is a random compact set $C\subset\overline\D$ with at most two
connected components such that $C\cap\partial\D=\{\xi_+,\xi_-\}$, one
component contains $\xi_+,z_1,\ldots,z_r$, and one, possibly the same,
contains $\xi_-,w_1,\ldots,w_s$. Moreover, for every compact
$E\subset\overline\D$ with $E\cap\{\xi_+,\xi_-\}=\varnothing$,
\[
 \mathbb E_C\!\left[
 \mathbf 1_{\{C\cap E=\varnothing\}}e^{-\lambda m_\D(C,E)}
 \right]
 =\prod_{j=1}^r
 \frac{P_{\D\setminus E}(\xi_+,z_j)}{P_\D(\xi_+,z_j)}
 \prod_{k=1}^s
 \frac{P_{\D\setminus E}(\xi_-,w_k)}{P_\D(\xi_-,w_k)}.
\]
\end{proposition}

\begin{proof}
We first focus on one-sided probes.
Let \(a_n=(1-1/n)\xi\), omitting any coincident value and taking \(n\)
sufficiently large, and apply
Proposition~\ref{prop:sec5-continuum-probe} to \((a_n,z_j)\).  Its
construction couples \(C_n\subset C_n^0\), where \(C_n^0\) is the union of independent Brownian bridges $\Gamma_\D^{a_n,z_j}$ from \(a_n\) to \(z_j\).  Along a subsequence,
\[
 (C_n,C_n^0)\Longrightarrow(C,C^0),\qquad C\subset C^0.
\]
Take a Skorokhod realization with almost-sure convergence. Closedness of
connectedness and convergence of the marked points give
\(\xi,z_1,\ldots,z_r\in C\).

Let \(B\) be a closed boundary arc not
containing \(\xi\), and \(N_\eta(B):=\{z\in\ol\D:\dist(z,B)\leq\eta\}\) be its closed \(\eta\)-collar.  For each
\(j\),
\begin{equation}\label{eq:poisson-limit-ratio}
 \lim_{n\to\infty}
 \Prob[\Gamma_\D^{a_n,z_j}\cap N_\eta(B)=\varnothing]
 =
 \frac{P_{\D\setminus N_\eta(B)}(\xi,z_j)}{P_\D(\xi,z_j)}.
\end{equation}
If \(C^0\) hits \(B\), then \(C_n^0\) hits \(N_\eta(B)\) eventually; hence $\Prob[C^0\cap B\ne\varnothing]
 \leq\liminf_n\Prob[C_n^0\cap N_\eta(B)\ne\varnothing]$.
As \(\eta\downarrow0\), the ratio on the right side of~\eqref{eq:poisson-limit-ratio} increases to one.  A finite union bound
and a countable cover of
\(\partial\D\setminus\{\xi\}\) prove
\(C^0\cap\partial\D=\{\xi\}\), and hence the same for \(C\).

For \(E\Subset\D\), choose \(s_\ell\downarrow0\) such that $\Prob[\dist(C,E)=s_\ell]=\Prob[\dist(C^0,E)=s_\ell]=0$ for all $\ell\geq1$,
and fix \(s=s_\ell\).  By
Lemma~\ref{lem:parallel-regular}, \(E^{(s)}\) is loop regular.  Since
\(\diam C_n\geq |a_n-z_1|\), Lemma~\ref{lem:moving-continuum} and the
two
zero-probability conditions above imply $\one_{\{C_n\cap E^{(s)}=\varnothing\}}
 e^{-\lambda m_\D(C_n,E^{(s)})}
 \longrightarrow
 \one_{\{C\cap E^{(s)}=\varnothing\}}
 e^{-\lambda m_\D(C,E^{(s)})}$
a.s.  Proposition
\ref{prop:sec5-continuum-probe} and the boundary Green function asymptotic give
\(
 \frac{G_{\D\setminus E^{(s)}}(a_n,z_j)}{G_\D(a_n,z_j)}
 \longrightarrow
 \frac{P_{\D\setminus E^{(s)}}(\xi,z_j)}{P_\D(\xi,z_j)}
\).
Bounded convergence proves \eqref{eq:sec5-boundary-fan-transform} for
\(E^{(s_\ell)}\).  Letting \(\ell\to\infty\) proves it for \(E\).

For general \(E\subset\ol\D\) avoiding \(\xi\), use
\(E_q=E\cap\{\dist(z,\partial\D)\geq1/q\}\).  The boundary probe has no boundary
contact other than \(\xi\), and Brownian loops stay in \(\D\). Continuity
from above completes \textup{(i)}.

For \textup{(ii)}, choose
$a_n,b_n\in\D$ with $a_n\to\xi_+$ and $b_n\to\xi_-$, apply
Proposition~\ref{prop:sec5-continuum-probe} to the combined list of pairs $(a_n,z_j)$ and
$(b_n,w_k)$, and pass to a subsequential limit as above. The result then follows.
\end{proof}

\subsubsection{Poisson kernel coordinates and hull uniqueness}
\label{subsec:sec5-hull-moments}

For a simply connected \(U\subset\D\) agreeing with \(\D\) near
\(\xi\in\partial\D\), define
\begin{equation}
 X_z^\xi(U)=
 \begin{cases}
 P_U(\xi,z)/P_\D(\xi,z),&z\in U,\\
 0,&z\notin U.
 \end{cases}
\end{equation}
Domain monotonicity gives \(0\leq X_z^\xi\leq1\). The following lemma shows that Poisson kernel coordinates determine a side domain.

\begin{lemma}
\label{lem:sec5-poisson-injective}
Fix a countable dense set \(\mathcal Z\subset\D\).  On the standard Borel
space of simply connected \(U\subset\D\) which agree with \(\D\) in a
collar of \(\xi\), the map $\Phi_\xi(U)=(X_z^\xi(U))_{z\in\mathcal Z}$ is injective.
\end{lemma}

\begin{proof}
Set \(C_\xi(z):=i(1+\bar\xi z)/(1-\bar\xi z)\), and normalize the conformal map \(F_j:U_j\to\mathbb H\) by \(F_j(z)=C_\xi(z)+O(1-\bar\xi z)\) as \(z\to\xi\). Then \(P_{U_j}(\xi,z)=C\operatorname{Im}F_j(z)\) for \(j=1,2\). Hence equality of the Poisson kernels on a collar $U_\xi$ of $\xi$ implies that \(F_1-F_2\) is a real constant on $U_\xi$, and the normalization forces this constant to be zero. Thus \(F_1=F_2\) on $U_\xi$, and the identity theorem gives \(U_1=U_2\).
\end{proof}

\begin{proof}[Proof of Theorem~\ref{thm:hull-unique}]
Fix \(0<\lambda\leq1\), and let \(\mathsf P,\mathsf Q\) be laws on
two-sided filling hulls satisfying the condition in Theorem~\ref{thm:hull-unique}.  
Choose \(\xi_+=i\) and \(\xi_-=-i\).
By Proposition~\ref{prop:sec5-boundary-fan} and taking the hulls of boundary probes, we have
\begin{equation}
 \int
 \prod_jX_{z_j}^{\xi_+}(D_H^+)
 \prod_kX_{w_k}^{\xi_-}(D_H^-)\,\mathsf P(\dd H)
 =
 \int
 \prod_jX_{z_j}^{\xi_+}(D_H^+)
 \prod_kX_{w_k}^{\xi_-}(D_H^-)\,\mathsf Q(\dd H).
\end{equation}
Moment determinacy and a monotone-class argument identify
the full coordinate law.
Lemma~\ref{lem:sec5-poisson-injective} then identifies \((D_H^+,D_H^-)\) and hence the filling hull~\eqref{eq:sec5-hull-recovery}.
The one-sided case follows in parallel.
\end{proof}

\begin{proof}[Proof of Theorem~\ref{thm:sle}]
Let \(H_\lambda\) be the two-sided filling hull of \(K_\lambda\), obtained by
retaining its upper and lower side domains and filling all other
complementary components.  
Then \(H_\lambda\) satisfies \eqref{eq:hull-id}.
On the other hand, by \cite[Proposition~6.2]{qian2021generalized} (with $\alpha=1$ there), we know that the two-sided filling hull with $\eta_+,\eta_-$ being its upper and lower boundaries also satisfies \eqref{eq:hull-id}. By Theorem~\ref{thm:hull-unique}, we conclude.
\end{proof}

\subsection{Non-existence of Brownian loop-catcher for \texorpdfstring{$\lambda>1$}{lambda > 1}}
\label{subsec:sec5-non-existence}

We prove the non-existence part of Theorem~\ref{thm:main} for
\(\lambda>1\), with \((D;a,b)=(\D;-1,1)\), using radial probes. Let
\(\kappa\in(0,2)\) be related to \(\lambda\) by \eqref{eq:c-kappa}.
We use the generalized radial restriction samples of~\cite{qian2021generalized}.
A \emph{radial filling hull from \(\xi\) to \(z\)} is a connected compact
\(C\subset\ol\D\) with \(z\in C\),
\(C\cap\partial\D=\{\xi\}\), and \(\D\setminus C\) simply connected.
Put \(\beta_*=(4-\kappa)^2/(2\kappa)\).

\begin{lemma}[Radial probes,~\cite{qian2021generalized}]
\label{lem:sec5-positive-radial-probes}
Fix \(z\in\D\), \(\xi\in\partial\D\setminus\{-1,1\}\), and
\(\beta>\beta_*\).  There is a radial probe \(C_\beta\), a random radial
filling hull from \(\xi\) to \(z\), such that for every compact \(A\subset\ol\D\) with \(z,\xi\notin A\) and
\(\D\setminus A\) simply connected, we have
\begin{equation}\label{eq:sec5-positive-radial-transform}
 \E\!\left[
 \one_{\{C_\beta\cap A=\varnothing\}}
 e^{-\lambda m_\D(C_\beta,A)}\right]=|f_A'(\xi)|^\beta,
\end{equation}
where \(f_A:\D\setminus A\to\D\) fixes \(z,\xi\).
\end{lemma}
\begin{proof}
Take the generalized radial restriction sample with
\((\alpha,\beta)=(0,\beta)\) in~\cite[Theorem~1.6]{qian2021generalized} on $(z,\xi)$. Since $\beta>\beta_*$, such sample exists.
\end{proof}

The collection of radial probes above is then sufficient to determine the probability that $z$ is outside the one-sided filling hull.
\begin{lemma}
\label{lem:sec5-mellin}
Let \(\mathsf P,\mathsf Q\) be laws on one-sided filling hulls such that
\(\mathsf P[\xi\in H]=\mathsf Q[\xi\in H]=0\).  Suppose
\[
 \int\one_{\{H\cap C=\varnothing\}}e^{-\lambda m_\D(H,C)}
 \,\mathsf P(\dd H)
 =\int\one_{\{H\cap C=\varnothing\}}e^{-\lambda m_\D(H,C)}
 \,\mathsf Q(\dd H)
\]
for every radial filling hull \(C\) from \(\xi\) to an interior point.
Then, for every \(z\in\D\),
\[
 \mathsf P[z\in\D\setminus H]
 =\mathsf Q[z\in\D\setminus H].
\]
\end{lemma}

\begin{proof}
Fix \(z\) and average the assumed equality over the independent radial probe
\(C_\beta\) of Lemma~\ref{lem:sec5-positive-radial-probes}.  For a deterministic
\(H\), put \(U=\D\setminus H\).  If \(z\notin U\), the radial probe intersects
\(H\).  If \(z\in U\), let \(f_U:U\to\D\) fix \(z,\xi\), and put
\(X_H=f_U'(\xi)\); set \(X_H=0\) otherwise.
Equation~\eqref{eq:sec5-positive-radial-transform} gives
\begin{equation}\label{eq:sec5-fixed-H-mellin}
 \E_{C_\beta}\!\left[
 \one_{\{C_\beta\cap H=\varnothing\}}
 e^{-\lambda m_\D(C_\beta,H)}\right]
 =\one_{\{z\in U\}}X_H^\beta.
\end{equation}
The derivative is finite and positive because the common free arc supplies
a boundary collar, depending on \(H\), around \(\xi\) on which \(U\)
agrees with \(\D\);
the left side of \eqref{eq:sec5-fixed-H-mellin} also shows \(X_H\leq1\).
Fubini's theorem therefore gives, for every \(\beta>\beta_*\),
\[
 \int\one_{\{z\in\D\setminus H\}}X_H^\beta\,\mathsf P(\dd H)
 =
 \int\one_{\{z\in\D\setminus H\}}X_H^\beta\,\mathsf Q(\dd H).
\]
These are the Mellin transforms of two finite measures on \((0,1]\).
Under \(x=e^{-t}\), and after tilting by one fixed
\(\beta_0>\beta_*\), they become Laplace transforms on \([0,\infty)\)
which agree at every positive argument.  Uniqueness of the Laplace
transform identifies the tilted measures.  Multiplication by
\(e^{\beta_0t}\) recovers the original finite measures and hence their
total masses.
\end{proof}

\begin{proof}[Proof of Theorem~\ref{thm:main}]
For $0<\lambda\le 1$, the existence follows from Proposition~\ref{prop:catcher-cluster}, and Theorem~\ref{thm:full-trace-unique} gives the uniqueness. For the non-existence of $\lambda>1$, suppose that a random \(K\in\mathfrak C_{-1,1}\) satisfies
\eqref{eq:intro-main-id}.  Set \(z=0\), \(\xi=-i\), and
\(H^+=\operatorname{Fill}^+(K)\).  For every radial filling hull \(C\) from
\(\xi\) to \(z\),
\[
\E\!\left[
 \one_{\{H^+\cap C=\varnothing\}}e^{-\lambda m_\D(H^+,C)}\right]
=\Prob\left[\gamma_{\BE}^{\D;-1,1}\cap C=\varnothing\right].
\]

Let \(H_{\rm SLE}^+\) be the one-sided filling hull of a chordal
\(\SLE_\kappa(\kappa-2)\) curve \(\eta\) with force point at \(a^+\). By~\cite{dubedat-rho} (see also~\cite[Proposition 6.2]{qian2021generalized}), \[
\E\!\left[
 \one_{\{H_{\rm SLE}^+\cap C=\varnothing\}}
 e^{-\lambda m_\D(H_{\rm SLE}^+,C)}\right]=\Prob\left[\gamma_{\BE}^{\D;-1,1}\cap C=\varnothing\right].
\]
Lemma~\ref{lem:sec5-mellin}
therefore identifies
\(\Prob[0\in D_K^-]\) with
\(\Prob[0\in\D\setminus H_{\rm SLE}^+]\). Since
\(\kappa-2\in(-2,0)\), this probability is larger than \(1/2\).  Similarly considering $\operatorname{Fill}^{-}(K):=\ol{\D\setminus D_K^+}$ yields \(\Prob[0\in D_K^+]>1/2\), contradicting the disjointness of
\(D_K^+\) and \(D_K^-\).
\end{proof}

\section{The one-point density of Brownian loop-catcher}\label{sec:one-arm}

In this section we prove Theorem~\ref{thm:one-arm}. Fix $0<\lambda<1$, and without loss of generality we consider $(D;a,b)=(\D,-1,1)$. For any compact subset $A\subset\ol\D$ and $a,b\in\partial\D$, set $F_\D^{a,b}(A):=\Prob\left[\gamma_{\BE}^{\D;a,b}\cap A=\varnothing\right]$. We also write \(K=K_\lambda^{\D;-1,1}\) for simplicity. Recall that \(c=-2\lambda\), with \(\kappa<\frac{8}{3}\) related to \(c\) by \eqref{eq:c-kappa}.

\subsection{The \(\SLE_\kappa\) loop and the Brownian loop-catcher}

For a simply connected domain \(\Omega\ni z\), let \(\mu_\Omega^{c,z}\) denote the $\SLE_\kappa$ loop
measure on \(\Omega\) restricted to surround \(z\) defined in~\cite{Zhan2021}.  Its
conformal restriction property~\cite[Theorem 5.1]{Zhan2021} gives
\begin{equation}\label{eq:one-arm-MKS-restriction}
 \frac{d\mu_{\Omega'}^{c,z}}{d\mu_\Omega^{c,z}}(\ell)
 =\one_{\{\ell\subset {\Omega'}\}}
   \exp\{-\lambda m_\Omega(\ell,\Omega\setminus {\Omega'})\},
 \qquad {\Omega'}\subset \Omega,
\end{equation}
where \(m_\Omega(A,B)\) is the Brownian loop mass of loops in \(\Omega\)
intersecting both \(A\) and \(B\). We also let $\mu_\C^{c,z}$ be the whole-plane $\SLE_\kappa$ loop measure restricted to surround $z$.

Let \(U=U_z(K)\) be the component of \(\D\setminus K\) containing \(z\).
\begin{lemma}[Fubini identity]\label{lem:one-arm-fubini}
For every nonnegative Borel functional \(g\) on loops surrounding \(z\),
\begin{equation*}
 \E\!\left[\int g(\ell)\,\mu_U^{c,z}(d\ell)\right]
 =\int g(\ell)F_\D^{-1,1}(\ell)\,\mu_\D^{c,z}(d\ell).
\end{equation*}
\end{lemma}

\begin{proof}
Note that a loop surrounding
\(z\) lies in \(U\) if and only if it avoids \(K\), and \(m_\D(\ell,\D\setminus U)=m_\D(\ell,K)\). The result follows by applying
\eqref{eq:one-arm-MKS-restriction}, Fubini's theorem, and
\eqref{eq:intro-main-id} with \(A=\ell\).
\end{proof}

For a simply connected \(\Omega\ni z\),
let \(f_\Omega:\Omega\to\D\) satisfy \(f_\Omega(z)=0\), \(f_\Omega'(z)>0\), and let
\(\operatorname{crad}_\Omega(z)=|f_\Omega'(z)|^{-1}\) be its conformal radius at \(z\). Define
\begin{equation*}
 T=\log\frac{\operatorname{crad}_\D(z)}
                  {\operatorname{crad}_U(z)}>0.
\end{equation*}
Let
\(\operatorname{Fill}_z(\ell)\) be the closure of the bounded component
of \(\C\setminus\ell\) containing \(z\), and put
\[
 \mathsf H_\Omega(\ell)=f_\Omega\bigl(\operatorname{Fill}_z(\ell)\bigr).
\]
For a compact \(H\ni0\) with simply connected complement, let \(t(H)\) be its logarithmic capacity such that the normalized map \(\C\setminus H\to
\{|w|>t(H)\}\) has derivative one at infinity.  Define
\begin{equation*}
 L_\Omega(\ell)=-\log t(\mathsf H_\Omega(\ell)).
\end{equation*}
This does not depend on the choice of rotation in \(f_\Omega\).
Write \(L=L_\D\).

We use the following uniform estimates from
\cite[Remark~2.8, Lemmas~2.10 and~2.12, and Proposition~3.3]{lawler-field-reversed}.  Note that every bounded compact \(H\ni0\) with simply connected complement satisfies
\begin{equation}\label{eq:one-arm-FL-radius}
 t(H)\le\rad H:=\sup\{|z|:z\in H\}\le4t(H).
\end{equation}
If in addition \(H\subset\D\), \(t(H)=e^{-L}\le1/8\), and \(h\) is a
univalent function on \(\D\) with \(h(0)=0\), then
\begin{align}
 \Lambda^*(H,\partial\D)&=-\log L+O(e^{-L}),
 \label{eq:one-arm-FL-loop}\\
 \operatorname{Exc}_{\D\setminus H}(w,H)
   &=L^{-1}[1+O(e^{-L})],\qquad |w|=1,
 \label{eq:one-arm-FL-exc}\\
 \log t(h(H))&=\log\bigl(|h'(0)|t(H)\bigr)+O(e^{-L}).
 \label{eq:one-arm-FL-map}
\end{align}
All errors are uniform over \(H\), and the last one is uniform over \(h\)
after normalizing by \(h'(0)\).  Here \(\Lambda^*\) is the renormalized Brownian loop mass~\cite{lawler-field-reversed}, and
\(\operatorname{Exc}_{\D\setminus H}(w,H)\) is the total mass of Brownian excursions
from \(w\) to \(H\) in $\D\setminus H$.

\begin{lemma}[Capacity estimates]\label{lem:one-arm-shells}
There are \(A_\lambda,q\in(0,\infty)\) and \(L_0<\infty\) such that, for
\(L\ge L_0\),
\begin{align}
 \mu_\D^{c,z}\{L(\ell)\in dL\}
   &=A_\lambda L^\lambda[1+O(e^{-L})]\,dL,
 \label{eq:one-arm-shell}\\
 \int_{\{L(\ell)\in dL\}}
   \bigl(1-F_\D^{-1,1}(\ell)\bigr)\,\mu_\D^{c,z}(d\ell)
   &=A_\lambda q L^{\lambda-1}[1+O(e^{-L})]\,dL.
 \label{eq:one-arm-defect-shell}
\end{align}
Both measures have finite mass on \(0\le L\le L_0\).
\end{lemma}

\begin{proof}
Push forward by \(f_\D\), and write
\(\widehat a=f_\D(-1)\), \(\widehat b=f_\D(1)\).  Conformal invariance gives
\((f_\D)_*\mu_\D^{c,z}=\mu_\D^{c,0}\) and
\(F_\D^{-1,1}(\ell)=F_\D^{\widehat a,\widehat b}(f_\D(\ell))\).
Relabel \(f_\D(\ell)\) as \(\ell\) and put
\(H=\operatorname{Fill}_0(\ell)\).  The pushforward of the whole-plane
$\SLE_\kappa$ loop measure \(\mu_\C^{c,0}\) by \(L=-\log t\) is a nonzero,
translation-invariant measure on \(\R\).  
The local finiteness of \(\mu_\C^{c,0}\) thus implies
this measure is \(A_\lambda\,dL\) for some
\(A_\lambda\in(0,\infty)\).  For
\(L\ge L_0\), \eqref{eq:one-arm-FL-radius} puts the loop inside \(\D\),
while the conformal restriction property~\cite[Theorem 5.1]{Zhan2021} and \eqref{eq:one-arm-FL-loop} give
\[
 \frac{d\mu_\D^{c,0}}{d\mu_\C^{c,0}}(\ell)
 =\exp\{-\lambda\Lambda^*(H,\partial\D)\}
 =L^\lambda[1+O(e^{-L})].
\]
Indeed, a continuous loop which meets both
\(H\) and \(\partial\D\) must cross \(\ell\), so replacing
\(\ell\) by its filling does not change the \(\Lambda^*\)-event.
This proves \eqref{eq:one-arm-shell}.

Let \(H_\D(\widehat a,\widehat b)\) be the boundary excursion Poisson
kernel and \(P_\D(\cdot,\widehat b)\) the interior Poisson kernel, with
normalizations compatible with the excursion measure above.  The first-hit
decomposition at \(H\), followed by \eqref{eq:one-arm-FL-exc}, gives
\[
 H_\D(\widehat a,\widehat b)-H_{\D\setminus H}(\widehat a,\widehat b)
 =\int \operatorname{Exc}_{\D\setminus H}(\widehat a,d\xi)
        P_\D(\xi,\widehat b).
\]
By \eqref{eq:one-arm-FL-radius}, uniformly for \(\xi\in H\),
\(P_\D(\xi,\widehat b)=P_\D(0,\widehat b)[1+O(e^{-L})]\).  Therefore
\begin{equation}
 1-F_\D^{\widehat a,\widehat b}(\ell)=\frac qL[1+O(e^{-L})],
 \qquad
 q:=\frac{P_\D(0,\widehat b)}
          {H_\D(\widehat a,\widehat b)}>0,
\end{equation}
and pulling back by \(f_\D\) proves
\eqref{eq:one-arm-defect-shell}.
Equivalently, \(q\) is the limit of \(L(1-F_\D^{-1,1})\) along
the loops \(f_\D^{-1}(\partial B(0,e^{-L}))\), and is therefore independent of the
kernel normalization.
Here an excursion from \(\partial\D\) hits \(H\) exactly
when it crosses \(\ell\), so the left side applies to
\(\ell\) itself.

It remains to prove finiteness of these measures on
\(0\leq L\leq L_0\).  If
\(t(H)\ge t_0>0\), \eqref{eq:one-arm-FL-radius} allows us to choose in
\(H\) an arc \(A\) from zero to radius
\(r_*=\min\{t_0/2,1/32\}\).  The same global estimate gives
\(r_*/4\le t(A)\le r_*\), so \eqref{eq:one-arm-FL-loop} applies to \(A\).
Its estimate, followed by monotonicity of \(\Lambda^*\), gives a uniform
lower bound for \(\Lambda^*(H,\partial\D)\).  Thus the density $\exp\{-\lambda\Lambda^*(H,\partial\D)\}$
is uniformly bounded.  The hulls lie in a fixed disk and
have diameters bounded below, so the local-finiteness of
\(\mu_\C^{c,0}\) gives finite mass.
\end{proof}

For \(p>0\), set
\begin{align*}
 Z(p)=\int e^{-pL(\ell)}\,\mu_\D^{c,z}(d\ell),\qquad
 J(p)=\int \bigl(1-F_\D^{-1,1}(\ell)\bigr)e^{-pL(\ell)}
                 \,\mu_\D^{c,z}(d\ell).
\end{align*}
Subtracting the leading densities in
Lemma~\ref{lem:one-arm-shells} leaves measures with an exponential moment.
Hence there are functions \(z_0,j_0\), holomorphic in a neighborhood of
zero, such that
\begin{align}
 Z(p)&=A_\lambda\Gamma(1+\lambda)p^{-1-\lambda}+z_0(p),
 \label{eq:one-arm-Z-expansion}\\
 J(p)&=A_\lambda q\Gamma(\lambda)p^{-\lambda}+j_0(p).
 \label{eq:one-arm-J-expansion}
\end{align}

\subsection{The compensated Laplace transform}

Throughout this section, \(C\) denotes a positive constant whose value may
change from line to line. We also write \(f\sim g\) if \(f/g\to1\). The following lemma compares the capacities between different domains.
\begin{lemma}\label{lem:one-arm-quasi-additivity}
There is a constant \(C<\infty\) such that, for every realization of
\(K\) and every loop \(\ell\) surrounding \(z\) in \(U\),
\begin{equation}\label{eq:one-arm-quasi-additivity}
 L(\ell)=T+L_U(\ell)+\delta(K,\ell),\qquad
 |\delta(K,\ell)|\le
 \begin{cases}
  Ce^{-L_U(\ell)},&L_U(\ell)\ge L_0,\\
  T+C,&L_U(\ell)<L_0.
 \end{cases}
\end{equation}
Here $L_0$ is as in Lemma~\ref{lem:one-arm-shells}.
\end{lemma}

\begin{proof}
The map
\[
 h=f_\D\circ f_U^{-1}:\D\longrightarrow\D
\]
is univalent, fixes zero, and satisfies \(|h'(0)|=e^{-T}\).  For
\(L_U(\ell)\ge L_0\), formula \eqref{eq:one-arm-FL-map}, applied to
\(\mathsf H_U(\ell)\), gives the first bound in
\eqref{eq:one-arm-quasi-additivity}.  If \(L_U(\ell)<L_0\), then
\(\mathsf H_U(\ell)\) contains a point a fixed distance from zero.  The
lower growth bound for \(h/h'(0)\), together with
\eqref{eq:one-arm-FL-radius}, gives
\(t(h(\mathsf H_U(\ell)))\ge C^{-1}e^{-T}\); together with
\(L,L_U,T\ge0\), this gives the second bound.
\end{proof}

For \(p>0\), write \(\phi(p)=\E e^{-pT}\) and set
\begin{equation}\label{eq:one-arm-error}
 \mathcal E(p)=\E\!\int
 \left[e^{-pL(\ell)}-e^{-p(T+L_U(\ell))}\right]
 \mu_U^{c,z}(d\ell).
\end{equation}

\begin{lemma}
\label{lem:one-arm-bootstrap}
The integral in \eqref{eq:one-arm-error} is absolutely integrable and, for
\(0<p\le1\),
\begin{align}
 Z(p)-J(p)&=Z(p)\phi(p)+\mathcal E(p),
 \label{eq:one-arm-Laplace-identity}\\
 |\mathcal E(p)|&\le Cp+C\E[\min\{p(T+C),1\}].
 \label{eq:one-arm-error-bootstrap}
\end{align}
Moreover,
\begin{equation}\label{eq:one-arm-mean}
 \E T=\frac q\lambda=:m<\infty,
 \qquad \mathcal E(p)=O(p).
\end{equation}
\end{lemma}

\begin{proof}
Apply Lemma~\ref{lem:one-arm-fubini} with
\(g(\ell)=e^{-pL(\ell)}\), and add and subtract the second integrand in
\eqref{eq:one-arm-error}.  Conformal invariance gives
\(\int e^{-pL_U}\,d\mu_U^{c,z}=Z(p)\), and hence
\eqref{eq:one-arm-Laplace-identity}.  On \(L_U\ge L_0\),
Lemma~\ref{lem:one-arm-quasi-additivity}, the elementary Lipschitz bound for
\(e^{-x}\), and Lemma~\ref{lem:one-arm-shells} give an \(O(p)\) contribution.
Note that the $\mu_U^{c,z}$-mass of \(\{L_U<L_0\}\) is
finite, hence the contribution of the $L_U<L_0$ part to $\mathcal{E}(p)$ is at most the second term in
\eqref{eq:one-arm-error-bootstrap}.  This proves both absolute
integrability and~\eqref{eq:one-arm-error-bootstrap}.

In particular \(\mathcal E(p)=O(1)\).  Rearranging gives
\(Z(p)(1-\phi(p))=J(p)+\mathcal E(p)\).  Dividing by \(pZ(p)\) and using
\eqref{eq:one-arm-Z-expansion}--\eqref{eq:one-arm-J-expansion} yields
\[
 \lim_{p\downarrow0}\frac{1-\phi(p)}p
 =\frac{q\Gamma(\lambda)}{\Gamma(1+\lambda)}=\frac q\lambda.
\]
Since \((1-e^{-pT})/p\uparrow T\) as \(p\downarrow0\), monotone
convergence proves the first assertion in \eqref{eq:one-arm-mean}.  The
right side of \eqref{eq:one-arm-error-bootstrap} is then \(O(p)\), proving
the second.  \end{proof}

Set
\begin{equation}\label{eq:one-arm-Delta-R}
 \Delta(p)=mpZ(p)-J(p),\qquad
 R(p)=\E[e^{-pT}-1+pT]=\phi(p)-1+mp.
\end{equation}
Then we have the following asymptotics of \(\Delta(p)\) and \(R(p)\).
\begin{proposition}
\label{prop:one-arm-positive-defect}
There is a constant \(\Delta_0>0\) such that
\begin{equation}\label{eq:one-arm-R-asymptotic}
 \Delta(p)=\Delta_0+O(p),\qquad
 R(p)\sim\frac{\Delta_0}{A_\lambda\Gamma(1+\lambda)}p^{1+\lambda}.
\end{equation}
\end{proposition}

\begin{proof}
The leading singularities in \(mpZ(p)\) and \(J(p)\) cancel by
\eqref{eq:one-arm-Z-expansion}--\eqref{eq:one-arm-J-expansion}; hence
\begin{equation}\label{eq:one-arm-Delta0}
 \Delta(p)=\Delta_0+O(p),\qquad \Delta_0=-j_0(0)\in\R.
\end{equation}
Equations \eqref{eq:one-arm-Laplace-identity} and
\eqref{eq:one-arm-Delta-R} give the exact identity
\begin{equation}\label{eq:one-arm-defect-identity}
 Z(p)R(p)=\Delta(p)-\mathcal E(p).
\end{equation}
Since \(R(p)\ge0\) and \(\mathcal E(p)=O(p)\), letting \(p\downarrow0\)
shows that \(\Delta_0\ge0\).

To prove $\Delta_0>0$, choose \(0<r<R_0<\infty\) with
\(\Prob[r\le T\le R_0]>0\), which is possible because
\(0<T<\infty\) a.s.  For \(pR_0\le1\), since $e^{-x}-1+x\ge Cx^2$ on $x\in[0,1]$ for some $C>0$, we have 
\[
 R(p)\ge C p^2r^2\Prob[r\le T\le R_0]\ge C'p^2.
\]
Together with \(Z(p)\asymp p^{-1-\lambda}\), this implies
\(Z(p)R(p)\ge Cp^{1-\lambda}\). If \(\Delta_0=0\),
\eqref{eq:one-arm-Delta0}, Lemma~\ref{lem:one-arm-bootstrap}, and
\eqref{eq:one-arm-defect-identity} would instead give
\(Z(p)R(p)=O(p)=o(p^{1-\lambda})\), a contradiction.  Thus
\(\Delta_0>0\), and the second assertion in
\eqref{eq:one-arm-R-asymptotic} follows from
\eqref{eq:one-arm-defect-identity}, Lemma~\ref{lem:one-arm-bootstrap}, and
\eqref{eq:one-arm-Z-expansion}.
\end{proof}

\subsection{Tauberian inversion}
Based on Proposition~\ref{prop:one-arm-positive-defect}, we have the following tail estimate of $T$.
\begin{proposition}
\label{prop:one-arm-T-tail}
As \(t\to\infty\),
\begin{equation}\label{eq:one-arm-T-tail}
 \Prob[T>t]\sim
 \frac{\Delta_0\sin(\pi\lambda)}{\pi A_\lambda}t^{-1-\lambda}.
\end{equation}
\end{proposition}

\begin{proof}
Since \(\E T<\infty\), integration by parts gives
\[
 R(p)=p\int_0^\infty(1-e^{-pt})\Prob[T>t]\,dt.
\]
Put \(Q(x)=\int_x^\infty\Prob[T>t]\,dt\).  Fubini's theorem and
Proposition~\ref{prop:one-arm-positive-defect} give
\[
 p^2\int_0^\infty e^{-px}Q(x)\,dx=R(p)
 \sim B_0p^{1+\lambda},\qquad
 B_0:=\frac{\Delta_0}{A_\lambda\Gamma(1+\lambda)}.
\]
Since \(Q\) and the tail are monotone, Karamata's Tauberian theorem and the
monotone density theorem \cite[Section~1.7]{reg-var-book} yield
\[
 Q(x)\sim\frac{B_0}{\Gamma(1-\lambda)}x^{-\lambda},
 \qquad
 \Prob[T>t]\sim\frac{B_0}{-\Gamma(-\lambda)}t^{-1-\lambda}.
\]
The reflection identity for the gamma function gives
\eqref{eq:one-arm-T-tail}.
\end{proof}

\begin{proof}[Proof of Theorem~\ref{thm:one-arm}]
Return to general \((D;a,b)\) and write
\(K=K_\lambda^{D;a,b}\).  The variable \(T\) is conformally invariant, so
Proposition~\ref{prop:one-arm-T-tail} remains valid.  Let
\(r_D=\operatorname{crad}_D(z)\) and \(d=\dist(z,K)\).  For
\(\varepsilon<\dist(z,\partial D)/2\), Koebe's theorem gives
\begin{equation*}
 \left\{T>\log\frac{r_D}{\varepsilon}\right\}
 \subset\{d<\varepsilon\}
 \subset\left\{T>\log\frac{r_D}{4\varepsilon}\right\}.
\end{equation*}
The result then follows from Proposition~\ref{prop:one-arm-T-tail}, with \(C_{D,a,b,z,\lambda}
=\Delta_0\sin(\pi\lambda)/(\pi A_\lambda)\in(0,\infty)\). Note that the previous proof in fact implies that \(\Delta_0/A_\lambda\) remains bounded as \(\lambda \uparrow 1\). Indeed, the $O(e^{-L})$ error in~\eqref{eq:one-arm-defect-shell} can be chosen uniformly
for $\lambda$ sufficiently close to $1$. Since
$\Delta_0=-j_0(0)$, 
expressing $j_0(0)$ in~\eqref{eq:one-arm-J-expansion} using the
definition of $J$ and splitting the integral at $L_0$
give
\[
0<\frac{\Delta_0}{A_\lambda}
\leq
q\frac{L_0^\lambda}{\lambda}
+Cq\int_{L_0}^{\infty}L^{\lambda-1}e^{-L}\,dL.
\]
The right side remains bounded as $\lambda \uparrow 1$.
Consequently, \(C_{D,a,b,z,\lambda}\to 0\) as \(\lambda \uparrow 1\).
\end{proof}

\appendix
\section{Linear-algebraic approach to the random-walk loop-catcher}\label{app:finite-positivity}

In this appendix we give an alternative linear-algebraic proof of Theorem~\ref{thm:finite-positivity}.
Recall the notation at the beginning of Section~\ref{sec:finite-positivity}.  
We start by introducing the logarithmic derivative
\(\Ac_\lambda^E=(T_\lambda^E)^{-1}\partial_\lambda T_\lambda^E\) for each
subgraph \(E\subset V\).

\subsection{The logarithmic derivative}
\label{subsec:log-derivative}

For every finite \(E\subset V\), define the triangular transform
\begin{align}\label{eq:finite-transform-local}
 (T_\lambda^Ef)(W)
 :=\sum_{C\subset E}f(C)\one_{\{C\subset W\}}
       e^{-\lambda m_E(C,E\setminus W)}=\sum_{C\subset W}f(C)
   \left(\frac{d_Ed_{W\setminus C}}
   {d_{E\setminus C}d_W}\right)^\lambda,
 \qquad W\subset E.
\end{align}
Set
\begin{equation}\label{eq:A-generator}
        \Ac_\lambda^E:=(T_\lambda^E)^{-1}\partial_\lambda T_\lambda^E .
\end{equation}
Rows and columns are indexed by subsets of \(E\), and
\(\Ac_\lambda^E(U,S)\) denotes the row \(U\), column \(S\) entry.  Note that they vanish unless the column index is contained in the row index, namely,
\begin{equation}\label{eq:A-incidence}
        \Ac_\lambda^E(U,S)=0\qquad\text{unless }S\subset U .
\end{equation}
Also, the top row of \(T_\lambda^E\) is identically one, so with \(e_E\)
the row-\(E\) basis vector and \(\one\) the all-ones column, we have \(\one^{\mathsf T}=e_E^{\mathsf T}T_\lambda^E\), while
\(e_E^{\mathsf T}\partial_\lambda T_\lambda^E=0\).  Hence
\begin{equation}\label{eq:column-sum-zero}
 \one^{\mathsf T}\Ac_\lambda^E
 =e_E^{\mathsf T}\partial_\lambda T_\lambda^E=0,
 \qquad
 \sum_{U\subset E}\Ac_\lambda^E(U,S)=0 .
\end{equation}
Thus \(\Ac_\lambda^E\) is a pure-birth Markov generator \emph{once} its off-diagonal
entries are nonnegative.

Fix \(S\subset E\), put \(B=E\setminus S\), and we will use the following notation in this section.
\[
\begin{array}{c|l}
E & \text{current finite subgraph} \\
S & \text{current column, or the trace already present} \\
B=E\setminus S & \text{vertices still available to be added} \\
W\subset B & \text{allowed remainder in a test row }S\cup W \\
C\subset B & \text{new vertices added to }S
\end{array}
\]
If \(S=\varnothing\), the \(S\)-column of \(\Ac_\lambda^E\) is zero.  Assume
from now on that \(S\ne\varnothing\).  For \(W\subset B\), define
\begin{equation}\label{eq:HS}
        \bH_S(W):=
        \Loop_Q\{\ell:\ell\cap S\ne\varnothing,
        \ \varnothing\ne\ell\setminus S\subset W\}
        =\log\frac{d_Sd_W}{d_{S\cup W}} .
\end{equation}
Define
\begin{equation}\label{eq:rho}
        \brho_S^\lambda:=(T_\lambda^B)^{-1}\bH_S .
\end{equation}
Since \(\bH_S(\varnothing)=0\), also \(\brho_S^\lambda(\varnothing)=0\).

The following lemma computes the column elements of $\Ac_\lambda^E$.

\begin{lemma}\label{lem:column-formula}
For every nonempty \(C\subset B\),
\[
        \Ac_\lambda^E(S\cup C,S)=\brho_S^\lambda(C),
        \qquad
        \Ac_\lambda^E(S,S)=-\bH_S(B).
\]
\end{lemma}

\begin{proof}
Define a vector \(a_S\) on \(2^E\) by
\[
        a_S(S)=-\bH_S(B),
        \qquad
        a_S(S\cup C)=\brho_S^\lambda(C)\quad(C\ne\varnothing),
\]
and \(a_S(U)=0\) for all other \(U\).  We show that
\(T_\lambda^Ea_S=\partial_\lambda T_\lambda^E e_S\), where \(e_S\) is the
basis vector at \(S\).  This will identify \(a_S\) as the \(S\)-column of
\(\Ac_\lambda^E\).

Rows not containing \(S\) give zero on both sides.  Let the row be
\(R=S\cup W\), where \(W\subset B\).  For \(C\subset W\), since
\(
 e^{-\lambda[m_E(S\cup C,B\setminus W)-m_E(S,B\setminus W)]}
 =e^{-\lambda m_B(C,B\setminus W)}
 =\left(\frac{d_Bd_{W\setminus C}}
 {d_{B\setminus C}d_W}\right)^\lambda
\)
, we have
\[
\begin{aligned}
        (T_\lambda^Ea_S)(S\cup W)
        &=e^{-\lambda m_E(S,B\setminus W)}
        \left[-\bH_S(B)+\sum_{C\subset W}
        \brho_S^\lambda(C)e^{-\lambda m_B(C,B\setminus W)}\right] \\
        &=e^{-\lambda m_E(S,B\setminus W)}
          \big[-\bH_S(B)+\bH_S(W)\big].
\end{aligned}
\]
By~\eqref{eq:HS}, the bracket equals
\[
        \log\frac{d_Ed_W}{d_Bd_{S\cup W}}
        =-m_E(S,B\setminus W)
        =\partial_\lambda\log e^{-\lambda m_E(S,B\setminus W)}.
\]
Thus \((T_\lambda^Ea_S)(S\cup W)=
\partial_\lambda e^{-\lambda m_E(S,B\setminus W)}\), as required.
\end{proof}

\medskip
\noindent
\textbf{Roadmap.}
The rest of the proof of Theorem~\ref{thm:finite-positivity} has three steps.
\begin{enumerate}[label=\textup{(\arabic*)},leftmargin=2.2em]
\item At the endpoint $\lambda=1$, prove \(\brho_S^1\ge0\) by combining a
Schur-complement expansion of \(\bH_S\) with the nonnegative expansion in Lemma~\ref{lem:multipath-positive}, based on the entangled multipath LERW.

\item Use induction over subgraphs and the equation
\(\frac{\dd}{\dd\theta}\brho_S^{1-\theta}
=\Ac_{1-\theta}^{E\setminus S}\brho_S^{1-\theta}\) to transport
\(\brho_S^1\ge0\) to $\brho_S^\lambda\ge0$ for all \(0\le\lambda\le1\).  This makes
\(\Ac_\lambda^E\) a pure-birth generator.

\item Transport the endpoint law \(\bp_1\) by
\(\frac{\dd}{\dd\theta}\bp_{1-\theta}=\Ac_{1-\theta}^V\bp_{1-\theta}\).  The
generator preserves nonnegativity and total mass, and the transformed
moments remain equal to \(\bh\).
\end{enumerate}
Equivalently,
\[
        \text{entangled multipath LERW}
        \Longrightarrow \brho_S^1\ge0
        \Longrightarrow \Ac_\lambda^E\text{ is a pure-birth generator}
        \Longrightarrow \bp_\lambda\ge0 .
\]

\subsection{The endpoint \texorpdfstring{\(\lambda=1\)}{lambda=1}}
\label{subsec:endpoint-deconvolution}

We need the following input of $\lambda=1$ from the entangled multipath LERW.

\begin{lemma}\label{lem:multipath-positive}
Let \(W\subset V\) and let \((a_i,b_i)\in W\times W\), \(1\le i\le r\),
with \(r\ge1\).  Write \(\mathbf a=(a_i)_{i=1}^r\) and
\(\mathbf b=(b_i)_{i=1}^r\).  For \(R\subset W\) and \(a,b\in W\), set \(\Gc_R(a,b)=0\) if either endpoint lies outside \(R\). There are unique finite coefficients
\(L_{\mathbf a,\mathbf b}(C)\ge0\), \(C\subset W\), such that, for every
\(R\subset W\),
\begin{equation}\label{eq:multipath-positive}
        \one_{\{a_i,b_i\in R\ \forall i\}}
        \left[\prod_{i=1}^r\Gc_R(a_i,b_i)\right]d_R
        =\sum_{C\subset R}L_{\mathbf a,\mathbf b}(C)d_{R\setminus C} .
\end{equation}
Moreover \(L_{\mathbf a,\mathbf b}(C)=0\) unless \(C\) contains all endpoints
and every connected component of \(C\) contains at least one endpoint.
\end{lemma}

\begin{proof}
Recall the terminology in Section~\ref{subsec:entangled-multipath}. For an ordered tuple
\(\mathbf P=(P_1,\ldots,P_r)\) of legal records in \(W\) with
\(P_i:a_i\to b_i\), define $C_0(\mathbf P)=\varnothing$ and $ C_i(\mathbf P)
 =C_{i-1}(\mathbf P)\cup\supp(P_i)$ for $1\le i\le r$.
For \(C\subset W\), set
\begin{equation}\label{eq:multipath-coefficient}
 L_{\mathbf a,\mathbf b}(C)
 :=
 \sum_{\substack{
 \mathbf P=(P_1,\ldots,P_r)\\
 \Phi_{C_{i-1}(\mathbf P)}(P_i)=P_i,\ 1\le i\le r\\
 C_r(\mathbf P)=C}}
 \prod_{i=1}^r q(P_i)\ge0,
\end{equation}
which is $0$ unless \(C\) contains all endpoints
and every connected component of \(C\) contains at least one endpoint.
Now fix \(R\subset W\) containing all the endpoints.  Repeatedly applying
\eqref{eq:one-step} in Lemma~\ref{lem:one-step} with \(A=C_{i-1}(\mathbf P)\) at the \(i\)-th step gives
\[
\begin{aligned}
 d_R\prod_{i=1}^r\Gc_R(a_i,b_i)
 =
 \sum_{\substack{
 P_i:a_i\to b_i\text{ legal in }R\\
 \Phi_{C_{i-1}(\mathbf P)}(P_i)=P_i,\ 1\le i\le r}}
 \left(\prod_{i=1}^r q(P_i)\right)
 d_{R\setminus C_r(\mathbf P)}  =
 \sum_{C\subset R}
 L_{\mathbf a,\mathbf b}(C)d_{R\setminus C}.
\end{aligned}
\]
This proves~\eqref{eq:multipath-positive}. The uniqueness of $L_{\mathbf a,\mathbf b}(C)$ follows from the triangularity of~\eqref{eq:multipath-positive}.
\end{proof}

Based on Lemma~\ref{lem:multipath-positive}, we obtain the following nonnegativitivy of $\brho_S^1$ on the endpoint $\lambda=1$.
\begin{proposition}\label{prop:endpoint-deconvolution}
Fix \(E\subset V\) and nonempty \(S\subset E\), and put
\(B=E\setminus S\).  There are numbers
\(\bnu_S(C)\ge0\), \(\varnothing\ne C\subset B\), such that, for every
\(W\subset B\),
\begin{equation}\label{eq:nu-expansion}
        \bH_S(W)d_W=
        \sum_{\varnothing\ne C\subset W}\bnu_S(C)d_{W\setminus C} .
\end{equation}
Consequently
\begin{equation}\label{eq:rho-one-positive}
        \brho_S^1(C)=\frac{d_{B\setminus C}}{d_B}\bnu_S(C)\ge0,
        \qquad C\ne\varnothing .
\end{equation}
Moreover \(\bnu_S(C)\), and hence \(\brho_S^1(C)\), vanishes unless \(C\) is
\(S\)-attached.
\end{proposition}

\begin{proof}
Let
\[
        \Gc_W:=(I-Q_W)^{-1},
        \qquad
        K:=Q_{B,S}(I-Q_S)^{-1}Q_{S,B} .
\]
For \(W\subset B\), write \(K|_W\) for the \(W\times W\) submatrix of
\(K\).
The Schur complement then gives
\[
        d_{S\cup W}=d_S\det(I-Q_W-K|_W).
\]
Hence
\begin{equation}\label{eq:H-schur}
        \bH_S(W)=
        -\log\det(I-\Gc_W(K|_W))
        =\sum_{r\ge1}\frac1r\operatorname{Tr}\big((\Gc_W(K|_W))^r\big).
\end{equation}
Indeed, path decomposition by the number of excursions from \(W\) through
\(S\) gives
\[
 \sum_{n=0}^N(\Gc_W(K|_W))^n\Gc_W
 \leq\bigl((I-Q_{S\cup W})^{-1}\bigr)_{W,W}.
\]
Since \(\Gc_W\ge I\), the partial sums are bounded entrywise. For the
nonnegative matrix \(\Gc_W(K|_W)\), the Perron--Frobenius theorem then forces
its spectral radius to be strictly smaller than one, which justifies the
nonnegative logarithmic series.

Expanding the trace, with \(x_{r+1}=x_1\), gives
\[
        \operatorname{Tr}\big((\Gc_W(K|_W))^r\big)d_W
        =\sum_{\mathbf x,\mathbf y\in W^r}
        \left[\prod_{i=1}^rK_{y_i x_{i+1}}\right]
        \left[\prod_{i=1}^r\Gc_W(x_i,y_i)\right]d_W .
\]
Apply Lemma~\ref{lem:multipath-positive}, for the matrix \(Q_B\) on \(B\), to
fixed endpoint pairs \((x_i,y_i)\).  After grouping by \(C\subset W\), the
final sentence in that lemma lets us extend the endpoint sums
from \(W^r\) to \(B^r\), so the resulting coefficients are independent of \(W\).  The partial sums \(1\le r\le N\) have coefficients
\[
        \bnu_S^{(N)}(C):=
        \sum_{r=1}^N\frac1r\sum_{\mathbf x,\mathbf y\in B^r}
        \left[\prod_{i=1}^rK_{y_i x_{i+1}}\right]
        L_{\mathbf x,\mathbf y}(C)\ge0 .
\]
At \(W=B\),
\[
        \sum_{\varnothing\ne C\subset B}\bnu_S^{(N)}(C)d_{B\setminus C}
        \le \bH_S(B)d_B<\infty .
\]
Since all \(d_{B\setminus C}\) are positive, each coefficient is monotone and
bounded in \(N\).  Letting \(N\to\infty\) gives \eqref{eq:nu-expansion} by
monotone convergence.

The support statement follows from the support part of
Lemma~\ref{lem:multipath-positive}.  If a summand is positive, every component
of \(C\) contains some endpoint \(x_i\) or \(y_i\).  The factor
\(K_{y_i x_{i+1}}>0\) means that \(y_i\) and \(x_{i+1}\) are connected to
\(S\) through a positive-weight path entering \(S\), moving inside \(S\), and
leaving \(S\).  Hence every endpoint is attached to \(S\), and so every
component of \(C\) is attached to \(S\).

Finally,
\[
 \one_{\{C\subset W\}}e^{-m_B(C,B\setminus W)}
 =\one_{\{C\subset W\}}
  \frac{d_Bd_{W\setminus C}}{d_{B\setminus C}d_W}.
\]
Dividing \eqref{eq:nu-expansion} by \(d_W\) shows that \(T_1^B\brho_S^1=
\bH_S\) when \(\brho_S^1\) is defined by \eqref{eq:rho-one-positive} and
\(\brho_S^1(\varnothing)=0\).  Since \(T_1^B\) is invertible, this is the
endpoint expansion.
\end{proof}

\subsection{Induction over subgraphs and transport in \texorpdfstring{\(\lambda\)}{lambda}}
\label{subsec:transport}

We use the following standard lemma on homogeneous linear systems.
\begin{lemma}\label{lem:cone-invariance}
Let \(L_t\), \(0\le t\le T\), be a continuous family of matrices indexed by a
finite set \(\Omega\).  Assume
\[
        L_t(\omega,\eta)\ge0\quad(\omega\ne\eta),
        \qquad
        \sum_{\omega\in\Omega}L_t(\omega,\eta)=0 .
\]
Then \(\dot x_t=L_tx_t\) preserves the nonnegative cone and the total mass.
If \(\Omega_0\subset\Omega\) satisfies
\(L_t(\omega,\eta)=0\) for \(\eta\in\Omega_0\), \(\omega\notin\Omega_0\),
then support in \(\Omega_0\) is preserved.
\end{lemma}

\begin{proof}
For any $\varepsilon>0$, consider
\(x_0^\varepsilon:=x_0+\varepsilon\one\) and
\(\dot x_t^\varepsilon=L_tx_t^\varepsilon+\varepsilon\one\). At the first contact of $(x_t^\varepsilon)$ with the boundary of the nonnegative cone, the vanishing
coordinate has derivative at least $\varepsilon$, so it cannot become negative. By Gronwall's inequality, $x_t^\varepsilon\to x_t$ uniformly on compact time intervals as $\varepsilon\downarrow0$, so $x_t\ge0$. $\sum_{\omega\in\Omega}L_t(\omega,\eta)=0$ gives
\(\frac{\dd}{\dd t}\sum_\omega x_t(\omega)=0\). The last statement follows since the restriction of $(x_t)$ on $\Omega\setminus\Omega_0$ satisfies a
homogeneous linear system with zero initial data.
\end{proof}

For \(E\subset V\), let \(G_E^{\rm und}\) be the underlying undirected
support graph: \(x\sim y\) if \(Q_{xy}>0\) or \(Q_{yx}>0\). All connected
components below are taken in the relevant induced subgraph of this graph.
If \(S\subset E\) and \(C\subset E\setminus S\), call \(C\)
\emph{\(S\)-attached} if every connected component of
\(G_E^{\rm und}[S\cup C]\) contains a vertex of \(S\).

\begin{proposition}[Nonnegativity and support of the generator]\label{prop:positive-generator}
For every finite \(E\subset V\) and every \(0\le\lambda\le1\), the matrix
\(\Ac_\lambda^E\) satisfies:
\begin{enumerate}[label=\textup{(\roman*)},leftmargin=2.2em]
\item \(\Ac_\lambda^E(U,S)=0\) unless \(S\subset U\);
\item each column sum is zero;
\item if \(S\subsetneq U\), then \(\Ac_\lambda^E(U,S)\ge0\);
\item if \(S\ne\varnothing\), \(S\subsetneq U\), and some connected component
of \(G_E^{\rm und}[U]\) is disjoint from \(S\), then
\(\Ac_\lambda^E(U,S)=0\).
\end{enumerate}
\end{proposition}

\begin{proof}
We argue by induction on \(|E|\).  The case \(E=\varnothing\) is trivial.
Assume the statement for all proper subsets of \(E\), and fix a column
\(S\subset E\).  If \(S=\varnothing\), then
\(m_E(\varnothing,E\setminus W)=0\) for all \(W\), so the empty-set
column of \(T_\lambda^E\) is identically one and this column of
\(\Ac_\lambda^E\) is zero.  We now take \(S\ne\varnothing\) and put
\(B=E\setminus S\).

By Lemma~\ref{lem:column-formula}, the off-diagonal entries in this column
are the coordinates of
\(\brho_S^\lambda=(T_\lambda^B)^{-1}\bH_S\).  Put
\(\br_\theta:=\brho_S^{1-\theta}\).  Since
\(T_{1-\theta}^B\br_\theta=\bH_S\), differentiation gives
\begin{equation}\label{eq:rho-ode}
        \frac{\dd}{\dd\theta}\br_\theta
        =\Ac_{1-\theta}^B\br_\theta,
        \qquad 0\le\theta\le1 .
\end{equation}
Proposition~\ref{prop:endpoint-deconvolution} gives a nonnegative initial
vector \(\br_0=\brho_S^1\), supported on the \(S\)-attached subsets of \(B\).

The induction hypothesis applied to \(B\) says that
\(\Ac_{1-\theta}^B\) has nonnegative off-diagonal entries and zero column
sums.  Lemma~\ref{lem:cone-invariance} applied to \eqref{eq:rho-ode} gives
\(\brho_S^\lambda(C)\ge0\) for every \(C\subset B\).

It remains only to keep track of support.  Let
\[
        \mathcal C_S:=\{C\subset B:C\text{ is }S\text{-attached in }
        G_E^{\rm und}\} .
\]
The case \(C=\varnothing\) has no outgoing transitions, because the
\(\varnothing\)-column of \(\Ac_{1-\theta}^B\) is zero.  If
\(C\in\mathcal C_S\), \(C\ne\varnothing\), and
\(\Ac_{1-\theta}^B(U,C)>0\) with \(C\subsetneq U\subset B\), the induction
hypothesis inside \(B\) implies that every component of
\(G_B^{\rm und}[U]\) meets \(C\).  Since every component of \(C\) is attached
to \(S\) in \(G_E^{\rm und}\), every component of \(U\) is also attached to
\(S\).  Thus no positive transition leaves \(\mathcal C_S\), and
Lemma~\ref{lem:cone-invariance} preserves this support along
\eqref{eq:rho-ode}.

The incidence property and the column-sum identity were proved in
\eqref{eq:A-incidence} and \eqref{eq:column-sum-zero}.  The off-diagonal
entries and the support statement now follow from
\(\Ac_\lambda^E(S\cup C,S)=\brho_S^\lambda(C)\) in
Lemma~\ref{lem:column-formula}.
\end{proof}

\begin{proof}[{Proof of Theorem~\ref{thm:finite-positivity}, existence for \(0\le\lambda\le1\)}]
At \(\lambda=1\), the solution \(\bp_1\) is the nonnegative probability
measure of the range of LERW.  Let \(\bq_\theta\), \(0\le\theta\le1\), solve
\begin{equation}\label{eq:p-ode}
        \frac{\dd}{\dd\theta}\bq_\theta
        =\Ac_{1-\theta}^V\bq_\theta,
        \qquad
        \bq_0=\bp_1 .
\end{equation}
By Proposition~\ref{prop:positive-generator} and Lemma~\ref{lem:cone-invariance},
\(\bq_\theta\) remains a probability vector.

Moreover,
\[
\begin{aligned}
        \frac{\dd}{\dd\theta}\big(T_{1-\theta}^V\bq_\theta\big)
        &=-(\partial_\lambda T_\lambda^V)|_{\lambda=1-\theta}\bq_\theta
          +T_{1-\theta}^V\Ac_{1-\theta}^V\bq_\theta  \\
        &=0,
\end{aligned}
\]
because \(T_\lambda^V\Ac_\lambda^V=\partial_\lambda T_\lambda^V\).  Hence
\(T_{1-\theta}^V\bq_\theta=T_1^V\bp_1=\bh\).  The triangular matrix
\(T_{1-\theta}^V\) has positive diagonal, so the solution is unique:
\[
        \bq_\theta=(T_{1-\theta}^V)^{-1}\bh=\bp_{1-\theta} .
\]
This proves \(\bp_\lambda\ge0\) for \(0\le\lambda\le1\).  The total mass is
one because, in the row \(R=V\), all entries of \(T_\lambda^V\) are one and
\(\bh(V)=1\).
\end{proof}

The connectedness of the support of $\bp_\lambda$ can be also obtained by the linear system approach.
Call \(S\subset V\) connected if the induced graph
\(G_V^{\rm und}[S]\) is connected.  \begin{corollary}\label{cor:connected-support}
For every \(0\le\lambda\le1\),
\(\bp_\lambda\) is supported on connected sets.
\end{corollary}

\begin{proof}
At \(\lambda=1\), this follows from the construction of LERW. Under~\eqref{eq:p-ode},
Proposition~\ref{prop:positive-generator} allows only
positive jumps \(S\to U\) for which every component of
\(G_V^{\rm und}[U]\) meets \(S\); hence a connected set can jump
only to a connected set.  The support part
of Lemma~\ref{lem:cone-invariance} therefore preserves connectedness for all
\(0\le\lambda\le1\).
\end{proof}

The pure-birth operator $\Ac_\lambda^V$ further gives the following monotone coupling of random-walk loop-catchers for different $\lambda$, which is not straightforward from the construction in Algorithm~\ref{alg:time-catcher}. 
\begin{corollary}\label{cor:finite-monotone-coupling}
The laws \((\bp_\lambda)_{0\le\lambda\le1}\) admit a simultaneous
monotone coupling. Namely, there exist random subsets
\((S_\lambda)_{0\le\lambda\le1}\) defined on a common probability space
such that the marginal law of $S_\lambda$ is $\bp_\lambda$ for all $0\le\lambda\le 1$, and $S_{\lambda'}\subset S_\lambda$ simultaneously for all
\(0\le\lambda\le\lambda'\le1\).
\end{corollary}

\begin{proof}
Let \(X_\theta\) be the time-inhomogeneous finite-state chain
with generator \(\Ac_{1-\theta}^V\) and initial law \(\bp_1\).  Its forward
equation is \eqref{eq:p-ode}, so
\(\operatorname{Law}(X_\theta)=\bp_{1-\theta}\). Since every jump of \(X_\theta\) has
the form \(S\to U\) with \(S\subset U\), its sample paths are increasing in
\(\theta\). Setting \(S_\lambda=X_{1-\lambda}\) gives the result.
\end{proof}

\footnotesize{
\newcommand{\etalchar}[1]{$^{#1}$}
\def\cprime{$'$}

}

\bigskip

\noindent
\textsc{Beijing International Center for Mathematical Research, Peking University,\\
No.5 Yiheyuan Road, Haidian District, Beijing 100871, P.R.China}\\
\textit{Email address:} \texttt{caigefei1917@pku.edu.cn}

\end{document}